\documentclass[10pt, a4paper]{article}

\usepackage{fullpage}
\usepackage{listings}
\usepackage{latexsym, amsfonts, amssymb, amsmath, amsthm, mathrsfs, mathtools, setspace, graphics, graphicx, bbm, float, bigints}
\usepackage{enumerate}
\usepackage{caption}
\usepackage{indentfirst}
\usepackage{framed}
\usepackage{yhmath}
\usepackage[nodayofweek,level]{datetime}
\usepackage{overpic}
\usepackage{bm}

\usepackage{lipsum}        

\RequirePackage[ruled,vlined]{algorithm2e}
\RequirePackage{longtable}

\usepackage{geometry}
\DeclareMathOperator*{\argmin}{argmin}

\RequirePackage{tikz}
\usetikzlibrary{calc,trees,positioning,arrows,chains,shapes.geometric,%
    decorations.pathreplacing,decorations.pathmorphing,shapes,%
    matrix,shapes.symbols}
\allowdisplaybreaks

\usepackage{multirow}

\usepackage{lscape}

\usepackage{hyperref}
\newcommand{\footremember}[2]{
   \footnote{#2}
    \newcounter{#1}
    \setcounter{#1}{\value{footnote}}
}
 
\makeatletter
\newcommand{\leqnomode}{\tagsleft@true}
\newcommand{\reqnomode}{\tagsleft@false}
\makeatother

\numberwithin{equation}{section}
\newtheorem{thm}{Theorem}[section]
\newtheorem{lem}[thm]{Lemma}
\newtheorem{prop}[thm]{Proposition}

\newtheorem{manualtheoreminner}{Assumption}
\newenvironment{manualtheorem}[1]{%
  \renewcommand\themanualtheoreminner{#1}%
  \manualtheoreminner
}{\endmanualtheoreminner}

\theoremstyle{remark}
\newtheorem{rem}[thm]{Remark}

\newcommand{\widesim}[2][1.5]{
  \mathrel{\overset{#2}{\scalebox{#1}[1]{$\sim$}}}
}

\newcommand{\vertiii}[1]{{\left\vert\kern-0.25ex\left\vert\kern-0.25ex\left\vert #1 
    \right\vert\kern-0.25ex\right\vert\kern-0.25ex\right\vert}}
\newcommand{\vertii}[1]{\left\Vert #1\right\Vert}
\newcommand{\Proj}{\mathrm{Proj}}

\newcommand{\tabincell}[2]{\begin{tabular}{@{}#1@{}}#2\end{tabular}}

\newcommand\independent{\protect\mathpalette{\protect\independenT}{\perp}}
\def\independenT#1#2{\mathrel{\rlap{$#1#2$}\mkern2mu{#1#2}}}

\newcommand{\PPC}{\mathcal{P}_{p}\big(\mathcal{C}([0, T], \mathbb{R}^{d})\big)}
\newcommand{\CPP}{\mathcal{C}\big([0, T], \mathcal{P}_{p}(\mathbb{R}^{d})\big)}
\newcommand{\CRD}{\mathcal{C}([0, T], \mathbb{R}^{d})}

\newcommand{\PP}{\mathbb{P}}
\newcommand{\RD}{\mathbb{R}^{d}}
\newcommand{\RR}{\mathbb{R}}

\newcommand{\EE}{\mathbb{E}\,}

\newcommand{\LPC}{L_{\CRD}^{p}\big(\Omega, \mathcal{F}, (\mathcal{F}_{t})_{t\geq0}, \mathbb{P}\big)}

\begin{document}
\title{Particle method and quantization-based schemes for the simulation of the McKean-Vlasov equation}

\author{%
 Yating Liu\footremember{a}{\small CEREMADE, CNRS, UMR 7534, Université Paris-Dauphine, PSL University, 75016 Paris, France, \texttt{liu@ceremade.dauphine.fr}. This paper is Chapter 7 of the author's Ph.D. thesis \cite{liuthese}. }%
  }

\maketitle

%
%
\begin{abstract} In this paper, we study three numerical schemes for the McKean-Vlasov equation 
\[\hspace{-2cm}\begin{cases}
\;dX_t=b(t, X_t, \mu_t) \, dt+\sigma(t, X_t, \mu_t) \, dB_t,\: \\
\;\forall\, t\in[0,T],\;\mu_t \text{ is the probability distribution of }X_t, 
\end{cases}\] 
where $X_0 : (\Omega, \mathcal{F}, \mathbb{P})\rightarrow (\RD, \mathcal{B}(\RD))$ is a known random variable. Under the assumption on the Lipschitz continuity of the coefficients $b$ and $\sigma$, our first result proves the convergence rate of the particle method with respect to the Wasserstein distance, which extends previous work \cite{bossy1997stochastic} established in a one-dimensional setting. In the second part, we present and analyse two quantization-based schemes, including the recursive quantization scheme (deterministic scheme) in the Vlasov setting,  and the hybrid particle-quantization scheme (random scheme  inspired by the $K$-means clustering). Two simulations are presented at the end of this paper: Burgers equation introduced in \cite{bossy1997stochastic} and the network of FitzHugh-Nagumo neurons (see \cite{MR2974499} and \cite{bossy2015clarification}) in dimension 3. 
\end{abstract}
%
%

\bigskip


%

\section{Introduction}

The McKean-Vlasov equation was originally introduced in \cite{mckean1967propagation} as a stochastic model naturally associated to a class of non-linear PDEs. Nowadays, it refers to the whole family of stochastic differential equations whose coefficients not only depend on the position of the process $X_t$ at time $t$ but also depend on its probability distribution $\mu_t$. This  distribution dependent structure of the McKean-Vlasov equation is widely used to model phenomenons in statistical physics (see e.g. \cite{martzel2001mean}, \cite{bossy2021instantaneous}), in mathematical biology (see e.g. \cite{MR2974499}, \cite{bossy2015clarification}),  also in social sciences and in quantitative finance, often motivated by the development of the mean field games and the interacting diffusion models (see e.g. \cite{lasry2018mean}, \cite{cardaliaguet2018mean} and \cite{MR3752669}). Besides the original paper \cite{mckean1967propagation}, an excellent reference of the general theory of the McKean-Vlasov equation and the propagation of chaos is \cite{sznitman1991topics}. See also the lecture note \cite{lacker2018mean}.

In this paper, we consider a filtered probability space $(\Omega, \mathcal{F}, (\mathcal{F}_{t})_{t\geq0}, \mathbb{P})$ satisfying the usual conditions and  an $(\mathcal{F}_{t})-$standard Brownian motion $(B_{t})_{t\geq0}$ valued in $\mathbb{R}^{q}$. For a separable Banach space $(E, \vertii{\cdot}_{E})$, we denote the set of probability distributions on $E$ by $\mathcal{P}(E)$ and denote by $\mathcal{P}_{p}(E)$ the set of probability distributions on $E$ having $p$-th finite moment, $p\geq1$. The McKean-Vlasov equation  writes
\begin{align}\label{Aeq}
\begin{cases}
\;dX_t=b(t, X_t, \mu_t) \, dt+\sigma(t, X_t, \mu_t) \, dB_t,\: \\
\;\mu_t \text{ is the probability distribution of }X_t, \; t\in[0, T],\,
\end{cases}
\end{align}
\reqnomode
where $X_{0}: (\Omega, \mathcal{F},  \mathbb{P})\rightarrow \big(\mathbb{R}^{d}, \mathcal{B}(\mathbb{R}^{d})\big)$ is a known random variable independent of the Brownian motion $(B_{t})_{t\geq0}$ and $b, \sigma$ are Borel functions defined on $[0, T]\times \mathbb{R}^{d}\times \mathcal{P}(\mathbb{R}^{d})$ having respective values in $\mathbb{R}^{d}$ and in $\mathbb{M}_{d, q}(\mathbb{R})$, the space of real matrices with $d$ rows and $q$ columns.  If there exist Borel functions $\beta: [0, T]\times\mathbb{R}^{d}\times \mathbb{R}^{d}\rightarrow\mathbb{R}^{d}$ and $a: [0, T]\times\mathbb{R}^{d}\times \mathbb{R}^{d}\rightarrow \mathbb{M}_{d, q}(\mathbb{R})$ such that 
\begin{equation}\label{vlasovcase}
b(t, x, \mu)=\int_{\mathbb{R}^{d}}\beta(t, x, u)\mu(du) \;\;\text{and}\;\;\sigma(t, x, \mu)=\int_{\mathbb{R}^{d}}a(t, x, u)\mu(du),
\end{equation}
we call \eqref{Aeq} the Vlasov equation.

This paper aims to show three implementable numerical methods for the simulation  of the McKean-Vlasov equation \eqref{Aeq}, including the particle method and two quantization-based schemes, accompanied by a quantitative analysis of  the corresponding simulation error.

\subsection{Construction of numerical schemes}

\smallskip

The construction of numerical schemes involves two components:  the temporal discretization on $[0,T]$ and the spatial discretization on $\RD$.

\smallskip
\noindent {\sc (1) Temporal discretization by theoretical Euler scheme}

\smallskip
We fix  $M\in\mathbb{N}^{*}$ and  set $h=\frac{T}{M}$ as the time step. Let $t_{m}=t_{m}^{M}\coloneqq m\cdot h, \: 0\leq m\leq M$ and let $Z_{m+1}\coloneqq\frac{1}{\sqrt{h}}(B_{t_{m+1}}\!\!-B_{t_{m}}), \; 0\leq m\leq M-1$.   
The \textit{theoretical} Euler scheme of (\ref{Aeq}) is defined as follows, 
\begin{align}\label{Ceq}
\bar{X}_{t_{m+1}}^{M}=\bar{X}_{t_{m}}^{M}+h\cdot b(t_{m}^{M}, \bar{X}_{t_{m}}^{M},\bar{\mu}_{t_{m}}^{M})+\sqrt{h\,} \sigma(t_{m}^{M}, \bar{X}_{t_{m}}^{M}, \bar{\mu}_{t_{m}}^{M})Z_{m+1}, \quad \bar{X}^{M}_{0}=X_{0},
\end{align}
where for every $m\in\{0,...,M\}$, $\bar{\mu}_{t_{m}}^{M}$ denotes the probability distribution of $\bar{X}_{t_{m}}^{M}$.  We also define the continuous expansion of $(\bar{X}_{t_0}^M, ..., \bar{X}_{t_M}^M)$, denoted by $(\bar{X}_{t}^{M})_{t\in[0,T]}$,  by the following \textit{continuous} Euler scheme \begin{equation}\label{contEuler}
\forall\, m\!=0, ..., M\!-\!1,\;\forall \,t\!\in\! (t_m, t_{m+1}], \:\; \bar{X}^M_{t}\coloneqq \bar{X}_{t_{m}}^{M}+(t-t_{m}) b(t_{m}^{M}, \bar{X}_{t_{m}}^{M},\bar{\mu}_{t_{m}}^{M})+\sigma(t_{m}^{M}, \bar{X}_{t_{m}}^{M}, \bar{\mu}_{t_{m}}^{M})(B_{t}-B_{t_{m}})
\end{equation}
with the same $\bar{X}^{M}_{0}=X_{0}$. 
When there is no ambiguity, such as when the time discretization number $M$ is fixed, we will omit the superscript $M$ in $\bar{X}^{M}_{t_m}$, $\bar{X}^{M}_t$ and in $\bar{\mu}_{t_{m}}^{M}$ defined in \eqref{Ceq} and in \eqref{contEuler}.

The convergence rate of this theoretical Euler scheme is of order $h^{1/2}$ if the coefficients $b$ and $\sigma$ are Lipschitz continuous in $(t,x,\mu)$ (see \cite[Proposition 2.1]{liu2020functional} or further Proposition \ref{propeuler} for more details).  However, contrary to the Euler scheme for a standard diffusion $dX_t=b(t, X_t)dt+\sigma(t, X_t)dB_t$, the scheme defined by \eqref{Ceq} is not directly implementable, since it does not directly indicate how to simulate $\bar{\mu}^{M}_{t_{m}}$ in the coefficient functions $b$ and $\sigma$. To do this, we need a further spatial discretization on $\RD$, a key point of this paper, to construct a discrete approximation of $\bar{\mu}_{t_{m}}^{M}$. This article will primarily focus on analyzing the errors associated with the spatial discretization approach.

\smallskip
\noindent {\sc (2) Spatial discretization} 
\smallskip

In this paper, we present two distinct spatial discretization approaches: the particle method and the quantization-based method. 

\smallskip

\noindent {\sc (2.1) Particle method}   

\smallskip
The particle method is inspired by the \textit{propagation of chaos} property of the McKean-Vlasov equation (see e.g. \cite{sznitman1991topics}, \cite{gartner1988onthe}, \cite{lacker2018mean} and \cite{chassagneux2019weak}). We consider the same temporal discretization number $M$ and the same time step $h$ as in $\eqref{Ceq}$. For the simplicity of notation, we will omit the superscript $M$ in the following discussion. Let $\bar{X}_{0}^{1, N}, ..., \bar{X}_{0}^{N, N}$ be i.i.d copies of $X_{0}$ in $\eqref{Aeq}$ and let $B^{n}\coloneqq (B_{t}^{n})_{t\in[0, T]}, \;1\leq n\leq N,$ be $N$ independent standard Brownian motions valued in $\RR^q$, independent of the Brownian motion $(B_t)_{t\in[0,T]}$ in the initial McKean-Vlasov equation \eqref{Aeq} and of $(X_0, \bar{X}_{0}^{1,N}, ..., \bar{X}_{0}^{N,N})$. The main idea of the  \textit{particle method} is to construct an $N$-particle system $(\bar{X}^{1, N}_{t_{m}}, ..., \bar{X}^{N, N}_{t_{m}})_{0\leq m\leq M}$ by computing for every $1\leq n\leq N, 0\leq m\leq M-1,$
\begin{align}\label{Deq}
\begin{cases}
\:\bar{X}^{n, N}_{t_{m+1}}=\bar{X}^{n, N}_{t_{m}}+h b(t_{m}, \bar{X}^{n, N}_{t_{m}},\bar{\mu}^{N}_{t_{m}})+\sqrt{h\,}\sigma(t_{m}, \bar{X}^{n, N}_{t_{m}}, \bar{\mu}^{N}_{t_{m}}) Z_{m+1}^{n},\\
\:\bar{\mu}^{N}_{t_{m}}\coloneqq \frac{1}{N}\sum_{n=1}^{N}\delta_{\bar{X}^{n, N}_{t_{m}}},\: Z_{m+1}^{n}=\frac{1}{\sqrt{h}}(B^{n}_{t_{m+1}}-B^{n}_{t_{m}})\widesim{\,\text{i.i.d}\,} \mathcal{N}(0, \mathbf{I}_{q}), 
\end{cases} \end{align}
and then to use $\bar{\mu}^{N}_{t_{m}}$ defined in \eqref{Deq} as an estimator of $\bar{\mu}_{t_{m}}$ in \eqref{Ceq} at each time step $t_m$, $1\leq m\leq M$. Remark that $\bar{\mu}^{N}_{t_{m}}$ in \eqref{Deq} is a random probability distribution, that is, it depends on $\omega\in \Omega$. 
  As detailed later in Theorem \ref{thm1}, we prove that the convergence rate with respect to the $L^p-$norm of  the Wasserstein distance, between $\bar{\mu}^{N}_{t_{m}}$ in \eqref{Deq} and $\bar{\mu}_{t_{m}}$ in \eqref{Ceq} is of order $N^{-1/d}$ ($N^{-1/2p}$ if $d=1$) under appropriate conditions. The result is established by using the convergence rate in the Wasserstein distance of the empirical measure of an i.i.d. sample as presented in \cite[Theorem 1]{fournier2015rate}.

In one-dimensional setting, the convergence rate of the distribution function and of the density function of the particle method has been previously established in \cite{bossy1997stochastic}. In this paper, we establish the convergence rate with respect to the Wasserstein distance, which does not rely on distribution function representations and also holds for a higher dimensional setting ($d\geq2$). Assuming the same conditions as in \cite{bossy1997stochastic}, the convergence rate presented in Theorem \ref{thm1}, particularly when considering the $L^{1}$-Wasserstein distance, aligns with the convergence rate of the distribution function in \cite{bossy1997stochastic} (see \cite{gibbs2002choosing} for the expression of the $L^{1}$-Wasserstein distance in terms of the distribution function). Note that, the convergence analysis in this paper requires fewer conditions, as some of the conditions in \cite{bossy1997stochastic} are necessary only to guarantee the existence of a density function.

\smallskip
\noindent {\sc (2.2) Quantization-based method}

\smallskip
A second way to implement the spatial discretization is to use the (optimal) \textit{vector quantization}, also known as \textit{$K$-means clustering} in unsupervised learning. Consider a random variable $X$ having probability distribution $\mu=\mathcal{L}(X)\coloneqq \mathbb{P}\circ X^{-1}$ on $\RD$. Let $K\in\mathbb{N}^*$ be the quantization level. The main idea of the vector quantization is to project $\mu$ onto a fixed quantizer $x=(x_{1}, ..., x_{K})\in (\RD)^{K}$ satisfying $x_i\neq x_j$ if $i\neq j$, and to use the following discrete projection \begin{equation}\label{quandistribution}
\widehat{\mu}^{\,x}\coloneqq \!\sum_{k=1}^{K}w_k\cdot \delta_{x_{k}}\qquad \text{with}\quad w_k=\mu\big(V_{k}(x)\big), \;1\leq k\leq K\end{equation} as an approximation of $\mu$.  In \eqref{quandistribution}, $\big(V_{k}(x)\big)_{1\leq k\leq K}$ denotes a Vorono\"i partition generated by $x=(x_{1}, ..., x_{K})$ (see Figure \ref{fig:voronoi}), which is a Borel partition on $\RD$ satisfying 
\begin{equation}\label{defvoi}
V_{k}(x)\subset \Big\{ y\in \RD\;\big|\; |y-x_k|=\min_{1\leq j\leq K} |y-x_j|\Big\}, \quad  \,1\leq k \leq K.
\end{equation}
The approximation $\widehat{\mu}^{\,x}$ defined in \eqref{quandistribution} is deterministic. By introducing the  projection function $\Proj_{x}$
\begin{equation}\label{defprojfun}
\xi\in\RD\longmapsto\mathrm{Proj}_{x}(\xi)\!\coloneqq\! \sum_{k=1}^{K}x_{k}\mathbbm{1}_{V_{k}(x)}(\xi)\in\{x_1, ..., x_K\}
\end{equation} 
based on the chosen Vorono\"i partition $\big(V_{k}(x)\big)_{1\leq k\leq K}$(\footnote{The Vorono\"i partition $\big(V_{k}(x)\big)_{1\leq k\leq K}$ generated by a fixed quantizer $x=(x_{1}, ..., x_{K})$ is not unique since we can place  points on the hyperplane $H_{i,j}\coloneqq\{\xi\in\RD\;|\; |x_i-\xi|=|x_j-\xi|\}$ in either the Vorono\"i cell $V_i(x)$ or $V_j(x)$. However, the choice of the Vorono\"i partition has no impact on the quantization error of the quantizer $x$ (see further discussion in \eqref{def:quanti-error-function} and in \eqref{sumvoi}).}), one can remark that $\widehat{\mu}^{\,x}=\mu\circ \Proj_{x}^{-1}$, which is the law of the discrete random variable $\widehat{X}^{\,x}\coloneqq \Proj_{x}(X)$.

If $\mu\in\mathcal{P}_{p}(\RD),\; p\geq1$, there exists (at least) an \textit{optimal} quantizer $x^{*}=(x^{*}_{1}, ..., x^{*}_K)\in (\RD)^{K}$ in the sense that $\widehat{\mu}^{\,x^{*}}$ is the closest probability measure to $\mu$ with respect to the Wasserstein distance $\mathcal{W}_p$, among all probability distributions having a support of at most $K$ points (see \cite[Lemma 3.4]{graf2000foundations} and further Proposition \ref{propclassical}-(3)). In the quadratic setting ($p=2$), such optimal quantizer $x^{*}$ can be found by several numerical methods such as  Lloyd's fixed point algorithm, as described in Algorithm \ref{lloyd} (see \cite{Lloyd1982least}, \cite{kieffer1982exponential}, \cite{pages2016pointwise} for more details) or the CLVQ algorithm (see e.g. \cite[Section 3.2]{pages2015introduction}). In this paper, we implement  Lloyd's algorithm to obtain the quadratic optimal  quantizer in the simulations but numerical results can be obtained by other methods as well.  Figure \ref{fig:n01optimal} is an illustration of the quadratic optimal quantization at level $K=60$ for the standard normal distribution $\mu=\mathcal{N}(0, \mathbf{I}_2)$, where the blue points represent an optimal quantizer and the colour in red  represents the weight $\mu(V_k(x))$ of each Vorono\"i cell  (the darker the heavier). Furthermore, in the framework of unsupervised learning with unlabeled dataset $\{\xi_1, ..., \xi_n\}\subset \RD$, the $K$-means clustering is to compute the (quadratic) optimal quantizer of the empirical measure $\nu=\frac{1}{n}\sum_{i=1}^{n}\delta_{\xi_i}$ on the dataset. 

In this paper, we propose the following two quantization-based schemes for the simulation of the McKean-Vlasov equation (\footnote{This idea of applying optimal quantization to simulate the McKean-Vlasov equation was firstly introduced in \cite{gobet2005discretization}[Section 4] in a different framework. Besides, another quantization-based scheme is proposed in \cite[Section 7.4]{liuthese}, called \textit{doubly quantized scheme}, in which we implement the quantized Gaussian random variables instead of $(Z_m)_{1\leq m \leq M}$ in \eqref{Ceq}. }): 

\begin{center}
\begin{minipage}{0.8\textwidth}
\begin{enumerate}
\item[$(a)$] Recursive quantization scheme in the Vlasov setting \eqref{vlasovcase}, 
\vspace{0.1cm}
\item[$(b)$] Hybrid particle-quantization scheme. 
\end{enumerate} 
\end{minipage}
\end{center}
\vspace{0.1cm}
Different from the particle method, the recursive quantization scheme is deterministic, that is, the simulation result does not depend on $\omega\in\Omega$. 

\vspace{-0.5cm}
\begin{figure}[H]
\centering
\begin{minipage}[t]{0.45\textwidth}
\centering
\includegraphics[height=6cm,width=6cm]{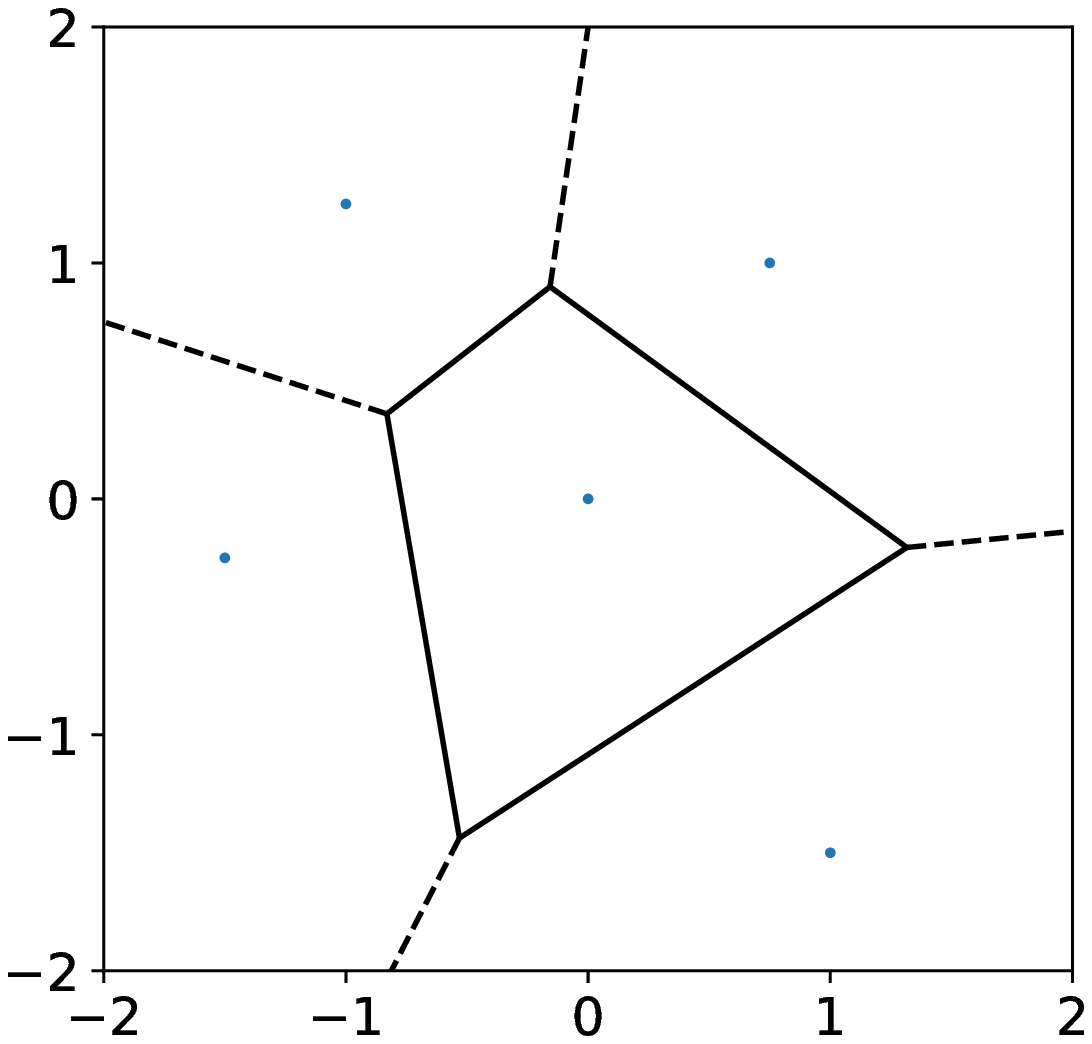}
\vspace{-0.5cm}
\caption{\small A Vorono\"i partition.\\\,}\label{fig:voronoi}
\end{minipage}
\begin{minipage}[t]{0.45\textwidth}
\centering
\includegraphics[height=6cm ,width=6cm]{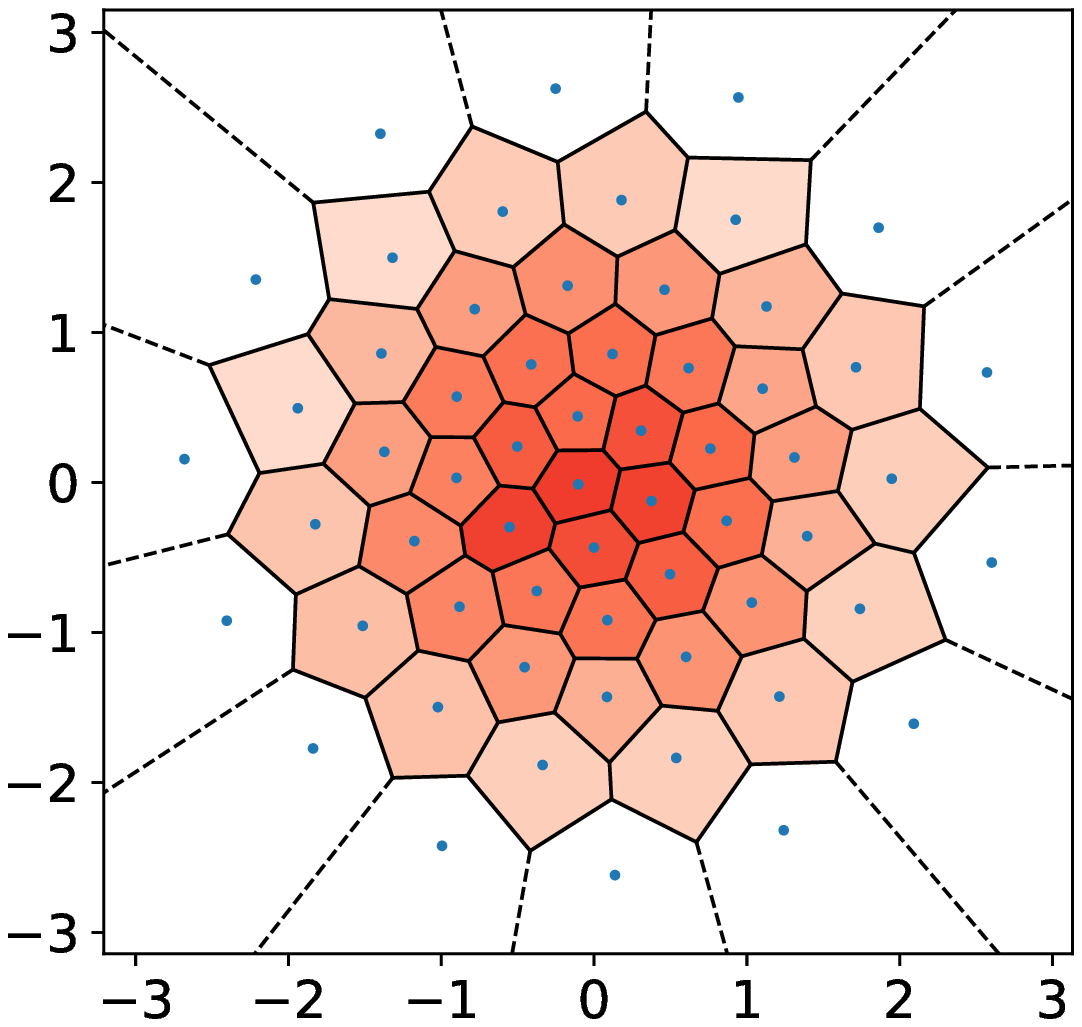}
\vspace{-0.5cm}
\caption{\small The quadratic optimal quantization for the standard normal distribution $\mathcal{N}(0, \mathbf{I}_2)$.}\label{fig:n01optimal}
\end{minipage}
\end{figure}

\vspace{-0.1cm}

\SetKwInOut{Input}{Input}
\SetKwRepeat{Repeat}{Repeat}{until}
\SetKw{KwRet}{Return}
\begin{algorithm}[H]
\caption{Lloyd's  (fixed point) algorithm for  a probability measure $\mu\in\mathcal{P}_{2}(\RD)$}\label{lloyd}

\smallskip

Set $K\in\mathbb{N}^{*}$ for the quantization level. 

\smallskip
\Input{ Initial quantizer $x^{[0]}= (x_{1}^{[0]}, ..., x_{K}^{[0]})$ such that $ x_{k}^{[0]}\in \text{supp}(\mu),\: 1\leq k\leq K$.
}

\Repeat{\quad {\em Some stopping criterion occurs}
(for example, $| \,x^{[l+1]}-x^{[l]} \,|<\varepsilon$ for a fixed threshold $\varepsilon>0$)}{\hspace{0.65cm} Compute  $x^{\,[l+1]}=\EE \big[\,X\;\big|\;\widehat{X}^{\;x^{[l]}}\,\big](\Omega) $, i.e. for every $k=1, ..., K$,

\begin{minipage}{0.94\textwidth}
\begin{equation}\label{lloyditeration}
 x^{[l+1]}_{k}=\begin{cases}
x^{[l]}_{k},\hspace{5.65cm} \text{ if } \mu \big(V_{k}(x^{[l]})\big)=0,\\ \\
\EE \big[\;X\;\big|\;\widehat{X}^{\;x^{[l]}}=x_{k}^{[l]}\;\big] =\frac{\int_{V_{k}(x^{[l]})}\xi\mu(d\xi)}{\mu \big(V_{k}(x^{[l]})\big)}, \hspace{0.6cm}\text{otherwise.}\end{cases} 
\end{equation}
\end{minipage}
}

\KwRet{ $x^{[l]}= (\,x_{1}^{[l]}, \,...,\, x_{K}^{[l]}\,)$ }

\smallskip
\end{algorithm}
\smallskip

\vspace{0.2cm}

\noindent{$(a)$ \sc Recursive quantization scheme.} The most natural idea of applying optimal quantization to the simulation of the McKean-Vlasov equation is to replace $\bar{X}_{t_{m}}$ and $\bar{\mu}_{t_m}$ in \eqref{Ceq} by their respective quantized projection. More precisely, we consider a quantizer sequence for each time step $t_m$ 
\begin{equation}\label{quantizerseq}
x^{(m)}=(x^{(m)}_{1}, ..., x^{(m)}_{K})\in (\mathbb{R}^{d})^{K}, \quad 0\leq m\leq M
\end{equation} 
and  define  the quantized random variable $\widehat{X}_{t_m}$ and its probability distribution $\widehat{\mu}_{t_m}$ by 
\begin{align}\label{Eeq}
\begin{cases}
\widetilde{X}_{0}=X_{0}, \;\;\widehat{X}_{0}=\text{Proj}_{x^{(0)}}(\widetilde{X}_{0}), \;\;\widehat{\mu}_{0}=\mathbb{P}\circ\,\widehat{X}_{0}^{-1},\\
\widetilde{X}_{t_{m+1}}=\widehat{X}_{t_{m}}+h\cdot b(t_m, \widehat{X}_{t_{m}},\widehat{\mu}_{t_{m}})+\sqrt{h\,}\sigma(t_m,\widehat{X}_{t_{m}}, \widehat{\mu}_{t_{m}}) Z_{m+1},\\
\widehat{X}_{t_{m+1}}=\mathrm{Proj}_{x^{(m+1)}}(\widetilde{X}_{t_{m+1}}), \quad \widehat{\mu}_{t_{m+1}}= \mathbb{P}\circ{\,\widehat{X}_{t_{m+1}}^{-1}},\;\;m=0, ..., M-1.
\end{cases}
\end{align}
In \eqref{Eeq}, at each time step $t_m$, we add a quantization procedure for $\widetilde{X}_{t_m}$, $\widetilde{\mu}_{t_{m}}=\mathcal{L}(\widetilde{X}_{t_m})$ and inject their quantization projection $\widehat{X}_{t_m}=\Proj_{x^{(m)}}(\widetilde{X}_{t_m})$ as well as $\widehat{\mu}_{t_m}=\mathcal{L}(\widehat{X}_{t_m})=\widetilde{\mu}_{t_{m}}\circ \Proj_{x^{(m)}}^{-1}$ into the time step $t_{m+1}$.  Later in Theorem \ref{thm:quadbasedscheme}, we establish an upper bound  on the error of this approximation \eqref{Eeq} in terms of the $L^2$-distance between the quantization projection  $\widehat{X}_{t_m}$ and the random variable $\bar{X}_{t_m}$ defined by the Euler scheme \eqref{Ceq}. This error depends on the quantization error at each time step $t_m$, which,  in turn, is influenced by the choice of quantizers $x^{(m)}=(x^{(m)}_{1}, ..., x^{(m)}_{K}), \; 0\leq m\leq M$. The upper bound presented in Theorem \ref{thm:quadbasedscheme} entails the accumulation of quantization errors over the time discretization, indicating potential sub-optimality and opening possibilities for future research and refinement.

This scheme \eqref{Eeq} is not directly implementable because of $\widehat{\mu}_{t_{m}}$ for the same reason as the theoretical Euler scheme \eqref{Ceq}, so we call \eqref{Eeq} \textit{theoretical quantization-based scheme}. However, in the Vlasov setting (\ref{vlasovcase}), we can circumvent this issue by using the recursive quantization method, which is first introduced in \cite{pages2015recursive} for the diffusion equation $dX_{t}=b(t, X_{t})dt + \sigma(t, X_{t})dB_{t}$. 

The main idea of the recursive quantization method is to construct  
the Markovian transitions of $(\widehat{X}_{t_{m}}, \widehat{\mu}_{t_{m}})$ based on \eqref{Eeq}. To be more specific, for each time step $t_m$, by considering Vorono\"i partitions $(V_{k}(x^{(m)}))_{1\leq k\leq K}$ generated by $x^{(m)}$, $0\leq m\leq M,$ and by applying \eqref{quandistribution}, we can rewrite $\widehat{\mu}_{t_m}$ in \eqref{Eeq}  as follow
\begin{equation}\label{muhatdef}
\widehat{\mu}_{t_{m}}=\sum_{k=1}^{K}p_{k}^{(m)}\delta_{x_{k}^{(m)}},
\end{equation} 
where $p_{k}^{(m)}=\widetilde{\mu}_{t_{m}}\big(V_{k}(x^{(m)})\big)=\PP(\widetilde{X}_{t_m}\in V_{k}(x^{(m)}))=\PP(\widehat{X}_{t_m}= x^{(m)}_k), \: 1\leq k\leq K$. 
The expression \eqref{muhatdef} shows that for every $m\in\{0, ..., M\}$,  the quantized distribution $\widehat{\mu}_{t_{m}}$ is determined by the quantizer $x^{(m)}=(x^{(m)}_{1}, ..., x^{(m)}_{K})$ and its weight vector $p^{(m)}=(p^{(m)}_{1}, ..., p^{(m)}_{K})$. In the Vlasov setting \eqref{vlasovcase}, the transition step of \eqref{Eeq}, with respect to $\widehat{\mu}_{t_{m}}$ in \eqref{muhatdef}, writes 
\[\widetilde{X}_{t_{m+1}}=\widehat{X}_{t_{m}}+h\sum_{k=1}^{K}p_{k}^{(m)}\,\beta(t_m, \widehat{X}_{t_{m}},x_k^{(m)})+\sqrt{h\,}\sum_{k=1}^{K}p_{k}^{(m)}\,a(t_m,\widehat{X}_{t_{m}}, x_k^{(m)}) Z_{m+1}.\]
Consequently, under the condition that the weight vector $p^{(m)}$ is known, and given the value of $\widehat{X}_{t_{m}}$, the transition probability is 
\begin{align}\label{Geq}
 \mathbb{P}& \Big(\widehat{X}_{t_{m+1}}=x_{j}^{(m+1)}\mid \widehat{X}_{t_{m}}=x_{i}^{(m)}\Big)\\
&=\mathbb{P}\Big[\Big(x_{i}^{(m)}+ h\sum_{k=1}^{K}p_{k}^{(m)}\beta(t_m, x_{i}^{(m)}, x_{k}^{(m)})+\sqrt{ h}\sum_{k=1}^{K}p_{k}^{(m)}a(t_m, x_{i}^{(m)}, x_{k}^{(m)})Z_{m+1}\Big)\in V_{j}(x^{(m+1)})\Big],\nonumber
\end{align}
 so that we can compute $p^{(m+1)}_j =\mathbb{P}(\widehat{X}_{t_{m+1}}=x^{(m+1)}_j),\,1\leq j\leq K,$ by 
\begin{align}\label{Geq2}
p_{j}^{(m+1)}& =\sum_{i=1}^{K}\mathbb{P}\big(\widehat{X}_{t_{m+1}}=x_{j}^{(m+1)}\mid \widehat{X}_{t_{m}}=x_{i}^{(m)}\;\big)\cdot p_{i}^{(m)}.
\end{align}
A detailed proof of the above equalities is provided in Section \ref{recurq}, where we also explain how to combine this recursive quantization method with  Lloyd's algorithm to optimally compute the quantizer $x^{(m)}$ at each time step. 

\smallskip
\noindent{$(b)$ \sc Hybrid particle-quantization scheme.} Another way to simulate the McKean-Vlasov equation is to apply the optimal quantization method to the particle system \eqref{Deq} and subsequently devise the following \textit{hybrid particle-quantization scheme}. Consider the same initial random variables $\bar{X}^{1, N}_0, ..., \bar{X}^{N, N}_0$ and the same Brownian motions $B^{n}, \,1\leq n\leq N$ as \eqref{Deq} and  define
\begin{align}\label{Feq}
\begin{cases}
\widetilde{X}^{n, N}_{t_{m+1}}=\widetilde{X}^{n, N}_{t_{m}}+ h \cdot b(t_m, \widetilde{X}^{n, N}_{t_{m}},\widehat{\mu}^{K}_{t_{m}})+\sqrt{h\,} \sigma(t_m, \widetilde{X}^{n, N}_{t_{m}}, \widehat{\mu}^{K}_{t_{m}})Z_{m+1}^{n},\: 1\leq n \leq N,\\
\widehat{\mu}^{K}_{t_{m}}=\big(\frac{1}{N}\sum_{n=1}^{N}\delta_{\widetilde{X}^{n, N}_{t_{m}}}\big)\circ \Proj_{x^{(m)}}^{-1}=\sum_{k=1}^{K} \big[\delta_{x_{k}^{(m)}}\cdot\sum_{n=1}^{N}\mathbbm{1}_{V_{k}(x^{(m)})}(\widetilde{X}_{t_{m}}^{n, N})\big], \\ 
\widetilde{X}^{n, N}_{t_{0}}=\bar{X}^{n, N}_{0}\widesim{\text{i.i.d}}X_{0},\;\;Z_{m}^{n}=\frac{1}{\sqrt{h}}(B^{n}_{t_{m+1}}-B^{n}_{t_m})\widesim{\text{i.i.d}}\mathcal{N}(0, \mathbf{I}_{q}).
\end{cases}
\end{align}
The concept behind constructing this scheme is as follows. Considering that the measure argument $\widehat{\mu}_{t_{m}}$ in \eqref{Eeq} cannot be directly simulated except for the Vlasov setting \eqref{vlasovcase}, an alternative strategy is injecting the quantization projection $\widehat{\mu}_{t_m}^{K}$ of 
the empirical measure $\frac{1}{N}\sum_{i=1}^{N}\delta_{\widetilde{X}_{t_m}^{n, N}}$ at time $t_m$ 
 into the simulation for $\widetilde{X}^{n,N}_{t_{m+1}}$ at time $t_{m+1}$, instead of using $\widehat{\mu}_{t_{m}}$. 

The theoretical foundation for this idea  comes from the consistency of the optimal quantizer as established in \cite{liu2018convergence}. Specifically, if a sequence of probability distributions $(\mu_n)_{n\geq1}$ converges to $\mu_\infty$ with respect to the Wasserstein distance,  any limiting point of the optimal quantizers $x^{(n)}$ of $\mu_n$ is an optimal quantizer of $\mu_\infty$.  Considering the convergence of the particle method obtained in Theorem \ref{thm1}, we can anticipate that an optimal quantizer of the empirical measure $\frac{1}{N}\sum_{n=1}^{N}\delta_{\widetilde{X}^{n, N}_{t_{m}}}$  is close to being optimal for the measure $\bar{\mu}_{t_m}$ of \eqref{Ceq}. Following this idea, Proposition \ref{quanNparti}, detailed in Section \ref{FtoD},   provides an error analyse of the scheme \eqref{Feq} in terms of an upper bound of the Wasserstein distance between the quantization projection  $\widehat{\mu}^K_{t_m}$ and the probability distribution $\bar{\mu}^{N}_{t_m}$ defined by the particle method  \eqref{Deq}. Moreover,  from a practical perspective, obtaining an optimal quantizer for the empirical measure $\frac{1}{N}\sum_{n=1}^{N}\delta_{\widetilde{X}^{n, N}_{t_{m}}}$ is more straightforward in  simulations by using e.g. \texttt{sklearn.cluster.KMeans} package in \texttt{Python}, as demonstrated in the pseudo-code in Algorithm \ref{hybridcode2}.

\subsection{Comparison of  spatial discretization approaches and future outlook} 

We now provide  concise comments on the numerical performance and characteristics of these three spatial discretization approaches. Additional details can be found in Section \ref{example}.

First, regarding computing time, both the particle method \eqref{Deq} and the recursive quantization scheme \eqref{Eeq}-\eqref{Geq}-\eqref{Geq2} (excluding Lloyd iteration) essentially involve the computation of a Markov chain in $\RR^{N}$ and $\RR^{K}$ respectively.  Consequently, these are theoretically the two fastest methods. However, the recursive quantization scheme requires rapid computation for terms such as $p_{k}^{(m)}=\widetilde{\mu}_{t_{m}}\big(V_{k}(x^{(m)})\big)$, which is a numerical integration over a Vorono\"i cell. This process becomes numerically costly in high-dimensional settings, particularly when utilizing packages like \texttt{qhull} (see  \texttt{www.qhull.com}). Furthermore, the accuracy of the quantization method depends on the choice of the quantizer, while integrating the Lloyd's algorithm into the quantization scheme to obtain an optimal quantizer adds to the computing time. In some cases, particularly in one-dimensional setting, adding the Lloyd's algorithm proves to be more numerically costly without significant improvement on the accuracy, compared to a simpler strategy of increasing the quantizer size $K$ and using a uniformly spaced quantizer.

In terms of accuracy, both the particle method \eqref{Deq} and the hybrid particle-quantization scheme \eqref{Feq} are classified as \textit{random} algorithms,  indicating that  their simulation results, depend on $\omega$ in the probability space $(\Omega, \mathcal{F}, \mathbb{P})$. Therefore, when evaluating the accuracy of these two methods, one needs to compute the standard deviation of errors simultaneously. On the other hand, the recursive quantization scheme \eqref{Eeq}-\eqref{Geq}-\eqref{Geq2} is deterministic, but this scheme is currently adapted only to the Vlasov case \eqref{vlasovcase}. 

Moreover,  in many research areas and applications associated to the McKean-Vlasov equation, such as mean field games (see e.g. \cite{MR3752669}) and opinion dynamics (see e.g. \cite{hegselmann2002opinion}), the primary focus is on an $N$-particle system $(X_t^{1, N}, ..., X_{t}^{N,N})_{t\in[0,T]}$ satisfying $(X_0^{1, N}, ..., X_{0}^{N,N})\widesim{\,\text{i.i.d.}\,}X_0$ and following the dynamics given by
\begin{align}\label{Beq}
\begin{cases}
dX_{t}^{n,N}=b(t, X_{t}^{n,N}, \mu_{t}^{N})dt+\sigma(t, X_{t}^{n,N}, \mu_{t}^{N})dB_t^n, \quad 1\leq n\leq N,\\
\forall \, t\in[0,T], \quad \mu_t^{N}=\frac{1}{N}\sum_{n=1}^{N}\delta_{X_t^{n, N}}.
\end{cases}
\end{align}
The McKean-Vlasov equation \eqref{Aeq} is considered as a limit equation of such particle system \eqref{Beq} when $N\rightarrow +\infty$, often appeared in the \textit{propagation of chaos} theory in the literature. The particle method \eqref{Deq} presented in this paper can be viewed as a temporal discretization of \eqref{Beq}. Meanwhile, the hybrid particle-quantization scheme \eqref{Feq} essentially incorporates a $K$-means clustering step into the particle method. From the perspective of unsupervised learning, considering the particles in \eqref{Deq} as a dataset, this additional step is inherently heuristic for modeling, given that $K$-means clustering is commonly employed for {identifying salient features}, {natural classification} and {data compression} in unsupervised learning (see e.g. \cite{jain2010data}). In Section \ref{simu2section}, we show through numerical experiments that the hybrid particle-quantization scheme reduces  data volume while maintaining  stability in both the mean and variance of the test function value.

In considering future directions, advancements in optimal quantization theory and numerical methods will  enhance the outcomes of this paper. First, the development of faster computational algorithms for numerical integrations over Vorono\"i cells in high-dimensional settings will expand the practical applicability of the recursive quantization scheme. Additionally, the current upper bounds on simulation errors for the quantization-based schemes in Theorem \ref{thm:quadbasedscheme} and Proposition \ref{quanNparti} are not optimal. These bounds represent the accumulation of quantization errors over the time discretization, lacking constraints when the time discretization number $M$ tends to infinity. However, such divergence was not observed in the simulation examples in Section \ref{example}, even with a large value of $M$. Therefore, investigating alternative error bounds and determining conditions for convergence will be another direction for the future research. 

Furthermore, since the convergence results in Theorem \ref{thm1} and Proposition \ref{quanNparti} are established with respect to the Wasserstein distance, developing precise and efficient computation methods for the Wasserstein distance is crucial when practically evaluating the convergence, especially in high-dimensional settings. Finally, the simulation example in Section \ref{simu2section} involves the density estimation of the solution to the McKean-Vlasov equation. For a more in-depth understanding of its dependency on the choice of kernel and bandwidth in the density estimation procedure, we recommend consulting \cite{MR1910635} and \cite{Hoffmann_2022}.

\subsection{Organization of this paper}

This paper is organised as follows.

Section \ref{notaandassu} gathers the notations used in this paper and the assumptions for the error analyses of the three numerical schemes.  

In Section \ref{DtoC}, we discuss the particle method and prove its convergence rate in Theorem \ref{thm1}. The  proof shares the same idea of the well-known \textit{propagation of chaos} (see e.g. \cite[Theorem 3.3]{lacker2018mean}). We  compare the particle system \eqref{Deq} with another particle system \textit{without} interaction, that is, a system composed by $N$ i.i.d. It\^o processes $(Y^{1}_{t}, ..., Y^{N}_{t})_{t\in[0,T]}$ simulated by the continuous Euler scheme \eqref{contEuler}. Then the error of the particle method can be bounded by, up to a constant multiplier, the Wasserstein distance between the probability measure of $\bar{X}$ defined by \eqref{contEuler} and its empirical measure defined on its i.i.d. copies $(Y^{1}, ..., Y^{N})$.

Section \ref{qscheme} is devoted to the quantization-based schemes \eqref{Eeq}, \eqref{Geq}   and \eqref{Feq}, along with their respective error analyses. In Subsection \ref{reviewoq},  we first give a review of the optimal quantization method. Then we present and prove Theorem \ref{thm:quadbasedscheme}, the $L^{2}$-error analysis of the theoretical  quantization-based scheme \eqref{Eeq}. In Subsection \ref{recurq}, we show how to obtain the transition probability \eqref{Geq} for the recursive quantization scheme in the Vlasov setting, as well as the integration of Lloyd's algorithm.  Finally, Subsection \ref{FtoD} contains Proposition \ref{quanNparti}, the error analysis of the hybrid particle-quantization scheme \eqref{Feq}, and its proof. 

Section \ref{example} illustrates two simulation examples by using  the above numerical methods. The first simulation is for a one-dimensional Burgers equation, introduced in \cite{sznitman1991topics} and \cite{bossy1997stochastic}. This  equation has a closed-form solution, so  we can compute and compare the accuracy of different methods. The second example is a 3-dimensional model, the network of FitzHugh-Nagumo neurons, firstly introduced and simulated in \cite{MR2974499} (see also some corrections in \cite{bossy2015clarification}) and also simulated in \cite{reis2018simulation}. 

Finally, for the reader's convenience,  Figure \ref{stru} in Appendix \ref{appa} displays all the numerical methods presented in this paper along with their connections, where we also briefly mention their convergence rates. Additionally, Appendix \ref{appa} includes the pseudo-codes for the two quantization-based schemes introduced in this paper.

\section{Notations and assumptions}\label{notaandassu}

\subsection{Notations and definitions}

In the whole paper, we place ourselves in a filtered probability space $(\Omega, \mathcal{F}, (\mathcal{F}_{t})_{t\geq0}, \mathbb{P})$ satisfying the usual conditions and $(B_{t})_{t\geq0}$ denotes an $(\mathcal{F}_{t})-$standard Brownian motion valued in $\mathbb{R}^{q}$. The same holds for $(B_{t}^n)_{t\geq0}$ with $n\in\mathbb{N}^{*}$. We use $\left|\cdot\right|_{d}$ for the Euclidean norm on $\mathbb{R}^{d}$ (we drop the subscript $d$ in $\left|\cdot\right|_{d}$ when there is no ambiguity) and  $\langle \cdot , \cdot \rangle$  for the Euclidean inner product. We write $\mathbb{M}_{d, q}(\mathbb{R})$ for the space of real matrices with $d$ rows and $q$ columns and endow the space $\mathbb{M}_{d, q}(\mathbb{R})$  with an operator norm $\vertiii{\cdot}$ defined by $\vertiii{A}\coloneqq\sup_{\left|z\right|_{q}\leq1}\left|Az\right|_{d\,}$. 

We write $(E, \Vert \cdot \Vert_E)$ for a Banach space $E$ endowed with the norm $ \Vert \cdot \Vert_E$.  For a random variable $X$ valued in $(E, \Vert \cdot \Vert_E)$, we write $P_X$ or $\mathcal{L}(X)$ for its probability distribution $\mathbb{P}\circ X^{-1}$ and write $\Vert X\Vert_p$ for its $L^p$-norm, defined by $\Vert X \Vert_p\coloneqq \big[\EE \Vert X \Vert_E^p\big]^{1/p}$, $p\geq1$.  
 Moreover, we consider the space of $\RD$-valued continuous mappings endowed with the supremum norm, denoted by $\big(\mathcal{C}([0,T], \RD), \Vert \cdot \Vert_{\sup}\big)$, that is, 
\[\mathcal{C}([0,T], \RD)\coloneqq\big\{\alpha\!=\!(\alpha_t)_{t\in[0,T]} \text{ such that $t\in[0,T]\mapsto\alpha_t\in\RD$ is continuous}\big\}\:\: \text{and}\: \:\Vert\alpha\Vert_{\sup}\!\!=\!\!\sup_{t\in[0,T]}|\alpha_t|.\]
The space  $L_{\CRD}^{p}\big(\Omega, \mathcal{F}, (\mathcal{F}_{t})_{t\geq0}, \mathbb{P}\big)$ denotes the $L^p$-space of $(\mathcal{F}_{t})$-adapted random processes, that is,  
\begin{align}\label{LPC}
&L_{\CRD}^{p}\big(\Omega, \mathcal{F}, (\mathcal{F}_{t})_{t\geq0}, \mathbb{P}\big)\coloneqq\Big\{\:Y=(Y_{t})_{t\in[0,T]}: (\Omega, \mathcal{F}, (\mathcal{F}_{t})_{t\geq0}, \mathbb{P})\rightarrow \mathcal{C}\big([0, T], \mathbb{R}^{d}\big)\text{ such that }\nonumber\\
&\hspace{6cm} Y \text{ is } (\mathcal{F}_{t})\text{-adapted and } \left\Vert Y\right\Vert_{p}\coloneqq\big[\mathbb{E}\left\Vert Y\right\Vert_{\sup}^{p}\big]^{1/p}<+\infty\:\Big\}.
\end{align}
The notation $C_{p_1, ..., p_n}$ means a positive constant depending on parameters $p_1, ..., p_n$ whose value is allowed to change from line to line,  $\mathcal{N}(0, \mathbf{I}_q)$ denotes the $\RR^q$ standard normal distribution, where $\mathbf{I}_q$ is the $q\times q$ identity matrix, and $\delta_a$ denotes the Dirac measure at a point $a$. For two random variables $X$ and $Y$, the notation $X\independent Y$ is used to indicate that $X$ is independent of $Y$.

\smallskip

Now we introduce two classes of notations frequently used in this paper.

\smallskip

\noindent {\sc  (1) Wasserstein type notations}

\smallskip

For a Polish space $(S, d_S)$,  we denote the set of probability distributions on $S$ by $\mathcal{P}(S)$ and denote by $\mathcal{P}_{p}(S)$ the set of probability distributions on $S$ having $p$-th finite moment, $p\geq1$.  The definition of the $L^p$-Wasserstein distance $W_p$ on $\mathcal{P}_p(S)$ is 
\begin{align}\label{defwaspo}
\forall\,\mu, \nu\in\mathcal{P}_{p}&(S),\nonumber\\
W_{p}(\mu,\nu)&\coloneqq\Big{(}\inf_{\pi\in\Pi(\mu,\nu)}\int_{S\times S}d_S(x,y)^{p}\pi(dx,dy)\Big{)}^{\frac{1}{p}}\nonumber\\
&=\inf\Big{\{}\big{(}\mathbb{E}\,[d_S(X,Y)^{p}]\big{)}^{\frac{1}{p}}\!\!,\,  X,Y:(\Omega,\mathcal{F},\mathbb{P})\rightarrow( S, \mathcal{S})  \:\text{ with } \:\mathbb{P}\circ X^{-1}=\mu, \mathbb{P}\circ Y^{-1}=\nu\,\Big{\}},
\end{align}
where in (\ref{defwaspo}), $\Pi(\mu,\nu)$ denotes the set of all probability measures on $\big( S\times  S, \mathcal{S}^{\otimes2}\big)$ with marginals $\mu$ and $\nu$, and $\mathcal{S}$ denotes the Borel $\sigma$-algebra on $S$ generated by the distance $d_S$.  We write $\mathcal{W}_p$ for the case $S = \RD$  endowed with the Euclidean norm $\left|\cdot\right|_{d}$ and write $\mathbb{W}_p$
for the case $S=\mathcal{C}([0,T], \RD)$ endowed with the sup norm $\Vert\cdot\Vert_{\sup}$. Inspired by \cite{lacker2018mean}, we also consider the ``truncated Wasserstein distance'' on $\mathcal{P}_{p}\big(\mathcal{C}([0, T], \mathbb{R}^{d})\big)$, defined by
\begin{equation}\label{truncatedwass}
\!\!\forall\, \mu, \nu\in \PPC,\:\mathbb{W}_{p,t}(\mu, \nu)\!\coloneqq\!\Big{[}\inf_{\pi\in\bar{\Pi}(\mu,\nu)}\int_{\mathcal{C}([0, T], \mathbb{R}^{d})\times \mathcal{C}([0, T], \mathbb{R}^{d})}\sup_{s\in[0,t]}\left| x_{s}-y_{s}\right|^{p}\pi(dx,dy)\Big{]}^{\frac{1}{p}}\!\!\!,
\end{equation}
where $\bar{\Pi}(\mu, \nu)$  denotes the set of all probability measures on $\mathcal{C}\big([0, T], \RD\big)\times \mathcal{C}\big([0, T], \RD\big)$ equipped with the Borel $\sigma$-algebra generated by $\Vert\cdot \Vert_{\sup}$.

We also introduce $\mathcal{C}\big( [0, T], \mathcal{P}_{p}(\mathbb{R}^{d})\big)$ as the marginal distributions space, defined by
\begin{align}
&\mathcal{C}\big( [0, T], \mathcal{P}_{p}(\mathbb{R}^{d})\big)\coloneqq\Big\{(\mu_{t})_{t\in[0, T]}\text{ s.t. }t\mapsto\mu_{t} \text{ is a continuous mapping from } [0,T] \text{ to } \big(\mathcal{P}_{p}(\mathbb{R}^{d}), \mathcal{W}_{p}\big) \Big\}\nonumber
\end{align}
equipped with the distance 
\begin{equation}\label{distancedc}
d_{\mathcal{C}}\big((\mu_{t})_{t\in[0, T]}, (\nu_{t})_{t\in[0, T]}\big)\coloneqq\sup_{t\in[0, T]}\mathcal{W}_{p}(\mu_{t}, \nu_{t}). 
\end{equation}

The mapping $\iota$ sends elements from $\PPC$ to $\CPP$ and is defined by 
\begin{equation}\label{defiota}
\mu\mapsto \iota(\mu)=(\mu\circ\pi_{t}^{-1})_{t\in[0, T]}=(\mu_{t})_{t\in[0, T]},
\end{equation}
where for any $t\in[0, T]$,  $\pi_{t}: \CRD\rightarrow\mathbb{R}^{d}$ is defined by $\alpha=(\alpha_t)_{t\in[0,T]}\mapsto\pi_{t}(\alpha)=\alpha_{t}$.

\smallskip
\noindent {\sc  (2) Quantization type notations}

\smallskip

For a given quantization level $K \in \mathbb{N}^*$, we represent a $K$-level quantizer as $x=(x_{1}, ..., x_{K})\in (\RD)^{K}$.  It is generally assumed that the components $x_i$ are distinct, that is, $x_i\neq x_j$ if $i\neq j$. 
Given the quantizer $x = (x_{1}, ..., x_{K})$, the notation  $\big(V_{k}(x)\big)_{1\leq k\leq K}$ represents a Vorono\"i partition generated by $x$, which is a Borel partition on $\RD$ satisfying \vspace{-0.2cm}
\begin{equation}
V_{k}(x)\subset \Big\{ y\in \RD\;\big|\; |y-x_k|=\min_{1\leq j\leq K} |y-x_j|\Big\}, \quad  \,1\leq k \leq K.\vspace{-0.2cm}
\nonumber
\end{equation}
and the notation  $\Proj_{x}$ is  the projection function based on a Vorono\"i partition $\big(V_{k}(x)\big)_{1\leq k\leq K}$, defined by 
\vspace{-0.2cm}
\begin{equation}
\xi\in\RD\:\longmapsto\:\mathrm{Proj}_{x}(\xi)\!\coloneqq\! \sum_{k=1}^{K}x_{k}\mathbbm{1}_{V_{k}(x)}(\xi)\in\{x_1, ..., x_K\}.\vspace{-0.2cm}
\nonumber
\end{equation} 

Let $X$ be a random variable taking values in $\mathbb{R}^d$ with probability distribution $\mu$ and let $x=(x_{1}, ..., x_{K})\in (\RD)^{K}$ be a $K$-level quantizer. We denote the quantization approximation of $X$ by $\widehat{X}^x\coloneqq \Proj_{x}(X)$ and denote the quantization approximation of $\mu$ by $\widehat{\mu}^{\,x}\coloneqq \!\sum_{k=1}^{K}\mu\big(V_{k}(x)\big)\cdot \delta_{x_{k}}\!=\!\mu\circ \Proj_{x}^{-1}$, which is the law of $\widehat{X}^x$. 
When there is no ambiguity on the quantizer, we drop the superscript $x$ in $\widehat{X}^{\,x}$ and in  $\widehat{\mu}^{\,x}$. 

When $\mu\in\mathcal{P}_p(\RD)$,  the accuracy of the quantization approximation can be assessed through the quantization error function $e_{K, p}(\mu, \cdot)$, defined by 
\begin{equation}\label{def:quanti-error-function}
x=(x_1, ..., x_K)\in (\RD)^{K}\mapsto e_{K, p}(\mu, x)\coloneqq \left[ \int_{\RD}\min_{1\leq k\leq K}|\xi-x_k|^{p}\mu(d\xi)\right]^{\frac{1}{p}}\in\RR_{+}. 
\end{equation}
Sometimes we also use the notation $e_{K, p}(X, \cdot)$ instead of $e_{K, p}(\mu, \cdot)$ for a random variable $X$ such that $P_X=\mu$. The optimal quantizer at level $K$ and order $p$ for $\mu$ is denoted by $x^{*}=(x^{*}_1, ..., x^{*}_K)$, which is defined by 
\begin{equation}\label{defopquantizer}
x^{*}=(x^{*}_1, ..., x^{*}_K)\in\argmin_{y\in(\RD)^K} e_{K, p}(\mu, y),
\end{equation}
and we call
\begin{equation}\label{def:optimalerror}
e_{K,p}^{*}(X)=e_{K,p}^{*}(\mu)\coloneqq\inf_{x\in (\RD)^K}e_{K,p}(\mu, x)
\end{equation}
 the optimal quantization error at level $K$ and order $p$ for $\mu$. Finally, the notation \[\mathcal{M}_{p}(\mu)\coloneqq \inf_{a \in \RR^d} \left[\int_{\RD}|\xi-a|^{p} \mu(d\xi)\right]^{\frac{1}{p}}\] will be repeatedly used in an upper bound for the optimal quantization error.

\subsection{Assumptions}

Throughout this paper, we make the following assumption on the coefficients $b$, $\sigma$ and on the initial random variable $X_0$. Remark that the following assumption depends on an index $p\in[2,+\infty)$. 
\begin{manualtheorem}{I}\label{AssumptionI}
There exists $p\in[2, +\infty)$ such that the following conditions (1) and (2) hold. 
\begin{enumerate}[(1)]
\item The initial random variable $X_0$ satisfies $\vertii{X_{0}}_{p}<+\infty$. 
\smallskip
\item The coefficient functions $b$ and $\sigma$ are $\gamma$-H\"older continuous in $t$, $\gamma \in (0, 1]$, and Lipschitz continuous in $x$ and in $\mu$ in the following sense: there exists a constant $L>0$ such that
\begin{align}
&\forall\, (x,\mu) \in\RR^{d}\times\mathcal{P}_{p}(\RD), \forall\,s, t\in[0,T] \:\text{with}\: s\leq t,\nonumber\\
&\qquad \left|b(t, x, \mu) - b(s, x, \mu)\right| \vee\vertiii{\sigma(t, x, \mu) - \sigma(s, x, \mu)}\leq L\big(1+|x|+\mathcal{W}_{p}(\mu, \delta_0)\big)(t-s)^{\gamma},
\end{align}
and such that 
\begin{align}
&\forall \,t\in[0, T], \forall \,x,y \in\mathbb{R}^{d} \text{ and } \forall\, \mu, \nu\in\mathcal{P}_{p}(\mathbb{R}^{d}),\nonumber\\
&\qquad \left|b(t, x, \mu) - b(t, y, \nu)\right| \vee\vertiii{\sigma(t, x, \mu) - \sigma(t, y, \nu)}\leq L\big(\,|x-y|+\mathcal{W}_{p}(\mu, \nu)\,\big).
\end{align}
\end{enumerate}
\end{manualtheorem}
\smallskip
Assumption \ref{AssumptionI} is a classical assumption for existence and uniqueness of a strong solution $X=(X_{t})_{t\in[0, T]}$ to the McKean-Vlasov equation \eqref{Aeq} (see e.g. \cite{lacker2018mean}, \cite[Chapter 5]{liuthese}) and the convergence of the Euler scheme \eqref{Ceq},~\eqref{contEuler}.  
\begin{prop}(\!\!\cite[Proposition 2.1]{liu2020functional})\label{propeuler} Assume that Assumption \ref{AssumptionI} holds for an index $p\in[2,+\infty)$.
Let $X=(X_t)_{t\in[0,T]}$ be the unique strong solution of \eqref{Aeq} and let $\bar{X}^{M}=(\bar{X}^{M}_{t})_{t\in[0,T]}$ denote the process defined by the continuous Euler scheme \eqref{contEuler}.  Then there exists a constant $C_{p, d, T, L, \gamma, \vertii{X_0}_p}$ such that 
\[\Big\Vert\sup_{t\in[0,T]}\big|X_t-\bar{X}^{M}_{t}\big|\Big\Vert_p\leq C_{p, d, T, L, \gamma, \vertii{X_0}_p} h^{\frac{1}{2}\wedge \gamma}.\]
\end{prop}
 Building upon the conclusion of Proposition \ref{propeuler}, we will focus on the convergence analysis of the three spatial discretization approaches defined in \eqref{Deq}, \eqref{Eeq} and \eqref{Feq}.  Besides,  in the Vlasov case \eqref{vlasovcase},  a sufficient condition for Assumption \ref{AssumptionI} is  Assumption \ref{AssumptionII-V1} (see below), which follows from the Kantorovich-Rubinstein dual representation of the Wasserstein distance $\mathcal{W}_{1}$: 
\[\mathcal{W}_{1}(\mu, \nu)=\sup \Big\{\int_{\RD}f(\xi)\mu(d\xi)-\int_{\RD}f(\xi)\nu(d\xi)\:\Big|\: f \text{ Lipschitz continuous with }[f]_{\text{Lip}}\coloneqq\sup_{x\neq y}\tfrac{|f(x)-f(y)|}{|x-y|}\leq 1\Big\}\] 
and follows from the fact that for every $p\geq 1, \;\mathcal{W}_{1}(\mu, \nu)\leq \mathcal{W}_{p}(\mu, \nu)$ (see  e.g. \cite{EDWARDS2011387} and \cite[Chapter 6]{villani2008optimal}). 
\begin{manualtheorem}{II$_{V}$}\label{AssumptionII-V1}
The functions $\beta$ and $a$ in \eqref{vlasovcase} are $\gamma$-H\"older continuous in $t$, $\gamma\in(0,1]$, Lipschitz continuous in $(x,\,u)$ uniformly with respect to $t\in[0, T]$ in the following sense: there exists a constant $L>0$ such that 
\begin{align}
&\forall\, x,u \in\RD, \forall\,s, t\in[0,T] \:\text{with}\: s\leq t,\nonumber\\
&\qquad \left|\beta(t, x, u) - \beta(s, x, u)\right| \vee\vertiii{a(t, x, u) - a(s, x, u)}\leq L\big(1+|x|+|u|\big)(t-s)^{\gamma},\nonumber
\end{align}
and such that 
\begin{align}
&\forall \,t\in[0, T], \forall \,x_1,\,x_2,\,u_1,\,u_2 \in\mathbb{R}^{d},\nonumber\\ 
&\qquad \left|\beta(t, x_1, u_1) - \beta(t, x_2, u_2)\right| \vee\vertiii{a(t, x_1, u_1) - a(t, x_2, u_2)}\leq L\big(\,|x_1-x_2|+|u_1-u_2|\,\big).\nonumber
\end{align}
\end{manualtheorem}

\section{Particle method and its convergence rate}\label{DtoC}

 The following theorem shows the convergence rate of the particle method \eqref{Deq}.

\begin{thm}\label{thm1} Assume that Assumption \ref{AssumptionI} is in force for an index $p\in[2,+\infty)$.
 Set $M\in\mathbb{N}^*$ and $h=\frac{T}{M}$ for the temporal discretization. 
Let $(\bar{\mu}_{t_m})_{1\leq m\leq M}$, $(\bar{\mu}^{N}_{t_m})_{1\leq m\leq M}$ be probability distributions respectively defined by the theoretical Euler scheme \eqref{Ceq} and the particle method \eqref{Deq}. 
\begin{enumerate}[$(i)$]
\item We have  
\[\left\Vert \sup_{0\leq m\leq M}\mathcal{W}_{p}(\bar{\mu}_{t_{m}}^{N}, \bar{\mu}_{t_{m}})\right\Vert_{p}\leq C_{d, p, L, T}\vertii{\mathbb{W}_{p}(\bar{\mu}, \nu^{N})}_{p}\xrightarrow{\;N\rightarrow+\infty\;}0,\] 
where $\bar{\mu}$ is the probability distribution of the process $\bar{X}=(\bar{X}_{t})_{t\in[0,T]}$ defined  by \eqref{contEuler} and $\nu^{N}\coloneqq\frac{1}{N}\sum_{i=1}^N\delta_{Y^i}$ is the empirical measure on the i.i.d. copies $(Y^{1},..., Y^{N})$ of  $\bar{X}$. 
\smallskip
\item If, in addition, $\vertii{X_{0}}_{p+\varepsilon}<+\infty$ for some $\varepsilon>0$,  we have the following inequality 
\begin{align}
\left\Vert\sup_{0\leq m\leq M}\!\!\mathcal{W}_{p}(\bar{\mu}_{t_{m}}^{N}, \bar{\mu}_{t_{m}})\right\Vert_{p}\!\!\leq \widetilde{C}\times\begin{cases}
N^{-\frac{1}{2p}}+N^{-\frac{\varepsilon}{p(p+\varepsilon)}} \!\!& \mathrm{if}\:p>d/2\:\mathrm{and}\:\varepsilon\neq p,\\
N^{-\frac{1}{2p}}\big[\log(1+N)\big]^{\frac{1}{p}}+N^{-\frac{\varepsilon}{p(p+\varepsilon)}} \!\!& \mathrm{if}\:p=d/2\:\mathrm{and}\:\varepsilon\neq p,\\
N^{-\frac{1}{d}}+N^{-\frac{\varepsilon}{p(p+\varepsilon)}} \!\!& \mathrm{if}\:p\in(0, d/2)\:\mathrm{and}\:p+\varepsilon\neq \frac{d}{(d-p)},
\end{cases}\nonumber
\end{align}
where $\widetilde{C}$ is a constant depending on $p,\varepsilon, d, b, \sigma, L, T$ and $\vertii{X_{0}}_{p+\varepsilon}$.
\end{enumerate}
\end{thm}
Theorem \ref{thm1} can be considered as a \textit{strong} error of the particle method. The weak error, that is, the upper-bound of $|\Phi(\bar{\mu}_{t_m})-\EE \Phi(\bar{\mu}_{t_{m}}^{N})|$ for a function $\Phi: \mathcal{P}_{p}(\RD)\rightarrow \RD$ with appropriate differentiability, can be established in the future by applying similar techniques as  \cite{chassagneux2019weak}. Moreover, we refer to \cite{MR1910635} and \cite{Hoffmann_2022} for the studies of  the density simulation of $\bar{\mu}_{t_m}$ based on the particle method \eqref{Deq}. 

The proof of Theorem \ref{thm1} can be divided into  3 steps. 

\begin{itemize}
\item Step 1 (Subsection \ref{subsection21}):  We present some preliminary results on the McKean-Vlasov equation \eqref{Aeq} and the process $\bar{X}=(\bar{X}_t)_{t\in[0,T]}$  defined by the continuous Euler scheme \eqref{contEuler}. Additionally, we discuss basic properties of the space of probability distributions and the space of marginal distributions  of these two processes $X=(X_t)_{t\in[0,T]}$ and $\bar{X}=(\bar{X}_t)_{t\in[0,T]}$. 

\smallskip
\item Step 2  (Subsection \ref{subsection22}): The system $(\bar{X}^{1, N}_{t_m}, ...,\bar{X}^{N, N}_{t_m})_{0\leq m\leq M}$ defined by \eqref{Deq} can be considered as a particle system \textit{with} interaction through $\bar{\mu}^{N}_{t_{m}}= \frac{1}{N}\sum_{n=1}^{N}\delta_{\bar{X}^{n, N}_{t_{m}}}$ at each time $t_{m},  \,0\leq m\leq M$.   We define another particle system $(Y^{1}_{t}, ..., Y^{N}_{t})_{t\in[0,T]}$ \textit{without} interaction, that is,  $Y^{n}, \, 1\leq n\leq N$, are i.i.d. It\^o processes with the same distribution as $\bar{X}=(\bar{X}_t)_{t\in[0,T]}$ defined by  \eqref{contEuler}.
\smallskip
\item Step 3  (Subsection \ref{subsection23}): We compare these two particle systems and prove Theorem \ref{thm1}. 
\end{itemize}

\subsection{Preliminary results}\label{subsection21}

The following lemma shows that the $p$-th moments of  the unique strong solution $X=(X_t)_{t\in[0,T]}$ to the McKean-Vlasov equation \eqref{Aeq} and of the process $\bar{X}=(\bar{X}_t)_{t\in[0,T]}$  defined by the continuous Euler scheme \eqref{contEuler} are finite. 

\begin{lem}\label{finitemomentbar}(\!\!\cite[Proposition 2.1]{liu2020functional})
Assume that Assumption \ref{AssumptionI} holds for an index $p\in[2,+\infty)$. Then the unique solution $X=(X_t)_{t\in[0,T]}$ of \eqref{Aeq} and the process $\bar{X}^M=(\bar{X}^M_{t})_{t\in[0,T]}$ defined by the continuous Euler scheme \eqref{contEuler} satisfy the following inequality 
\begin{equation}\label{dnormX0}
\forall \,M\geq 1,\quad \Big\Vert \sup_{t\in[0,T]}\left|X_{t}\right|\,\Big\Vert _{p}\vee \Big\Vert \sup_{t\in[0,T]}\left|\bar{X}^{M}_{t}\right|\,\Big\Vert _{p}\leq C_{p,d,b,\sigma, T, L}\Big(1+\vertii{X_{0}}_{p}\Big),
\end{equation}
where $C_{p,d,b,\sigma, T, L}$ is a constant depending on $p, d,b,\sigma, T, L$. 
Hence, $X$, $\bar{X}\in L_{\CRD}^{p}\big(\Omega, \mathcal{F}, (\mathcal{F}_{t})_{t\geq0}, \mathbb{P}\big)$ and their respective probability distributions $\mu=\mathbb{P}\circ X^{-1}$ and $\bar{\mu}=\mathbb{P}\circ \bar{X}^{-1}$ lie in $\PPC$. 
\end{lem}

For any random variable $Y\in\LPC$, its probability distribution ${P}_{Y}$ naturally lies in  $\mathcal{P}_{p}\big(\mathcal{C}([0, T], \mathbb{R}^{d})\big)$. Moreover, the space $\mathcal{P}_{p}\big(\mathcal{C}([0, T], \mathbb{R}^{d})\big)$ equipped with the Wasserstein distance is complete and separable since $\big(\mathcal{C}([0, T], \mathbb{R}^{d}), \left\Vert\cdot\right\Vert_{\sup}\big)$ is a Polish space (see \cite{bolley2008separability}).  
Moreover, as $\big(\mathcal{P}_{p}(\mathbb{R}^{d}), \mathcal{W}_{p}\big)$ is a complete space (see \cite{bolley2008separability}), the space $\mathcal{C}\big( [0, T], \mathcal{P}_{p}(\mathbb{R}^{d})\big)$ equipped with the uniform distance $d_{\mathcal{C}}$ defined in \eqref{distancedc} is also complete. 
Recall that the  mapping $\iota: \PPC \rightarrow \CPP$ is defined by \eqref{defiota} and 
the well-posedness of the mapping $\iota$ has been proved in \cite[Lemma 5.1.2]{liuthese}, which also implies that the  marginal distributions of processes $X$ and $\bar{X}$ lie in the space $\CPP$. 
\begin{lem}Let $(\mu_t)_{t\in[0, T]}$ and $(\bar{\mu}_t)_{t\in [0, T]}$ be respective marginal distributions of the unique solution $X$ of \eqref{Aeq} and the process $\bar{X}$ defined by the continuous Euler scheme \eqref{contEuler}. Then $(\mu_t)_{t\in[0, T]},\, (\bar{\mu}_t)_{t\in [0, T]}\in \CPP$.
\end{lem}

Finally, the following Lemma, whose proof is postponed to Appendix \ref{appb}, shows  the relation between $\mathbb{W}_{p,t}(\mu, \nu)$ and $\sup_{s\in[0, t]}\mathcal{W}_{p}(\mu_{s}, \nu_{s})$.

\begin{lem}\label{relationd} 
For every $\mu, \nu\in \mathcal{P}_{p}\big(\mathcal{C}([0, T], \mathbb{R}^{d})\big)$, we have, for every $ t\in[0,T],$
\[\;\sup_{s\in[0, t]}\mathcal{W}_{p}(\mu_{s}, \nu_{s})\leq \mathbb{W}_{p,t}(\mu, \nu),\]
where $\mu_{s}=\mu\,\circ\,\pi_{s}^{-1}$ and $\nu_{s}=\nu\,\circ\,\pi_{s}^{-1}$.
In particular, for every $\mu, \nu\in \mathcal{P}_{p}\big(\mathcal{C}([0, T], \mathbb{R}^{d})\big)$, $d_{\mathcal{C}}(\iota(\mu), \iota(\nu))\leq \mathbb{W}_{p}(\mu,\nu)$. Hence the mapping $\iota$ defined in \eqref{defiota} is 1-Lipschitz continuous.
\end{lem}

\subsection{Definition of the particle systems with and without interaction}\label{subsection22}

We define the continuous expansion $(X_{t}^{1, N}, ..., X_{t}^{N, N})_{t\in[0,T]}$ of the particle system \eqref{Deq} as follows : set $\bar{X}_{0}^{\,n, N}, 1\leq n \leq N \widesim{\,\text{i.i.d.}\,}X_0$; for any $n\in\{1, ..., N\}$ and for any $t\in(t_{m}, t_{m+1}]$, set 
\begin{equation}\label{Nparteulercon}
\bar{X}_{t}^{\,n, N}=\bar{X}_{t_{m}}^{\,n, N}+(t-t_{m})b\big(t_{m}, \bar{X}_{t_{m}}^{\,n, N}, \bar{\mu}_{t_{m}}^{N}\big)+\sigma\big(t_{m}, \bar{X}_{t_{m}}^{\,n, N}, \bar{\mu}_{t_{m}}^{N}\big)(B_{t}^{n}-B_{t_{m}}^{n}) \text{ with }  \bar{\mu}^{N}_{t_{m}}\coloneqq \frac{1}{N}\sum_{n=1}^{N}\delta_{\bar{X}^{n, N}_{t_{m}}}.
\end{equation} 
For every $t\in[t_{m}, t_{m+1})$, we define $\underline{t}=t_{m}$. Then, the continuous expansion $(X_{t}^{1, N}, ..., X_{t}^{N, N})_{t\in[0,T]}$ are solutions to the following  equations with initial random variable $\bar{X}_{0}^{\,n, N}, 1\leq n \leq N,$
\begin{equation}\label{underlineuX}
\bar{X}_{t}^{\,n, N}\!\!=\bar{X}_{0}^{\,n, N}\!+\!\int_{0}^{t}b(\underline{u}, \bar{X}_{\underline{u}}^{\,n, N}, \bar{\mu}_{\underline{u}}^{N})du+\!\int_{0}^{t}\sigma(\underline{u}, \bar{X}_{\underline{u}}^{\,n, N}, \bar{\mu}_{\underline{u}}^{N})dB^{n}_{t}, \;1\!\leq \!n\!\leq \!N, \text{ with }\bar{\mu}_{\underline{t}}^{N}\coloneqq\frac{1}{N}\sum_{n=1}^{N}\delta_{\bar{X}_{\underline{t}}^{\,n, N}}.
\end{equation}
\begin{lem}\label{finitmomentXN}
Suppose that Assumption \ref{AssumptionI} holds for some $p\in[2,+\infty)$.
\begin{enumerate}[$(a)$]
\item  The coefficients $b$ and $\sigma$ have a linear growth in the sense that there exists a constant $C_{b,\sigma, L, T}$ depending on $b, \sigma$, $L$ and $T$ such that 
\begin{equation}\label{lgrowth}
\forall \, t\in[0, T], \forall \,x\in\mathbb{R}^{d}, \forall \,\mu\in\mathcal{P}_{p}(\mathbb{R}^{d}), \;\left|b(t, x,\mu)\right|\vee\vertiii{\sigma(t, x,\mu)}\leq C_{b,\sigma, L, T}(1+\left|x\right|+\mathcal{W}_{p}(\mu,\delta_{0})).
\end{equation}
\item Let $(\bar{X}_{t}^{1, N}, ..., \bar{X}_{t}^{N, N})_{t\in[0,T]}$ be processes defined by \eqref{Nparteulercon}.  Then for a fixed temporal discretization number $M\in\mathbb{N}^{*}$, we have  
\begin{equation}\label{finitenorm}
\forall \,n\in\{1, ..., N\}, \quad \Big\Vert \sup_{\;t\in[0,T]}\big|\bar{X}_{t}^{n, N}\big|\;\Big\Vert_p<+\infty.\end{equation}
\end{enumerate}
\end{lem}
The proof of Lemma \ref{finitmomentXN} is postponed to Appendix \ref{appb}. 
The  system $(\bar{X}_{t}^{1, N}, ..., \bar{X}_{t}^{N, N})_{t\in[0,T]}$ defined by \eqref{Nparteulercon} can be considered as a particle system having interaction through $\bar{\mu}^{N}_{t_{m}}\!\!= \!\frac{1}{N}\sum_{n=1}^{N}\!\delta_{\bar{X}^{n, N}_{t_{m}}}, \,0\leq\! m\leq \!M$.   Now we define another particle system $(Y^{1}_{t}, ..., Y^{N}_{t})_{t\in[0,T]}$ without interaction, which are essentially i.i.d. processes of $\bar{X}=(\bar{X}_t)_{t\in[0,T]}$ defined by the continuous Euler scheme \eqref{contEuler}. Based on the same Brownian motions $B^{n}, n=1, ..., N$ as in (\ref{Nparteulercon}), these processes $Y^{n}=(Y^{n}_{t})_{t\in[0, T]}, \,1\leq n\leq N,$ are defined by the following equations
\begin{align}\label{noninteraction}
&\forall\, n=1, ..., N,\, \forall \,m=0, ..., M-1, \,\forall\,t\in(t_m, t_{m+1}]\nonumber\\
&\qquad Y_t^{n}=Y_{t_{m}}^{n}+(t-t_{m})b(t_m, Y_{t_{m}}^{n}, \bar{\mu}_{t_{m}})+\sigma(t_m, Y_{t_{m}}^{n}, \bar{\mu}_{t_{m}})(B_{t}^{n}-B_{t_m}^{n}),\quad Y_{0}^{n}=\bar{X}_{0}^{n, N}
\end{align}
where for every $t\in[0,T]$, $\bar{\mu}_t$ in \eqref{noninteraction} is the marginal distribution at time $t$ of the process $\bar{X}=(\bar{X}_t)_{t\in[0,T]}$ defined by \eqref{contEuler}. By using the same notation $\underline{t}$ as in \eqref{underlineuX}, for every $n=1, ..., N$, $(Y^{n}_t)_{t\in[0,T]}$ defined by \eqref{noninteraction} is the solution of 
\[Y_t^{n}=Y_0^n+\int_{0}^{t}b(\underline{u}, Y_{\underline{u}}, \bar{\mu}_{\underline{u}})du+\int_{0}^{t}\sigma(\underline{u}, Y_{\underline{u}}, \bar{\mu}_{\underline{u}})dB^n_u \quad \text{with}\quad Y_0^n=\bar{X}_0^{n, N}. \]
Moreover, it is obvious that 
the processes $Y^{n}, \,1\leq n \leq N,$  are i.i.d copies of $\bar{X}$ and then 
\begin{equation}\label{empme}
\nu^{N, \omega}\coloneqq\frac{1}{N}\sum_{n=1}^{N}\delta_{Y^{n}(\omega)},\;\omega\in \Omega
\end{equation}
is the empirical measure of $\bar{\mu}=\PP\circ \bar{X}^{-1}$. When there is no ambiguity, we will write $\nu^{N}$ instead of $\nu^{N, \omega}$. The random measure $\nu^{N, \omega}$ is valued in $\PPC$. In fact, for every $\omega\in\Omega$, $Y^{n}(\omega)$ lies in $\CRD$ so that $\vertii{Y^{n}(\omega)}_{\sup}<+\infty$. Hence, for every $\omega\in\Omega$,
\[\int_{\CRD}\vertii{\xi}_{\sup}^{p}\nu^{N, \omega}(d\xi)=\frac{1}{N}\sum_{n=1}^{N}\vertii{Y^{n}(\omega)}_{\sup}^{p}<+\infty.\]
As before, we write $\nu^{N, \omega}_{t}\coloneqq \nu^{N, \omega}\circ\pi_{t}^{-1}=\frac{1}{N}\sum_{n=1}^{N}\delta_{Y^{n}_{t}(\omega)}.$ Thus, $\iota(\nu^{N, \omega})=(\nu^{N, \omega}_{t})_{t\in[0, T]}\in\CPP$. 

\subsection{Proof of Theorem \ref{thm1}}\label{subsection23}

The proof of Theorem \ref{thm1} relies on a variant version of Gronwall's Lemma (see e.g. \cite[Lemma 7.3]{pages2018numerical} for the proof) and on \cite[Theorem 1]{fournier2015rate}. For the reader's convenience, we restate the following Theorem \ref{FG} from \cite[Theorem 1]{fournier2015rate}, which provides a non-asymptotic upper bound on the convergence rate in the Wasserstein distance of the empirical measures.

\begin{lem}[``\`A la Gronwall" Lemma]\label{Gronwall}
Let $f : [0, T]\rightarrow\mathbb{R}_{+}$ be a Borel, locally bounded, non-negative and non-decreasing function and let $\psi: [0, T]\rightarrow\mathbb{R}_{+}$ be a non-negative non-decreasing function satisfying 
\[\forall \,t\in[0, T],\,\, f(t)\leq A\int_{0}^{t}f(s)ds+B\left(\int_{0}^{t}f^{2}(s)ds\right)^{\frac{1}{2}}+\psi(t),\]
where $A, B$ are two positive real constants. Then, for any $t\in[0, T],$ $ f(t)\leq 2e^{(2A+B^{2})t}\psi(t).$
\end{lem}

\begin{thm}(\!\!\cite[Theorem 1]{fournier2015rate})\label{FG}
Let $p>0$ and let $\mu\in\mathcal{P}_{q}(\mathbb{R}^{d})$ for some $q>p$. Let $U^{1}(\omega), ..., U^{n}(\omega), ...$ be i.i.d random variables with distribution $\mu$. Let $\mu_{n}^{\omega}$ denote the empirical measure of $\mu$ defined by
\[\mu_{n}^{\omega}\coloneqq \frac{1}{n}\sum_{i=1}^{n}\delta_{U^{i}(\omega)}.\]
Then, there exists a real constant $C$ only depending on $p,d,q$ such that, for all $n\geq 1$,
\[\small{\mathbb{E}\Big(\mathcal{W}_{p}^{p}(\mu_{n}^{\omega}, \mu)\Big)\leq CM_{q}^{p/q}(\mu)\times\begin{cases}
n^{-1/2}+n^{-(q-p)/q} & \:\mathrm{if}\:p>d/2\:\mathrm{and}\:q\neq2p\\
n^{-1/2}\log(1+n)+n^{-(q-p)/q} & \:\mathrm{if}\:p=d/2\:\mathrm{and}\:q\neq2p\\
n^{-p/d}+n^{-(q-p)/q} & \:\mathrm{if}\:p\in(0, d/2)\:\mathrm{and}\:q\neq d/(d-p)
\end{cases}},\]
where $M_{q}(\mu)\coloneqq \int_{\mathbb{R}^{d}}\left|\xi\right|^{q}\mu(d\xi)$.
\end{thm}
In particular, Theorem \ref{FG} implies that for $p\geq 2$, 
\begin{equation}\label{empir}
\small{\big\Vert\mathcal{W}_{p}(\mu_{n}^{\omega}, \mu)\big\Vert_{p}\leq CM_{q}^{1/q}(\mu)\times\begin{cases}
n^{-1/2p}+n^{-(q-p)/qp} & \:\mathrm{if}\:p>d/2\:\mathrm{and}\:q\neq2p\\
n^{-1/2p}\big(\log(1+n)\big)^{1/p}+n^{-(q-p)/qp} & \:\mathrm{if}\:p=d/2\:\mathrm{and}\:q\neq2p\\
n^{-1/d}+n^{-(q-p)/qp} & \:\mathrm{if}\:p\in(0, d/2)\:\mathrm{and}\:q\neq d/(d-p)
\end{cases}}.
\end{equation}

Moreover, we need the following Lemma, whose proof is postponed to Appendix \ref{appb}, for the proof of Theorem \ref{thm1}.

\begin{lem}\label{lemthm1} Assume that Assumption \ref{AssumptionI} holds for an index $p\in[2, +\infty)$. Then the coefficients $b$ and $\sigma$ satisfy the following inequalities 
\begin{align}\label{lineargrowth}
&\forall\,\big(X, (\mu_{t})_{t\in[0,T]}\big), \big(Y, (\nu_{t})_{t\in[0,T]}\big)\in\LPC\times \mathcal{C}\big([0, T], \mathcal{P}_{p}(\mathbb{R}^{d})\big),\;\forall \,t\in[0, T]\nonumber\\
&\quad\vertii{\sup_{s\in[0,t]}\left|\int_{0}^{s}\big[b(u, X_{u}, \mu_{u})-b(u, Y_{u}, \nu_{u})\big]du\right|}_{p}\leq L\int_{0}^{t}\big[\vertii{X_{u} - Y_{u}}_{p}+\vertii{\mathcal{W}_{p}(\mu_{u}, \nu_{u})}_{p}\big]du,\\
&\quad\vertii{\sup_{s\in[0,t]}\left|\int_{0}^{s}\big[\sigma(u, X_{u}, \mu_{u})-\sigma(u, Y_{u}, \nu_{u})\big]dB_{u}\right|}_{p}\leq C_{d, p, L}\Big\{\int_{0}^{t}\big[\vertii{X_{u}-Y_{u}}_{p}^{2}+\vertii{\mathcal{W}_{p}(\mu_{u}, \nu_{u})}_{p}^{2}\big]du\Big\}^{\frac{1}{2}},\nonumber
\end{align}
where $C_{d, p, L}$ is a constant  depending on $d, p, L$.
\end{lem} 

\begin{proof}[Proof of Theorem \ref{thm1}]

$(i)$ For every $n\in\{1, ..., N\}$, we have
\[\left|Y_{t}^{n}-\bar{X}_{t}^{\,n, N}\right|=\left| \int_{0}^{t}\big[ b(\underline{u},Y^{n}_{\underline{u}}, \bar{\mu}_{\underline{u}}) - b(\underline{u}, \bar{X}^{\,n, N}_{\underline{u}}, \bar{\mu}^{N}_{\underline{u}})\big]du + \int_{0}^{t}\big[\sigma(\underline{u}, Y^{n}_{\underline{u}}, \bar{\mu}_{\underline{u}}) - \sigma(\underline{u}, \bar{X}^{\,n, N}_{\underline{u}}, \bar{\mu}^{N}_{\underline{u}})\big]dB_{u} \right|.\]
Hence, 
\begin{flalign}
&\left\Vert \sup_{s\in[0, t]}\left| Y_{s}^{n}-\bar{X}_{s}^{n, N}\right|\right\Vert_{p}&\nonumber\\
&\quad\leq\vertii{\sup_{s\in[0, t]}\left| \int_{0}^{s}\big[b(\underline{u}, Y^{n}_{\underline{u}}, \bar{\mu}_{\underline{u}}) - b(\underline{u}, \bar{X}^{\,n, N}_{\underline{u}}, \bar{\mu}^{N}_{\underline{u}})\big]du\right|}_{p}  +\vertii{\sup_{s\in[0, t]}\left|\int_{0}^{s}\big[\sigma(\underline{u}, Y^{n}_{\underline{u}}, \bar{\mu}_{\underline{u}}) - \sigma(\underline{u}, \bar{X}^{\,n, N}_{\underline{u}}, \bar{\mu}^{N}_{\underline{u}})\big]dB_{u}\right|}_{p}\nonumber\\
&\quad\leq L\int_{0}^{t}\Big[\vertii{Y^{n}_{\underline{u}}-\bar{X}^{\,n, N}_{\underline{u}}}_{p}+\vertii{\mathcal{W}_{p}(\bar{\mu}_{\underline{u}},\bar{\mu}^{N}_{\underline{u}})}_{p}\Big]du\nonumber\\
& \quad\hspace{1cm} +C_{d, p, L}\Big[\int_{0}^{t}\Big[\vertii{Y^{n}_{\underline{u}}-\bar{X}^{\,n, N}_{\underline{u}}}_{p}^{2}+\vertii{\mathcal{W}_{p}(\bar{\mu}_{\underline{u}},\bar{\mu}^{N}_{\underline{u}})}_{p}^{2}\Big]du\Big]^{\frac{1}{2}}\hspace{1cm}\text{\big(by Lemma \ref{lemthm1}\big)}&\nonumber\\
&\quad\leq L\int_{0}^{t}\Big\Vert \sup_{v\in[0,u]} \left|Y^{n}_{v} - \bar{X}^{\,n, N}_{v}\right|\Big\Vert_{p}du+C_{d, p, L}\Big[ \int_{0}^{t}\Big\Vert \sup_{v\in[0,u]}\left|Y^{n}_{v} - \bar{X}^{\,n, N}_{v}\right|\Big\Vert^{2}_{p}du\Big]^{\frac{1}{2}}+\psi(t),&\nonumber
\end{flalign}
where 
\begin{equation}\label{defpsi}
\psi(t)=L\int_{0}^{t}\left\Vert \mathcal{W}_{p}(\bar{\mu}_{\underline{u}}, \bar{\mu}_{\underline{u}}^{N})\right\Vert_{p}du+C_{d, p, L}\Big[\int_{0}^{t}\left\Vert \mathcal{W}_{p}(\bar{\mu}_{\underline{u}}, \bar{\mu}_{\underline{u}}^{N})\right\Vert^{2}_{p}du\Big]^{\frac{1}{2}},
\end{equation}
owing to $\sqrt{a+b}\leq\sqrt{a}+\sqrt{b}$ for any $a\geq 0, b\geq0$. Then by Lemma \ref{Gronwall}, we have 
\[\Big\Vert \sup_{s\in[0, t]}\left| Y_{s}^{n}-\bar{X}_{s}^{n, N}\right|\Big\Vert_{p}\leq 2\,e^{(2L+C_{d, p, L}^{2})\,t}\psi(t).\]
Here we can apply Lemma \ref{Gronwall} since Lemma \ref{finitemomentbar} and Lemma \ref{finitmomentXN}-(b)  imply 
\[Y^{n}, \bar{X}^{n, N} \in L_{\CRD}^{p}\big(\Omega, \mathcal{F}, (\mathcal{F}_{t})_{t\geq0}, \mathbb{P}\big), \quad n=1, ..., N\]
as $Y^{n}, n=1,..., N\widesim{\text{i.i.d.}} \bar{X}$, thus the process $Y^{n}- \bar{X}^{n, N}$ is also in 
$L_{\CRD}^{p}\big(\Omega, \mathcal{F}, (\mathcal{F}_{t})_{t\geq0}, \mathbb{P}\big)$ for every $n=1, ..., N$. 

Moreover, the empirical measure $\frac{1}{N}\sum_{n=1}^{N}\delta_{(\bar{X}^{\,n, N}, Y^{n})}$ is a coupling of the random measures $\bar{\mu}^{N}$ and $\nu^{N}$. Thus 
\begin{align*}
\mathbb{E}\,\mathbb{W}_{p,t}^{p}& (\bar{\mu}^{N}, \nu^{N})=\mathbb{E}\Big[\inf_{\pi\in\Pi(\bar{\mu}^{N}, \nu^{N})}\int_{\mathcal{C}([0, T], \mathbb{R}^{d})\times\mathcal{C}([0, T], \mathbb{R}^{d})}\sup_{s\in[0,t]}\left|x_{s}-y_{s}\right|^{p}\pi(dx, dy)\Big]\\
&\leq \mathbb{E}\Big[\int_{\mathcal{C}([0, T], \mathbb{R}^{d})\times\mathcal{C}([0, T], \mathbb{R}^{d})}\sup_{s\in[0,t]}\left|x_{s}-y_{s}\right|^{p}\frac{1}{N}\sum_{n=1}^{N}\delta_{(\bar{X}^{\,n, N}, Y^{n})}(dx, dy)\Big]\\
&=\mathbb{E}\Big[\frac{1}{N}\sum_{n=1}^{N}\sup_{s\in[0,t]}\left| \bar{X}_{s}^{n, N}- Y_{s}^{n}\right|^{p}\Big]=\frac{1}{N}\sum_{n=1}^{N}\Big\Vert\sup_{s\in[0,t]}\left| \bar{X}_{s}^{n, N}- Y_{s}^{n}\right|\Big\Vert_{p}^{p}\\
&\leq \Big[2\,e^{(2L+C_{d, p, L}^{2})\,t}\psi(t)\Big]^{p}\leq\Big[2\,e^{(2L+C_{d, p, L}^{2})\,T}\psi(t)\Big]^{p}.
\end{align*}
 Lemma \ref{relationd} implies $\sup_{s\in[0,t]}\mathcal{W}_{p}^{p}(\bar{\mu}_{s}^{N}, \nu_{s}^{N})\leq \mathbb{W}_{p, t}^{p}(\bar{\mu}^{N}, \nu^{N})$. Hence, \begin{equation}\label{supw}
\Big\Vert \sup_{s\in[0,t]}\mathcal{W}_{p}(\bar{\mu}_{s}^{N}, \nu_{s}^{N})\Big\Vert_{p}\leq C_{d,p,L, T} \psi(t)
\end{equation}
with $C_{d,p,L, T}=2\,e^{(2L+C_{d, p, L}^{2})\,T}$. It follows that, 
\begin{align*}
&\Big\Vert \sup_{s\in[0,t]}\mathcal{W}_{p}(\bar{\mu}_{s}^{N}, \bar{\mu}_{s})\Big\Vert_{p}\leq \Big\Vert \sup_{s\in[0,t]}\mathcal{W}_{p}(\bar{\mu}_{s}^{N}, \nu_{s}^{N})\Big\Vert_{p}+\Big\Vert \sup_{s\in[0,t]}\mathcal{W}_{p}(\nu_{s}^{N}, \bar{\mu}_{s})\Big\Vert_{p}\\
&\quad\leq C_{d,p,L, T} \psi(t)+\Big\Vert \sup_{s\in[0,t]}\mathcal{W}_{p}(\nu_{s}^{N}, \bar{\mu}_{s})\Big\Vert_{p}\;\;\; (\text{by applying (\ref{supw})})\\
&\quad\leq \Big\Vert \sup_{s\in[0,t]}\mathcal{W}_{p}(\nu_{s}^{N}, \bar{\mu}_{s})\Big\Vert_{p}+ C_{d,p,L, T}\cdot L\int_{0}^{t}\Big\Vert \mathcal{W}_{p}(\bar{\mu}_{\underline{u}}, \bar{\mu}_{\underline{u}}^{N})\Big\Vert_{p}du+C_{d,p,L, T}\cdot C_{d, p, L}\Big[\int_{0}^{t}\Big\Vert \mathcal{W}_{p}(\bar{\mu}_{\underline{u}}, \bar{\mu}_{\underline{u}}^{N})\Big\Vert^{2}_{p}du\Big]^{\frac{1}{2}}\nonumber\\
&\hspace{1cm}(\text{by the defintion of $\psi(t)$ in (\ref{defpsi})})\\
&\quad\leq \Big\Vert \sup_{s\in[0,t]}\mathcal{W}_{p}(\nu_{s}^{N}, \bar{\mu}_{s})\Big\Vert_{p}+C_{d,p,L, T}\cdot L\int_{0}^{t}\Big\Vert\sup_{v\in[0,u]} \mathcal{W}_{p}(\bar{\mu}_{v}, \bar{\mu}_{v}^{N})\Big\Vert_{p}du\\
&\hspace{6cm}+C_{d,p,L, T}\cdot C_{d, p, L}\Big[\int_{0}^{t}\Big\Vert\sup_{v\in[0,u]} \mathcal{W}_{p}(\bar{\mu}_{v}, \bar{\mu}_{v}^{N})\Big\Vert^{2}_{p}du\Big]^{\frac{1}{2}}.
\end{align*}

\noindent Then, by Lemma \ref{Gronwall}, we obtain 
\begin{equation}\label{majemprical}
\Big\Vert\sup_{s\in[0,t]}\mathcal{W}_{p}(\bar{\mu}_{s}^{N}, \bar{\mu}_{s})\Big\Vert_{p}\leq 2e^{(2A+B^{2})T}\Big\Vert\sup_{s\in[0,t]}\mathcal{W}_{p}(\bar{\mu}_{s}, \nu_{s}^{N})\Big\Vert_{p},
\end{equation}
where $A=C_{d, p, L, T}L$ and $B=C_{d,p,L, T}\cdot C_{d, p, L}$. Finally, 
\begin{align*}
\Big\Vert\sup_{0\leq m\leq M}\mathcal{W}_{p}(\bar{\mu}_{t_{m}}^{N}, \bar{\mu}_{m})\Big\Vert_{p}&\leq 2e^{(2A+B^{2})T}\Big\Vert\sup_{s\in[0,T]}\mathcal{W}_{p}(\bar{\mu}_{s}, \nu_{s}^{N})\Big\Vert_{p}\nonumber\\
&\leq 2e^{(2A+B^{2})T}\vertii{\mathbb{W}_{p}(\bar{\mu}, \nu^{N})}_{p}\longrightarrow0 \;\;\text{ as }\;\;N\rightarrow+\infty,
\end{align*}
where the second inequality above follows from Lemma \ref{relationd}.

\smallskip

\noindent$(ii)$ If $\vertii{X_{0}}_{p+\varepsilon}<+\infty$ for some $\varepsilon>0$,  Assumption \ref{AssumptionI} holds with index $p+\varepsilon$ since for every $\mu, \nu\in\mathcal{P}_{ p+\varepsilon}(\RD), \,\mathcal{W}_p(\mu, \nu)\leq \mathcal{W}_{p+\varepsilon}(\mu, \nu)$ (see e.g. \cite[Remark 6.6]{villani2008optimal}). Then, the inequality \eqref{dnormX0} implies
\[\vertii{\bar{X}}_{p+\varepsilon}=\Big\Vert\sup_{u\in[0, T]}\left|\bar{X}_{u}\right|\Big\Vert_{p+\varepsilon}\leq C_{p, d, b, \sigma} \big(1+\vertii{X_{0}}_{p+\varepsilon}\big)<+\infty.\]
Thus $\bar{\mu}\in\mathcal{P}_{p+\varepsilon}\big(\mathcal{C}([0, T], \mathbb{R}^{d})\big)$, which implies that $\bar{\mu}_{s}\in\mathcal{P}_{p+\varepsilon}(\mathbb{R}^{d})$ for any $s\in[0, T]$. 

For any $s\in[0,T]$, $\nu_{s}^{N}$ is the empirical measure of $\bar{\mu}_{s}$. It follows from Theorem \ref{FG} that for any $s\in[0,T]$,
\begin{align}\label{empir1}
\vertii{\mathcal{W}_{p}(\nu_{s}^{N}, \bar{\mu}_{s})}_{p}&\leq CM_{p+\varepsilon}^{1/p+\varepsilon}(\bar{\mu}_{s})\nonumber\\
&\qquad\times\begin{cases}
N^{-1/2p}+N^{-\frac{\varepsilon}{p(p+\varepsilon)}} & \mathrm{if}\:p>d/2\:\mathrm{and}\:\varepsilon\neq p\\
N^{-1/2p}\big(\log(1+N)\big)^{1/p}+N^{-\frac{\varepsilon}{p(p+\varepsilon)}} &\mathrm{if}\:p=d/2\:\mathrm{and}\:\varepsilon\neq p\\
N^{-1/d}+N^{-\frac{\varepsilon}{p(p+\varepsilon)}} &\mathrm{if}\:p\in(0, d/2)\:\mathrm{and}\:p+\varepsilon\neq \frac{d}{(d-p)}
\end{cases},
\end{align}
where $C$ is a constant depending on $p, \varepsilon, d$.
Moreover, the inequality \eqref{dnormX0} implies that
\begin{align}
\sup_{s\in[0, T]}M_{p+\varepsilon}(\bar{\mu}_{s})&=\sup_{s\in[0, T]}\mathbb{E}\big[\left|X_{s}\right|^{p+\varepsilon}\big]\leq \Big\Vert\sup_{s\in[0, T]}\left| X_{s}\right|\Big\Vert_{p+\varepsilon}^{p+\varepsilon}\leq\Big[C_{p, d, b, \sigma, L,T}(1+\vertii{X_{0}}_{p+\varepsilon})\Big]^{p+\varepsilon}<+\infty.\nonumber
\end{align}
Hence, \vspace{-0.2cm}
\begin{align}\label{empir2}
\sup_{s\in[0, T]}\vertii{\mathcal{W}_{p}(\nu_{s}^{N}, \bar{\mu}_{s})}_{p}&\leq C'\times\begin{cases}
N^{-1/2p}+N^{-\frac{\varepsilon}{p(p+\varepsilon)}} & \mathrm{if}\:p>d/2\:\mathrm{and}\:\varepsilon\neq p\\
N^{-1/2p}\big(\log(1+N)\big)^{1/p}+N^{-\frac{\varepsilon}{p(p+\varepsilon)}} &\mathrm{if}\:p=d/2\:\mathrm{and}\:\varepsilon\neq p\\
N^{-1/d}+N^{-\frac{\varepsilon}{p(p+\varepsilon)}} &\mathrm{if}\:p\in(0, d/2)\:\mathrm{and}\:p+\varepsilon\neq \frac{d}{(d-p)}
\end{cases},
\end{align}
where ${C}'=C\sup_{s\in[0, T]}M^{\frac{1}{p+\varepsilon}}_{p+\varepsilon}(\bar{\mu}_{s})$ which is a constant depending on $p, \varepsilon, d, b, \sigma, L, T$ and $\vertii{X_{0}}_{p+\varepsilon}$.

Moreover, the inequality \eqref{supw} implies that
\begin{equation}\label{new}
\sup_{s\in[0, t]}\vertii{\mathcal{W}_{p}(\bar{\mu}_{s}^{N}, \nu_{s}^{N})}_{p}\leq \Big\Vert \sup_{s\in[0,t]}\mathcal{W}_{p}(\bar{\mu}_{s}^{N}, \nu_{s}^{N})\Big\Vert_{p}\leq C_{d,p,L, T}\psi(t).
\end{equation}
Then,
\begin{align*}
\sup_{s\in[0,t]}\left\Vert \mathcal{W}_{p}(\bar{\mu}_{s}^{N}, \bar{\mu}_{s})\right\Vert_{p}
&\leq  \sup_{s\in[0,t]} \left\Vert\mathcal{W}_{p}(\bar{\mu}_{s}^{N}, \nu_{s}^{N})\right\Vert_{p}+ \sup_{s\in[0,t]}\left\Vert \mathcal{W}_{p}(\nu_{s}^{N}, \bar{\mu}_{s})\right\Vert_{p}\\
&\leq C_{d,p,L, T} \psi(t)+\sup_{s\in[0,t]}\left\Vert \mathcal{W}_{p}(\nu_{s}^{N}, \bar{\mu}_{s})\right\Vert_{p}\;\;\; (\text{by applying (\ref{new})})\\
&\leq \sup_{s\in[0,t]}\left\Vert \mathcal{W}_{p}(\nu_{s}^{N}, \bar{\mu}_{s})\right\Vert_{p}+ C_{d,p,L, T}\cdot L\int_{0}^{t}\left\Vert \mathcal{W}_{p}(\bar{\mu}_{\underline{u}}, \bar{\mu}_{\underline{u}}^{N})\right\Vert_{p}du\nonumber\\
&\quad\;+C_{d,p,L, T}\cdot C_{d, p, L}\Big[\int_{0}^{t}\left\Vert \mathcal{W}_{p}(\bar{\mu}_{\underline{u}}, \bar{\mu}_{\underline{u}}^{N})\right\Vert^{2}_{p}du\Big]^{\frac{1}{2}}\:(\text{by the defintion of $\psi(t)$ in \eqref{defpsi}})\\
&\leq \sup_{s\in[0,t]}\left\Vert \mathcal{W}_{p}(\nu_{s}^{N}, \bar{\mu}_{s})\right\Vert_{p}+C_{d,p,L, T}\cdot L\int_{0}^{t}\sup_{v\in[0,u]}\left\Vert \mathcal{W}_{p}(\bar{\mu}_{v}, \bar{\mu}_{v}^{N})\right\Vert_{p}du\\
&\hspace{1cm}+C_{d,p,L, T}\cdot C_{d, p, L}\Big[\int_{0}^{t}\sup_{v\in[0,u]}\left\Vert \mathcal{W}_{p}(\bar{\mu}_{v}, \bar{\mu}_{v}^{N})\right\Vert^{2}_{p}du\Big]^{\frac{1}{2}},
\end{align*}

\noindent Then, by Lemma \ref{Gronwall}, we obtain 
\begin{equation}\label{majemprical}
\sup_{s\in[0,T]}\left\Vert\mathcal{W}_{p}(\bar{\mu}_{s}^{N}, \bar{\mu}_{s})\right\Vert_{p}\leq 2e^{(2A+B^{2})T}\sup_{s\in[0,T]}\left\Vert\mathcal{W}_{p}(\bar{\mu}_{s}, \nu_{s}^{N})\right\Vert_{p},
\end{equation}
where $A=C_{d, p, L, T}L$ and $B=C_{d,p,L, T}\cdot C_{d, p, L}$. Finally, it follows from (\ref{empir2}) that 
\begin{align}\label{sup1}
\sup_{0\leq m\leq M}&\vertii{\mathcal{W}_{p}(\bar{\mu}_{t_{m}}^{N}, \bar{\mu}_{m})}_{p}\leq \sup_{s\in[0,T]}\left\Vert\mathcal{W}_{p}(\bar{\mu}_{s}^{N}, \bar{\mu}_{s})\right\Vert_{p}\nonumber\\
&\leq \widetilde{C}\times\begin{cases}
N^{-\frac{1}{2p}}+N^{-\frac{\varepsilon}{p(p+\varepsilon)}} & \:\mathrm{if}\:p>d/2\:\mathrm{and}\:\varepsilon\neq p\\
N^{-\frac{1}{2p}}\big[\log(1+N)\big]^{\frac{1}{p}}+N^{-\frac{\varepsilon}{p(p+\varepsilon)}} & \:\mathrm{if}\:p=d/2\:\mathrm{and}\:\varepsilon\neq p\\
N^{-\frac{1}{d}}+N^{-\frac{\varepsilon}{p(p+\varepsilon)}} & \:\mathrm{if}\:p\in(0, d/2)\:\mathrm{and}\:p+\varepsilon\neq \frac{d}{(d-p)}
\end{cases},
\end{align}
where $\widetilde{C}$ is a constant depending on $p,\varepsilon, d, b, \sigma, L, T$ and $\vertii{X_{0}}_{p+\varepsilon}$.

By the definition of $\psi(t)$ in \eqref{defpsi}, we have 
\begin{align}
\psi(t)&=L\int_{0}^{t}\left\Vert \mathcal{W}_{p}(\bar{\mu}_{\underline{u}}, \bar{\mu}_{\underline{u}}^{N})\right\Vert_{p}du+C_{d, p, L}\Big[\int_{0}^{t}\left\Vert \mathcal{W}_{p}(\bar{\mu}_{\underline{u}}, \bar{\mu}_{\underline{u}}^{N})\right\Vert^{2}_{p}du\Big]^{\frac{1}{2}},\nonumber\\
&\leq LT\sup_{u\in[0, T]}\left\Vert \mathcal{W}_{p}(\bar{\mu}_{\underline{u}}, \bar{\mu}_{\underline{u}}^{N})\right\Vert_{p}+C_{d,p,L}\sqrt{T}\cdot \sup_{u\in[0, T]}\left\Vert \mathcal{W}_{p}(\bar{\mu}_{\underline{u}}, \bar{\mu}_{\underline{u}}^{N})\right\Vert_{p}\leq C_{d,p,L,T}\sup_{u\in[0, T]}\left\Vert \mathcal{W}_{p}(\bar{\mu}_{\underline{u}}, \bar{\mu}_{\underline{u}}^{N})\right\Vert_{p}\nonumber
\end{align}
so that \eqref{supw} and \eqref{sup1} imply that 
\begin{align}
\left\Vert \sup_{s\in[0,T]}\mathcal{W}_{p}(\bar{\mu}_{s}^{N}, \nu_{s}^{N})\right\Vert_{p}&\leq C_{d,p,L, T}\sup_{u\in[0, T]}\left\Vert \mathcal{W}_{p}(\bar{\mu}_{\underline{u}}, \bar{\mu}_{\underline{u}}^{N})\right\Vert_{p}\nonumber\\
&\leq  C\times\begin{cases}
N^{-\frac{1}{2p}}+N^{-\frac{\varepsilon}{p(p+\varepsilon)}} & \:\mathrm{if}\:p>d/2\:\mathrm{and}\:\varepsilon\neq p\\
N^{-\frac{1}{2p}}\big[\log(1+N)\big]^{\frac{1}{p}}+N^{-\frac{\varepsilon}{p(p+\varepsilon)}} & \:\mathrm{if}\:p=d/2\:\mathrm{and}\:\varepsilon\neq p\\
N^{-\frac{1}{d}}+N^{-\frac{\varepsilon}{p(p+\varepsilon)}} & \:\mathrm{if}\:p\in(0, d/2)\:\mathrm{and}\:p+\varepsilon\neq \frac{d}{(d-p)}
\end{cases}\nonumber\end{align}
where $C$ is a constant depending on $p,\varepsilon, d, b, \sigma, L, T$ and $\vertii{X_{0}}_{p+\varepsilon}$.
\end{proof}

\section{Quantization-based schemes and their error analyses}\label{qscheme}

This section is devoted to the error analyses of the quantization-based schemes \eqref{Eeq}, \eqref{Geq} and \eqref{Feq}. We start with a review of  optimal quantization, as well as its connection with the $K$-means clustering. We also present and  prove Theorem \ref{thm:quadbasedscheme}, offering an $L^2$-error analysis of the quantization-based scheme \eqref{Eeq} in Subsection \ref{reviewoq}. Next, in Subsection \ref{recurq}, we provide computational details of transition probabilities \eqref{Geq} for the recursive quantization scheme  and show how to optimally compute the quantizer $x^{(m)}$ by integrating Lloyd's algorithm to reduce the quantization error. Finally,  in Subsection \ref{FtoD}, we establish Proposition \ref{quanNparti}, the error analysis of the hybrid particle-quantization scheme.

\subsection{Preliminaries on the optimal quantization and the error analysis of the quantization-based scheme \eqref{Eeq}}\label{reviewoq}

In this subsection, $X: (\Omega, \mathcal{F}, \mathbb{P})\rightarrow \big(\RD, \mathcal{B}(\RD)\big)$ is a random variable having  $p$-th finite moment, $p\geq 1$, and its probability distribution is denoted by $\mu$. Let $K \in\mathbb{N}^{*}$ be the fixed quantization level.  
 Recall that the quantization error function $e_{K, p}(\mu, \cdot)$ of $\mu$ at level $K$ and order $p$ is defined by \eqref{def:quanti-error-function}. 
For a quantizer $x=(x_1, ..., x_K)$ such that $x_i\neq x_j$, $i\neq j$ and for $(V_k(x))_{1\leq k\leq K}$ a Vorono\"i partition generated by $x$ defined in \eqref{defvoi}, the quantization error of $x$ satisfying 
\begin{equation}\label{sumvoi}
e^{p}_{K, p}(\mu, x)=\sum_{k=1}^{K}\int_{V_k(x)}|\xi-x_k|^{p}\mu(d\xi).
\end{equation}
It is obvious that the error value does not depend on the choice of Vorono\"i partition. Moreover, still considering this Vorono\"i partition $(V_k(x))_{1\leq k\leq K}$, we have (see e.g. \cite[Lemma 3.4]{graf2000foundations} \footnote{The statement in \cite[Lemma 3.4]{graf2000foundations} is established only for an optimal quantizer. However, its proof is also valid for an arbitrary quantizer. })
\begin{equation}\label{eq:diff-repre}
e_{K, p}(\mu, x)=\big\Vert X-\widehat{X}^{x}\big\Vert_p=\mathcal{W}_p(\mu, \widehat{\mu}^{x})
\end{equation}
where $\widehat{X}^{x}\coloneqq \Proj_x (X)$ and $\widehat{\mu}^{x}\coloneqq \mu\circ \Proj_x^{-1}=\mathcal{L}(\widehat{X}^{x})$  with the projection function $\Proj_x$ defined by \eqref{defprojfun}.  

Recall that the optimal quantizer $x^{*}=(x^{*}_1, ..., x^{*}_K)$ at level $K$ and order $p$ for $\mu$ is defined by \eqref{defopquantizer}. 
The existence of such an optimal quantizer has been proved in \cite{pollard1982quantization}. There exist few results in the literature for the uniqueness of the optimal quantizer (see e.g. \cite{kieffer1983uniqueness}), but in practice, we only need one optimal quantizer for a further simulation as all optimal quantizers have the same quantization error by \eqref{defopquantizer}. 

\smallskip
Now we recall several fundamental properties of optimal quantization.  

\begin{prop}\label{propclassical} Let $K\in\mathbb{N}^{*}$ be the  quantization level and let $p\geq 1$. We consider an $\RD$-valued random variable $X$ having probability distribution $\mu\in\mathcal{P}_p(\RD)$. Let $x^{*}=(x_1^{*}, ..., x_K^{*})$ be an optimal quantizer of $X$ at level $K$ and order $p$.\begin{enumerate}[(1)]
\item (see \cite[Theorem 4.1 and Theorem 4.12]{graf2000foundations}) Assume that the support of $\mu$ satisfies card(supp$(\mu))\geq K$, where card(supp$(\mu))$ denotes the cardinality of the support of $\mu$. Then any $K$-level optimal quantizer $x^{*}$ contains $K$ different points, i.e. for every $i,j\in\{1, \dots, K\}$ such that $i\neq j$, we have $x_i^*\neq x_j^*.$

\item (Zador's theorem, see \cite{MR2398762} and \cite[Theorem 5.2]{pages2018numerical}) If $\mu\in \mathcal{P}_{p+\varepsilon}(\RD)$ for some $\varepsilon>0$, the optimal error defined by \eqref{def:optimalerror} has the following non-asymptotic upper bound for every quantization level $K\geq 1$, 
\begin{equation}\label{zador}
  e_{p, K}^{*}(X)=  e_{p, K}^{*}(\mu)\leq C_{d, p, \varepsilon}\,\mathcal{M}_{p+\varepsilon}(\mu)K^{-\frac{1}{d}},
\end{equation}
where  $C_{d, p, \varepsilon} \ge 0$ is a constant depending only on $d$, $p$ and $\varepsilon$, and where
\begin{align} \label{eq:V}
\mathcal M_p(\mu)=\mathcal M_p(X) \coloneqq \inf_{a \in \RR^d} \mathbb \Vert X - a\Vert_p.
\end{align} 
\item (Optimal discrete representation, see \cite[Lemma 3.4]{graf2000foundations} and \cite[Proposition 5.1-(b)]{pages2018numerical}) Let $\widehat{X}^{x^*}=\Proj_{\,x^*}(X)$ and let $\widehat{\mu}^{\,x^*}\coloneqq \mu\circ \Proj_{\,x^*}^{-1}=\mathcal{L}(\widehat{X}^{x^*})$. Then 
\begin{align}
&\big\Vert X-\widehat{X}^{x^*}\big\Vert_p  =\min\Big\{ \Vert X-Y\Vert_p\;\Big|\;Y:(\Omega, \mathcal{F}, \PP)\rightarrow \big(\RD, \mathcal{B}(\RD)\big)\text{ random variable such that}\nonumber\\
&\hspace{6cm} \text{the support of $Y$ has at most $K$ points} \Big\}\nonumber\\
&\text{and}\quad \mathcal{W}_p(\mu, \widehat{\mu}^{\,x^*})=\min\Big\{\mathcal{W}_{p}(\mu, \nu)\;\Big|\;\nu\in\mathcal{P}_p(\RD) \text{ such that } \mathrm{card(supp(}\nu\mathrm{))}\leq K\Big\}.\nonumber
\end{align}
\item  (Stationary property when $p=2$, see e.g. \cite[Proposition 5]{pages2018numerical}) For the quadratic optimal quantization (i.e. $p=2$),    any optimal quantizer $x^{*}$ is stationary in the sense that  
\begin{equation}\label{stationary}
\EE \big[\,X\,\big|\,\widehat{X}^{\,x^*}\,\big]=\widehat{X}^{\,x^*}.
\end{equation}

\item (Consistency of the optimal quantization, see \cite[Theorem 9]{pollard1982quantization} and \cite[Theorem 4 and Appendix A]{liu2018convergence}) Consider a probability distribution sequence $\mu_n, \, \mu\in\mathcal{P}_2(\RD),\,n\geq 1$, such that $\mathcal{W}_2(\mu_n, \mu)\rightarrow 0$ when $n\rightarrow 0$. For every $n\geq 1$, let $x^{(n)}$ be a quadratic optimal quantizer of $\mu_n$. Then any limiting point of $(x^{(n)})_{n\geq1}$ is a quadratic optimal quantizer of $\mu$. Moreover, we have 
\begin{equation}\label{consistencyopterr}
\big(e_{K, 2}^{*}(\mu_n)\big)^2-\big(e_{K, 2}^{*}(\mu)\big)^2\leq 4e_{K, 2}^{*}(\mu)\, \mathcal{W}_2(\mu_n, \mu)+4  \mathcal{W}_2^2(\mu_n, \mu). 
\end{equation}
\end{enumerate}
\end{prop}

From a numerical point of view, the main idea of the optimal quantization is to use $\widehat{X}^{\,x^*}=\Proj_{\,x^*}(X)$ (\textit{respectively,} $\widehat{\mu}^{\,x^*}=\mu\,\circ\, \Proj_{\,x^*}^{-1}=\mathcal{L}(\widehat{X}^{\,x^*})$) as an approximation of the target random variable $X$ (\textit{resp.} the target probability distribution $\mu$).  Proposition \ref{propclassical}-(3) shows that $\widehat{X}^{\,x^*}$ (\textit{respectively} $\widehat{\mu}^{\,x^*}$) is the closest discrete representation of $X$ (\textit{resp.} of $\mu$) with respect to the $L^p$-norm (\textit{resp.} the Wasserstein distance $\mathcal{W}_p$), among all random variables (\textit{resp.} probability distributions) having a support of at most $K$ points.
Furthermore, the Proposition \ref{propclassical}-(5) shows the connection between the optimal quantization and the \textit{$K$-means clustering} in unsupervised learning (see e.g. \cite{pollard1982quantization} and \cite{liu2018convergence}). Consider a sample $\{\eta_1, ..., \eta_n\}\subset \RD$, the \textit{$K$-means clustering} is essentially to find an optimal quantizer with respect to the empirical measure $\bar\mu_n\coloneqq\frac{1}{n}\sum_{i=1}^{n}\delta_{\eta_i}$ over the sample. In the unsupervised learning context, such an optimal quantizer is called \textit{cluster centers} and $\eta_1, ..., \eta_n$ are often considered as i.i.d. samples having probability distribution $\mu$ so that $\mathcal{W}_p(\bar\mu_n, \mu)\rightarrow 0$ a.s. if $\mu\in\mathcal{P}_p(\RD)$ (see e.g. \cite[Thereom 7]{pollard1982quantization}). 

\smallskip
The following theorem shows the $L^2$-error analysis of the quantization-based scheme  \eqref{Eeq}.

\begin{thm}[Error analysis of the quantization-based scheme]\label{thm:quadbasedscheme}  Assume that Assumption \ref{AssumptionI} is satisfied with $p=2$.
 Set $M\in \mathbb{N}^{*}$ and $h=\frac{T}{M}$ for the temporal discretization. Let $(\bar{X}_{t_m})_{0\leq m \leq M}$ be random variables defined by the Euler scheme \eqref{Ceq}. Consider a fixed $K$-level quantizer sequence $x^{(m)}=(x^{(m)}_1, ..., x^{(m)}_K), \; 0\leq m \leq M$ and define $(\widetilde{X}_{t_m})_{0\leq m \leq M}$, $(\widehat{X}_{t_m})_{0\leq m \leq M}$ by the quantization-based scheme \eqref{Eeq}. 
\begin{enumerate}[$(i)$]
\item For every $m\in\{1, ..., M\}$, we have \vspace{-0.2cm} 
\begin{equation}\label{errortheoq}
\vertii{\bar{X}_{t_{m}}-\widehat{X}_{t_{m}}}_{2}\leq \sum_{j=0}^{m}\big[1+2Lh (1+Lh+Lq)\big]^{j}\cdot \Xi_{m-j},\vspace{-0.2cm}
\end{equation}
where for every $j\in\{0, ..., m\}$, $\Xi_{j}$ denotes the quadratic quantization error of the $j$-th step, that is,  $\Xi_{j}=\big\Vert \widetilde{X}_{t_j}-\widehat{X}_{t_j}\big\Vert_2.$
\item If moreover, there exists $\varepsilon>0$ such that $\Vert X_0\Vert_{2+\varepsilon}<+\infty$  and if for every $0\leq m \leq M$, $x^{(m)}$ is a quadratic optimal quantizer of $\widetilde{X}_{t_m}$,  we have 
\begin{equation}\label{errortheoq2}
\vertii{\bar{X}_{t_{m}}-\widehat{X}_{t_{m}}}_{2}\leq C_{b, \sigma, L, T, d, \varepsilon, q, \Vert X_0\Vert_{2+\varepsilon}}\cdot K^{-1/d}\Big(\sum_{j=0}^{m}\big[1+2Lh (1+Lh+Lq)\big]^{j}\Big).
\end{equation}
\end{enumerate}
\end{thm}

\vspace{-0.5cm}

\begin{proof}[Proof of Theorem \ref{thm:quadbasedscheme}]$(i)$ 
To simplify the notation, we will denote by \[\bar{b}_{m}\coloneqq b(t_m, \bar{X}_{t_{m}}, \bar{\mu}_{t_{m}}),\: \bar{\sigma}_{m}\coloneqq\sigma(t_m, \bar{X}_{t_{m}}, \bar{\mu}_{t_{m}}),\: \widehat{b}_{m}\coloneqq b(t_m, \widehat{X}_{t_{m}}, \widehat{\mu}_{t_{m}}),\: \widehat{\sigma}_{m}\coloneqq\sigma(t_m, \widehat{X}_{t_{m}}, \widehat{\mu}_{t_{m}}).\] 
The definitions of $\bar{X}_{t_{m}}$, $\widetilde{X}_{t_{m}}$ and  in  \eqref{Ceq} and \eqref{Eeq} directly imply that
\[\left|\bar{X}_{t_{m+1}}-\widetilde{X}_{t_{m+1}}\right|=\left| \big(\bar{X}_{t_{m}}-\widehat{X}_{t_{m}}\big)+h\big( \bar{b}_{m}-\widehat{b}_{m}\big) +\sqrt{ h}\big( \bar{\sigma}_{m}-\widehat{\sigma}_{m}\big)Z_{m+1}\right|.\]
Hence,
\begin{align}\label{theoq}
&\EE \Big[\big|\bar{X}_{t_{m+1}}-\widetilde{X}_{t_{m+1}}\big|^{2}\Big]\nonumber\\
&=\EE \Big[\big|\bar{X}_{t_{m}}-\widehat{X}_{t_{m}}\big|^{2}\Big]+ h^{2}\EE \Big[\big|\bar{b}_{m}-\widehat{b}_{m}\big|^{2}\Big]+2h\, \EE\Big[ \big\langle \bar{X}_{t_{m}}-\widehat{X}_{t_{m}}\,,\, \bar{b}_{m}-\widehat{b}_{m}\big\rangle\Big]\nonumber\\
&\;\;\;\;+h\EE \Big[\big |\big(\bar{\sigma}_{m}-\widehat{\sigma}_{m}\big)Z_{m+1}\big|^{2}\Big]+2\sqrt{h}\,\EE\Big[\big\langle\big(\bar{X}_{t_{m}}-\widehat{X}_{t_{m}}\big) + h\big(\bar{b}_{m}-\widehat{b}_{m}\big) \,,\, \big(\bar{\sigma}_{m}-\widehat{\sigma}_{m}\big)Z_{m+1}\big\rangle\Big].
\end{align}
Let $\mathcal{F}_0$ denote the $\sigma$-algebra generated by $X_{0}$. For every $m\in\{1, ..., M\}$, we define $\mathcal{F}_{m}$ the $\sigma$-algebra generated by $X_{0}, Z_{1}, ..., Z_{m}$. Then, as $Z_{m+1}$ is independent of $\mathcal{F}_{m}$ and $\bar{X}_{t_{m}}, \widehat{X}_{t_{m}}, \bar{b}_{m}, \widehat{b}_{m}, \bar{\sigma}_{m}, \widehat{\sigma}_{m}$ are $\mathcal{F}_{m}$-measurable, we have
\begin{align}
\EE\Big[\Big\langle\big(\bar{X}_{t_{m}}&-\widehat{X}_{t_{m}}\big) + h\big(\bar{b}_{m}-\widehat{b}_{m}\big) \,,\,\big(\bar{\sigma}_{m}-\widehat{\sigma}_{m}\big)Z_{m+1}\Big\rangle\Big]\nonumber\\
&=\EE \Big\{\EE \Big[\big[\big(\bar{X}_{t_{m}}-\widehat{X}_{t_{m}}\big) + h\big(\bar{b}_{m}-\widehat{b}_{m}\big)\big]^{\top}\big(\bar{\sigma}_{m}-\widehat{\sigma}_{m}\big)Z_{m+1}\big| \,\mathcal{F}_{m}\,\Big]\Big\}\nonumber\\
&=\EE \Big\{\Big[\big[\big(\bar{X}_{t_{m}}-\widehat{X}_{t_{m}}\big) + h\big(\bar{b}_{m}-\widehat{b}_{m}\big)\big]^{\top}\big(\bar{\sigma}_{m}-\widehat{\sigma}_{m}\big)\Big]\EE \big[Z_{m+1}\big] \Big\}=0. \nonumber
\end{align}

Moreover, Assumption \ref{AssumptionI} implies that
\begin{align}
&\mathbb{E}\Big[\big|\bar{b}_{m}-\widehat{b}_{m}\big|^{2}\Big]\leq 2L^{2}\Big[\mathbb{E}\Big[\big|\bar{X}_{t_{m}}-\widehat{X}_{t_{m}}\big|^{2}\Big]+\mathcal{W}_{2}^{2}(\bar{\mu}_{m}, \widehat{\mu}_{m})\Big]\leq 4L^{2}\,\EE \Big[\big|\bar{X}_{t_{m}}-\widehat{X}_{t_{m}}\big|^{2}\Big]\nonumber&
\end{align}
and
\begin{align}
&\EE\Big[ \big\langle \bar{X}_{t_{m}}-\widehat{X}_{t_{m}}\,,\, \bar{b}_{m}-\widehat{b}_{m}\big\rangle\Big]\leq \vertii{\bar{X}_{t_{m}}-\widehat{X}_{t_{m}}}_{2}\vertii{\bar{b}_{m}-\widehat{b}_{m}}_{2}\leq 2L \,\EE\Big[ \big|\bar{X}_{t_{m}}-\widehat{X}_{t_{m}}\big|^{2}\Big],\nonumber&
\end{align}
as well as 
\begin{align}
\mathbb{E}\Big[\big|(\bar{\sigma}_{m}-\widehat{\sigma}_{m})Z_{m+1}\big|^{2}\Big]&\leq\mathbb{E}\Big[\vertiii{\bar{\sigma}_{m}-\widehat{\sigma}_{m}}^{2}|Z_{m+1}|^{2}\Big]\leq \mathbb{E}\Big[\mathbb{E}\big[\vertiii{\bar{\sigma}_{m}-\widehat{\sigma}_{m}}^{2}|Z_{m+1}|^{2}\:\big | \:\mathcal{F}_{m}\big]\Big]&\nonumber\\
&=\mathbb{E}\Big[\vertiii{\bar{\sigma}_{m}-\widehat{\sigma}_{m}}^{2}\mathbb{E}\big[|Z_{m+1}|^{2}\big]\Big]\leq 4L^{2}q \,\EE \Big[\big|\bar{X}_{t_{m}}-\widehat{X}_{t_{m}}\big|^{2}\Big],\nonumber&
\end{align}
where we recall that $\vertiii{\cdot}$ denotes the operator norm defined by $\vertiii{A}\coloneqq \sup_{\left|z\right|\leq 1}\left|Az\right|$.
Consequently, 
\begin{align}
\EE \Big[\big|\bar{X}_{t_{m+1}}-\widetilde{X}_{t_{m+1}}\big|^{2}\Big]\leq \big[1+4Lh (1+Lh+Lq)\big]\cdot\EE \Big[\big|\bar{X}_{t_{m}}-\widehat{X}_{t_{m}}\big|^{2}\Big]\nonumber
\end{align}
so that 
\begin{align}
\vertii{\bar{X}_{t_{m+1}}-\widetilde{X}_{t_{m+1}}}_{2}&\leq \sqrt{\,1+4Lh (1+Lh+Lq)\,}\vertii{\bar{X}_{t_{m}}-\widehat{X}_{t_{m}}}_{2}\nonumber\\
&\leq \big[1+2Lh (1+Lh+Lq)\big]\vertii{\bar{X}_{t_{m}}-\widehat{X}_{t_{m}}}_{2}\nonumber
\end{align}
where in the last inequality, we use the fact that $\sqrt{1+2x}\leq 1+x$ when $x\geq 0$. Hence,
\begin{align}
\vertii{\bar{X}_{t_{m+1}}-\widehat{X}_{t_{m+1}}}_{2}&\leq \vertii{\bar{X}_{t_{m+1}}-\widetilde{X}_{t_{m+1}}}_{2}+\vertii{\widetilde{X}_{t_{m+1}}-\widehat{X}_{t_{m+1}}}_{2}\nonumber\\
&\leq \big[1+2Lh (1+Lh+Lq)\big]\vertii{\bar{X}_{t_{m}}-\widehat{X}_{t_{m}}}_{2}+\Xi_{m+1},\nonumber
\end{align}
where $\Xi_m=\big\Vert \widetilde{X}_{t_m}-\widehat{X}_{t_m} \big\Vert_2$ denotes the quadratic quantization error at time $t_m$ (see \eqref{eq:diff-repre}). 
This directly implies \[\vertii{\bar{X}_{t_{m}}-\widehat{X}_{t_{m}}}_{2}\leq \sum_{j=0}^{m}\big[1+2Lh (1+Lh+Lq)\big]^{j}\Xi_{m-j}.\]

\noindent $(ii)$  {\sc Step 1.} We first prove that there exists a constant $\widetilde{C}_{\max}>0$ such that 
\begin{equation}\label{objstep1}
\forall \, m\in\{0, ..., M\}, \quad \Vert \widetilde{X}_{t_m}\Vert_{2+\varepsilon}\leq\widetilde{C}_{\max}.
\end{equation} 
As for every $m\in\{0, ..., M\}$, the quantizer $x^{(m)}=(x^{(m)}_1, ..., x^{(m)}_K)$ is quadratic optimal for $\widetilde{X}_{t_m}$, Proposition \ref{propclassical}-(4) implies that $\EE\big[\widetilde{X}_{t_m}\;\big|\;\widehat{X}_{t_m}\big]=\widehat{X}_{t_m}$. Hence, for every $p\geq1$, 
\begin{equation}\label{hatletilde}
\EE\Big[\big|\widehat{X}_{t_m}\big|^p\Big]=\EE\Big[\;\Big|\,\EE\big[\widetilde{X}_{t_m}\;\big|\;\widehat{X}_{t_m}\big]\,\Big|^p\,\Big]\leq \EE\Big[\,\EE\big[\big|\widetilde{X}_{t_m}\big|^p\;\big|\;\widehat{X}_{t_m}\big]\,\Big]=\EE\Big[\big|\widetilde{X}_{t_m}\big|^p\Big].
\end{equation}
On the other hand, we have 
\begin{align}
\Vert \widetilde{X}_{t_{m+1}}\Vert_p &\leq \Vert \widehat{X}_{t_m}\Vert_p + h \big\Vert b(t_m, \widehat{X}_{t_m}, \widehat{\mu}_{t_m}) \big\Vert_p +\sqrt{h}\vertii{  \vertiii{\sigma(t_m, \widehat{X}_{t_m}, \widehat{\mu}_{t_m})} \big|Z_{m+1}\big|}_p\nonumber\\
&\leq  \Vert \widehat{X}_{t_m}\Vert_p + h \big\Vert b(t_m, \widehat{X}_{t_m}, \widehat{\mu}_{t_m}) \big\Vert_p +\sqrt{h}\Big\Vert \vertiii{\sigma(t_m, \widehat{X}_{t_m}, \widehat{\mu}_{t_m})}\Big\Vert_p \big\Vert Z_{m+1}\big\Vert_p\nonumber\\
&\leq \Big(1+2h C_{b, \sigma, L, T}+2\sqrt{h} \Vert Z_{m+1}\big\Vert_pC_{b, \sigma, L, T}\Big)\big\Vert \widehat{X}_{t_m}\big\Vert_p+h C_{b, \sigma, L, T}+\sqrt{h} \Vert Z_{m+1}\big\Vert_pC_{b, \sigma, L, T}\nonumber
\end{align}
where the first inequality follows from the definition of $\widetilde{X}_{t_m}$ in \eqref{Eeq} and Minkowski's inequality, the second inequality follows from the fact that $\sigma(t_m, \widehat{X}_{t_m}, \widehat{\mu}_{t_m})\independent Z_{m+1}$ and the third inequality follows from Lemma \ref{finitmomentXN}-(a) and the inequality $\mathcal{W}_p(\widehat{\mu}_{t_m}, \delta_0)\leq \Vert \widehat{X}_{t_m}\Vert_p$.
Thus, the inequality \eqref{hatletilde} implies that 
\[\big\Vert \widetilde{X}_{t_{m+1}}\big\Vert_p\leq C_1\big\Vert  \widetilde{X}_{t_{m}}\big\Vert_p+C_2\]
with $C_1\coloneqq 1+2h C_{b, \sigma, L, T}+2\sqrt{h} \Vert Z_{m+1}\big\Vert_pC_{b, \sigma, L, T}$ and $C_2\coloneqq h C_{b, \sigma, L, T}+\sqrt{h} \Vert Z_{m+1}\big\Vert_pC_{b, \sigma, L, T}$.
As $\widetilde{X}_0=X_0$,  we have 
\begin{equation}\label{concstep1}
\forall\, m\in\{1, ..., M\},\quad \big\Vert \widetilde{X}_{t_{m}}\big\Vert_p \leq C_1^m \big\Vert X_{0}\big\Vert_p+\sum_{j=0}^{m-1}C_1^jC_2
\end{equation}
Thus \eqref{objstep1} is obtained by considering $p=2+\varepsilon$ and by letting $\widetilde{C}_{\max}\coloneqq C_1^M \big\Vert X_{0}\big\Vert_{2+\varepsilon}+\sum_{j=0}^{M-1}C_1^jC_2$.

\noindent{\sc Step 2.} As for every $m\in\{0, ..., M\}$, the quantizer $x^{(m)}=(x^{(m)}_1, ..., x^{(m)}_K)$ is quadratic optimal for $\widetilde{X}_{t_m}$, Zador's theorem (Proposition \ref{propclassical}-(2)) implies that   
\[\forall\,m\in\{0, ..., M\}, \quad \Xi_{m} \leq C_{d, \varepsilon}\mathcal{M}_{2+\varepsilon}(\widetilde{X}_{t_m}) K^{-\frac{1}{d}}.\]
Then we can conclude by remarking $\mathcal{M}_{2+\varepsilon}(\widetilde{X}_{t_m})\leq \big\Vert \widetilde{X}_{t_m}\big\Vert_{2+\varepsilon}\leq \widetilde{C}_{\max}$.
\end{proof}

\vspace{-0.3cm}

\subsection{Recursive quantization-based scheme for the Vlasov equation}\label{recurq}

In this section, we present the recursive quantization scheme in the Vlasov setting and show how to integrate Lloyd's algorithm into this scheme to reduce the quantization error at each time step.

Recall that in the Vlasov setting, there exist $\beta:[0,T]\times\mathbb{R}^{d}\times\mathbb{R}^{d}\rightarrow \mathbb{R}^{d}$ and $a: [0,T]\times\mathbb{R}^{d}\times\mathbb{R}^{d}\rightarrow \mathbb{M}_{d, q}(\mathbb{R})$ such that 
$b(t, x, \mu)=\int_{\mathbb{R}^{d}}\beta(t, x, u)\mu(du)\hspace{0.15cm}\text{and}\hspace{0.15cm}\sigma(t, x, \mu)=\int_{\mathbb{R}^{d}}a(t, x, u)\mu(du).$
Consider a quantizer sequence $x^{(m)}=(x_{1}^{(m)}, ..., x_{K}^{(m)})\in(\mathbb{R}^{d})^{K}, \;0\leq m\leq M$ such that for every $m=\{0, ..., M\}$, we have $x^{(m)}_i\neq x^{(m)}_j$ if $i\neq j$. Write $\big(V_{k}(x^{(m)})\big)_{1\leq k\leq K}$ for a Vorono\"i partition generated by $x^{(m)}$. 
Then the transition step of theoretical quantization-based scheme  $\eqref{Eeq}$ can be written as 
\vspace{-0.2cm}
\begin{align}
\widetilde{X}_{t_{m+1}}=\widehat{X}_{t_{m}}+ h\sum_{k=1}^{K}p_{k}^{(m)}\beta(t_m, \widehat{X}_{t_{m}}, x_{k}^{(m)})+\sqrt{ h} \Big[\sum_{k=1}^{K}p_{k}^{(m)}a(t_m, \widehat{X}_{t_{m}}, x_{k}^{(m)})\Big]Z_{m+1},\nonumber
\end{align}
where $p_{k}^{(m)}=\PP\big(\widehat{X}_{t_m}=x_k^{(m)}\big)=\PP\big(\widetilde{X}_{t_m}\in V_{k}(x^{(m)})\big)$ is the weight of $x_{k}^{(m)}$. 
Hence, given $\widehat{X}_{t_{m}}$, $\widetilde{X}_{t_{m+1}}$ has a normal distribution
\vspace{-0.2cm}
\[\widetilde{X}_{t_{m+1}}\sim\mathcal{N}\Big(\widehat{X}_{t_{m}}+ h\sum_{k=1}^{K}p_{k}^{(m)}\beta(t_m, \widehat{X}_{t_{m}}, x_{k}^{(m)}), \;\; h\big[\sum_{k=1}^{K}p_{k}^{(m)}a(t_m, \widehat{X}_{t_{m}}, x_{k}^{(m)})\big]^{\top}\big[\sum_{k=1}^{K}p_{k}^{(m)}a(t_m, \widehat{X}_{t_{m}}, x_{k}^{(m)})\big]\;\Big)\]
since $Z_{m+1}\sim \mathcal{N}(0, \mathbf{I}_{q})$.  Under the condition that $p^{(m)}$ is known, it follows that 
\begin{align}\label{transproba}
& \mathbb{P}\big(\widehat{X}_{t_{m+1}}=x_{j}^{(m+1)}\mid \widehat{X}_{t_{m}}=x_{i}^{(m)}\big)\nonumber\\
&\quad=\mathbb{P}\big(\widetilde{X}_{t_{m+1}}\in V_{j}(x^{(m+1)})\mid \widehat{X}_{t_{m}}=x_{i}^{(m)}\big)\nonumber\\
&\quad=\mathbb{P}\Big[\big(x_{i}^{(m)}+ h\sum_{k=1}^{K}p_{k}^{(m)}\beta(t_m, x_{i}^{(m)}, x_{k}^{(m)})+\sqrt{ h}\sum_{k=1}^{K}p_{k}^{(m)}a(t_m, x_{i}^{(m)}, x_{k}^{(m)})Z_{m+1}\big)\in V_{j}(x^{(m+1)})\Big]
\end{align}
and obviously, by letting 
\[\mathcal{E}_{i}(x^{(m)}, p^{(m)}, Z_{m+1})\coloneqq x_{i}^{(m)}+ h\sum_{k=1}^{K}p_{k}^{(m)}\beta(t_m, x_{i}^{(m)}, x_{k}^{(m)})+\sqrt{ h}\sum_{k=1}^{K}p_{k}^{(m)}a(t_m, x_{i}^{(m)}, x_{k}^{(m)})Z_{m+1},\] 
we get
\begin{align}\label{pm1given}
& \mathbb{P}\big(\widehat{X}_{t_{m+1}}=x_{j}^{(m+1)} \big)=\mathbb{P}\big(\widetilde{X}_{t_{m+1}}\in V_{j}(x^{(m+1)}) \big) =\sum_{i=1}^{K}\mathbb{P}\big(\widehat{X}_{t_{m+1}}=x_{j}^{(m+1)}\mid \widehat{X}_{t_{m}}=x_{i}^{(m)}\big)\cdot\mathbb{P}(\widehat{X}_{t_{m}}=x_{i}^{(m)})\nonumber\\
&\quad=\sum_{i=1}^{K}\mathbb{P}\Big(\mathcal{E}_{i}(x^{(m)}, p^{(m)}, Z_{m+1})\in V_{j}(x^{(m+1)})\Big)\cdot p_{i}^{(m)}\eqqcolon g(x^{(m)}, p^{(m)}, Z_{m+1}).
\end{align}
Thus one can compute the weights  $p^{(m+1)}_j=\PP\big(\widehat{X}_{t_{m+1}}=x^{(m+1)}_j\big),\; 1\leq j\leq K,$  starting from the weight vector $p^{(m)}$ at time $t_{m}$ by applying \eqref{pm1given}.

The formula \eqref{pm1given} is valid for any quantizer sequence $x^{(m)}$ with distinct components. Nonetheless, the quantization errors $\Xi_{m}=\Vert \widetilde{X}_{t_m}-\widehat{X}_{t_m}\Vert_2$ of each time step $t_m$ impact on the simulation result, which is also indicated in the error analysis \eqref{errortheoq}. One way to reduce the quantization error $\Xi_{m}$ is to integrate Lloyd's algorithm (see Algorithm \ref{lloyd}) into the recursive quantization scheme. Given $x^{(m)}$ and $p^{(m)}$, the Lloyd's iteration \eqref{lloyditeration} sending $x^{(m+1,\, [l])}=(x^{(m+1,\, [l])}_1, ..., x^{(m+1,\, [l])}_K)$ to $x^{(m+1,\, [l+1])}=(x^{(m+1,\, [l+1])}_1, ..., x^{(m+1,\, [l+1])}_K)$ for time step $t_{m+1}$ is
\begin{equation}\label{lloyditer1}
x_{k}^{(m+1,\, [l+1])}=\frac{\int_{V_{k}(x^{(m+1,\, [l])})}\xi\,\widetilde{\mu}_{t_{m+1}}(d\xi)}{\widetilde{\mu}_{t_{m+1}}\big(V_{k}(x^{(m+1,\, [l])})\big)}, \quad k=1, ..., K.
\end{equation}
The denominator of \eqref{lloyditer1} can be directly computed by \eqref{pm1given} while the numerator of \eqref{lloyditer1} is essentially to compute the following value 
\begin{align}\label{Elloyd}
\mathbb{E}\big[\widetilde{X}_{t_{m+1}}&\mathbbm{1}_{V_{j}(x^{(m+1)})}(\widetilde{X}_{t_{m+1}})\big]=\sum_{i=1}^{K}\mathbb{E}\big[\widetilde{X}_{t_{m+1}}\mathbbm{1}_{V_{j}(x^{(m+1)})}(\widetilde{X}_{t_{m+1}})\mid \widehat{X}_{t_{m}}=x_{i}^{(m)}\big]\cdot p_{i}^{(m)},
\end{align}
where 
\begin{equation}\label{whyslow}
\mathbb{E}\big[\widetilde{X}_{t_{m+1}}\mathbbm{1}_{V_{j}(x^{(m+1)})}(\widetilde{X}_{t_{m+1}})\mid \widehat{X}_{t_{m}}=x_{i}^{(m)}\big]=\mathbb{E}\big[Y\mathbbm{1}_{V_{j}(x^{(m+1)})}(Y)\big]
\end{equation}
with 
\begin{equation}
Y\sim \mathcal{N}\Big(x_{i}^{(m)}+ h\sum_{k=1}^{K}p_{k}^{(m)}\beta(t_m, x_{i}^{(m)}, x_{k}^{(m)}), \;\; h\big[\sum_{k=1}^{K}p_{k}^{(m)}a(t_m, x_{i}^{(m)}, x_{k}^{(m)})\big]^{\top}\big[\sum_{k=1}^{K}p_{k}^{(m)}a(t_m, x_{i}^{(m)}, x_{k}^{(m)})\big]\,\Big).\nonumber
\end{equation}

\begin{rem} In dimension 1,  Lloyd's iteration above is numerically low cost, 
since the  Vorono\"i cells in dimension 1 are in fact  intervals of $\overline{\RR}$. For example, let $x=(x_{1}, ..., x_{K})\in\mathbb{R}^{K}$ be a quantizer such that $x_{i}<x_{i+1}$, $i=1, ..., K-1$, one can choose a Vorono\"i partition as follows: 
\begin{align}
&V_{1}(x)=\big(-\infty, \frac{x_{1}+x_{2}}{2}\big)\quad\text{and}\quad V_{K}(x)=\big[\frac{x_{K-1}+x_{K}}{2}, +\infty),\nonumber\\
&V_{k}(x)=\big[\frac{x_{k-1}+x_{k}}{2},\frac{x_{k}+x_{k+1}}{2}\big), \hspace{1cm} k=2, ..., K-1.\nonumber
\end{align}
Let $x^{(m)}=(x^{(m)}_{1}, ..., x^{(m)}_{K})$ be the quantizer of the $m$-th time step. The transition probability $\pi_{ij}^{(m)}=\mathbb{P}\big(\widehat{X}_{t_{m+1}}=x_{j}^{(m+1)}\mid \widehat{X}_{t_{m}}=x_{i}^{(m)}\big)$ in  (\ref{transproba}) reads 
\[F_{m, \sigma^{2}}\Big(\frac{x_{j+1}^{(m)}+x_{j}^{(m)}}{2}\Big)-F_{m, \sigma^{2}}\Big(\frac{x_{j-1}^{(m)}+x_{j}^{(m)}}{2}\Big),\]
where $F_{m, \sigma^{2}}$ denotes the cumulative distribution function of $\mathcal{N}(m, \sigma^{2})$ with 
\begin{align}\label{mandsigma}
m=x_{i}^{(m)}+ h\sum_{k=1}^{K}p_{k}^{(m)}\beta(x_{i}^{(m)}, x_{k}^{(m)}) \;\text{ and }\;\sigma={\sqrt{h\,}}\Big[\,\sum_{k=1}^{K}p_{k}^{(m)}a(x_{i}^{(m)}, x_{k}^{(m)})\,\Big].
\end{align}
Moreover,  Lloyd's iteration (\ref{whyslow}) depends on the value 
\begin{equation}\label{avoidinter}
\int_{(x_{j-1}^{(m)}+x_{j}^{(m)})/{2}}^{(x_{j+1}^{(m)}+x_{j}^{(m)})/{2}}\xi \cdot f_{m, \sigma^{2}}(\xi)d\xi
\end{equation}
where $f_{m, \sigma^{2}}(\xi)$ is the density function of $\mathcal{N}(m, \sigma^{2})$ with $m$ and $\sigma$ defined in (\ref{mandsigma}). In fact, to avoid computing the integral, (\ref{avoidinter}) can be alternatively calculated by the following formula: for every $a,b\in\RR$,
\begin{align}\label{alterformula}
\int_{a}^{b}\xi \cdot f_{m, \sigma^{2}}(\xi)d\xi&=\int_{a}^{b} \frac{1}{\sqrt{2\pi\sigma^{2}}}\xi e^{-\frac{(\xi-m)^{2}}{2\sigma^{2}}}d\xi=-\frac{\sigma}{\sqrt{2\pi}}\left[e^{-\frac{(\xi-m)^{2}}{2\sigma^{2}}}\right]^{b}_{a}+m\big[F_{m, \sigma^{2}}(b)-F_{m, \sigma^{2}}(a)\big].
\end{align} 

However, in higher dimension, the main difficulty of this method is computing the integral of the density function of a multi-dimensional normal distribution over a Vorono\"i cell (see \eqref{pm1given}, \eqref{whyslow}). There exists several numerical solutions such as \texttt{pyhull} package in \texttt{Python} (see also the website \texttt{www.qhull.com}) for 
 the cubature formulas of the numerical integration over a convex set in low and medium dimensions ($d=2,3,4$). 
\end{rem}

\subsection{$L^{2}$-error analysis of the hybrid particle-quantization scheme }\label{FtoD}

The error analysis of the hybrid particle-quantization scheme \eqref{Feq} is established in the following proposition.

\vspace{-0.1cm}
\begin{prop}\label{quanNparti}
Set the same temporal discretization as Theorem \ref{thm1} and Theorem \ref{thm:quadbasedscheme}.
Assume that Assumption \ref{AssumptionI} holds true with $p=2$. For any $m\in\{1, ..., M\}$, let $\bar{\mu}_{t_m}^N$ be the empirical measure on the particles $(\bar{X}_{t_m}^{0, N}, ..., \bar{X}_{t_m}^{N, N})$, defined by the particle system \eqref{Deq} and let $\widehat{\mu}_{t_m}^{K}$ be the quantized measure defined in \eqref{Feq}. Then, 
\begin{equation}\label{quapar}
\forall\, 1\leq m\leq M, \quad \mathbb{E}\,\Big[\mathcal{W}_{2}\big(\widehat{\mu}_{t_{m}}^{K}, \bar{\mu}_{t_{m}}^{N}\big)\Big]\leq C_{2}\sum_{j=0}^{m-1}C_{1}^{j}\sqrt{\mathbb{E}\,\hat{\Xi}^{2}_{m-1-j}}+\mathbb{E}\,\hat{\Xi}_{m}.
\end{equation}
where $\hat{\Xi}_{m}=\mathcal{W}_{2}\big(\widehat{\mu}_{t_{m}}^{K}, \frac{1}{N}\sum_{i=1}^{N}\delta_{\widetilde{X}_{t_{m}}^{n, N}}\big)$ denotes the quadratic quantization error at time $t_m$ and  $C_{1}, C_{2}$ are positive constants depending on $L$, $q$ and $T$. 
\end{prop}

\begin{rem}\label{markov-quan-error}
$(a)$ We would like to highlight that the upper-bounds in \eqref{errortheoq}, \eqref{errortheoq2} and \eqref{quapar} are not numerically optimal since the two sums $\sum_{j=0}^{M}\big[1+2Lh (1+Lh+Lq)\big]^{j}$ and $\sum_{j=0}^{M-1}C_{1}^{j}\sqrt{\mathbb{E}\,\hat{\Xi}^{2}_{m-1-j}}$ at the right-hand side  converge to infinity as the temporal discretization number $M\rightarrow+\infty$. 
 These error bounds are commonly observed in Markovian quantization-based schemes (see also in \cite{pages2015recursive}). The resolution of the convergence of these error bounds, when the temporal discretization number $M\rightarrow+\infty$, remains an open problem. However, we do not observe a significant simulation error on numerical experiments (see further Section \ref{example}) even when considering large values of $M$.

\noindent $(b)$   On the right-hand side of \eqref{quapar}, the quantization error term  $\hat{\Xi}_{m}=\mathcal{W}_{2}\big(\widehat{\mu}_{t_{m}}^{K}, \frac{1}{N}\sum_{i=1}^{N}\delta_{\widetilde{X}_{t_{m}}^{n, N}}\big)$ is random due to the random probability measure  $\frac{1}{N}\sum_{i=1}^{N}\delta_{\widetilde{X}_{t_{m}}^{n, N}}$. Consequently, the direct application of Zador's theorem (see Proposition \ref{propclassical}-(2)) is not theoretically feasible when considering an optimal quantizer. This limitation comes from the fact that the term $\mathcal{M}_{p+\varepsilon}\big(\frac{1}{N}\sum_{i=1}^{N}\delta_{\widetilde{X}_{t_{m}}^{n, N}(\omega)}\big)$ is random and cannot be uniformly bounded for every $\omega\in \Omega$. 
\end{rem}

\begin{proof}[Proof of Proposition \ref{quanNparti}]
To simplify the notation, we will denote by 
\[b_{m}^{\text{Q},n}\coloneqq b(t_m, \widetilde{X}_{t_{m}}^{n, N}, \widehat{\mu}_{t_{m}}^{K}),\; b_{m}^{\text{P},n}\coloneqq b(t_m, \bar{X}_{t_{m}}^{n, N}, \bar{\mu}_{t_{m}}^{N}),\;\sigma_{m}^{\text{Q},n}\coloneqq \sigma(t_m, \widetilde{X}_{t_{m}}^{n, N}, \widehat{\mu}_{t_{m}}^{K}),\;\sigma_{m}^{\text{P},n}\coloneqq\sigma(t_m,\bar{X}_{t_{m}}^{n, N}, \bar{\mu}_{t_{m}}^{N}),\]
where the superscript ``P'' indicates the \textit{Particle method} and the superscript ``Q'' indicates the \textit{hybrid particle-Quantization scheme}. Moreover, let $\mathcal{F}_{m}$ be the $\sigma-$algebra generated by $X_{0}, \bar{X}_0^{1, N}, ..., \bar{X}_0^{N, N}$, $Z_{j}^{n},  n=1, ..., N, j=1, ..., m$. Then $b_{m}^{\text{P},n}$, $b^{\text{Q},n}_{m}$, $\sigma_{m}^{\text{P},n}$, $\sigma_{m}^{\text{Q},n}$ are $\mathcal{F}_{m}$-measurable and $Z_{m+1}^{n}, 1\leq n\leq N$ are independent of $\mathcal{F}_{m}$. It follows that 
\begin{align}\label{difftildebar}
\widetilde{X}_{t_{m+1}}^{n, N}-\bar{X}_{t_{m+1}}^{n, N}=&\widetilde{X}_{t_{m}}^{n, N}-\bar{X}_{t_{m}}^{n, N}+h\big(b_{m}^{\text{Q},n}-b_{m}^{\text{P},n}\big)+\sqrt{ h} \big(\sigma_{m}^{\text{Q},n}-\sigma_{m}^{\text{P},n}\big)Z_{m+1}^{n},
\end{align}
then 
\begin{align}\label{deftildebar2}
\mathbb{E}\Big[\big|\widetilde{X}_{t_{m+1}}^{n, N}-\bar{X}_{t_{m+1}}^{\,n, N}\big|^{2}\Big]&=\mathbb{E}\Big[\big|(\widetilde{X}_{t_{m}}^{n, N}-\bar{X}_{t_{m}}^{\,n, N})+h\big(b_{m}^{\text{Q},n}-b_{m}^{\text{P},n}\big) \big|^{2}\Big]+\mathbb{E}\Big[\big|\sqrt{ h}\big(\sigma_{m}^{\text{Q},n}-\sigma_{m}^{\text{P},n}\big)Z_{m+1}^{n}\big|^{2}\Big]\nonumber\\
&\hspace{0.5cm}+2\mathbb{E}\Big[\Big\langle\big(\widetilde{X}_{t_{m}}^{n, N}-\bar{X}_{t_{m}}^{\,n, N}\big)+h\big(b_{m}^{\text{Q},n}-b_{m}^{\text{P},n}\big) \,,\,\sqrt{h}\big(\sigma_{m}^{\text{Q},n}-\sigma_{m}^{\text{P},n}\big)Z_{m+1}^{n}\Big\rangle\Big]. 
\end{align}
The third term in \eqref{deftildebar2} equals to 0 since $Z_{m+1}^{n}, n\in\{1, ..., N\}$, are independent of $\mathcal{F}_{m}$. 
Moreover, Assumption \ref{AssumptionI} implies,
\[\mathbb{E}\Big[\big|b_{m}^{\text{Q},n}-b_{m}^{\text{P},n}\big|^{2}\Big]\leq 2L^{2}\Big\{\mathbb{E}\Big[\big|\widetilde{X}^{n, N}_{t_{m}}-\bar{X}_{t_{m}}^{n, N}\big|^{2}\Big]+\mathbb{E}\Big[\mathcal{W}_{2}^{2}(\widehat{\mu}_{t_{m}}^{K}, \bar{\mu}_{t_{m}}^{N})\Big]\Big\},\]
and
\begin{align*}
&\mathbb{E}\Big[\big|\sqrt{ h\,}(\sigma_{m}^{\text{Q},n}-\sigma_{m}^{\text{P},n})Z_{m+1}^{n}\big|^{2}\Big]\leq h\mathbb{E}\Big[\mathbb{E}\Big[\vertiii{\sigma_{m}^{\text{Q},n}-\sigma_{m}^{\text{P},n}}^{2} \left|Z_{m+1}^{n}\right|^{2}\;\Big|\; \mathcal{F}_{m}\Big]\Big]= h\,\mathbb{E}\Big[\vertiii{\sigma_{m}^{\text{Q},n}-\sigma_{m}^{\text{P},n}}^{2} \mathbb{E}\Big[\left|Z_{m+1}^{n}\right|^{2}\Big]\Big]\\
&\qquad= h\,{q}\,\mathbb{E}\Big[\vertiii{\sigma_{m}^{\text{Q},n}-\sigma_{m}^{\text{P},n}}^{2}\Big]\leq 2L^{2} h\,{q}\Big[\mathbb{E}\big|\widetilde{X}^{n, N}_{t_{m}}-\bar{X}_{t_{m}}^{n, N}\big|^{2}+\mathbb{E}\mathcal{W}_{2}^{2}(\widehat{\mu}_{t_{m}}^{K}, \bar{\mu}_{t_{m}}^{N})\Big].
\end{align*}
Hence, (\ref{deftildebar2}) becomes
\begin{align*}
\mathbb{E}&\Big[\big|\widetilde{X}_{t_{m+1}}^{n, N}-\bar{X}_{t_{m+1}}^{\,n, N}\big|^{2}\Big]\\
&=\mathbb{E}\Big[\big|(\widetilde{X}_{t_{m}}^{n, N}-\bar{X}_{t_{m}}^{n, N})+h\big(b_{m}^{\text{Q},n}-b_{m}^{\text{P},n}\big) \big|^{2}\Big]+\mathbb{E}\Big[\big|\sqrt{ h}\big(\sigma_{m}^{\text{Q},n}-\sigma_{m}^{\text{P},n}\big)Z_{m+1}^{n}\big|^{2}\Big]\\
&=\mathbb{E}\Big[\big|\widetilde{X}_{t_{m}}^{n, N}-\bar{X}_{t_{m}}^{n, N}\big|^{2}\Big]+\mathbb{E}\Big[h^{2}\left|b_{m}^{\text{Q},n}-b_{m}^{\text{P},n}\right|^{2} \Big]+2 h\mathbb{E}\Big[\big\langle\widetilde{X}_{t_{m}}^{n, N}-\bar{X}_{t_{m}}^{n, N}\,,\, b_{m}^{\text{Q},n}-b_{m}^{\text{P},n}\big\rangle\Big]\\
&\hspace{0.4cm}+\mathbb{E}\Big[\big|\sqrt{ h}\big(\sigma_{m}^{\text{Q},n}-\sigma_{m}^{\text{P},n}\big)Z_{m+1}^{n}\big|^{2}\Big]\\
&\leq \mathbb{E}\Big[\big|\widetilde{X}_{t_{m}}^{n, N}-\bar{X}_{t_{m}}^{n, N}\big|^{2}\Big]+ 2L^{2}( h^{2}+ h{q})\Big[\mathbb{E}\big|\widetilde{X}^{n, N}_{t_{m}}-\bar{X}_{t_{m}}^{n, N}\big|^{2}+\mathbb{E}\mathcal{W}_{2}^{2}(\widehat{\mu}_{t_{m}}^{K}, \bar{\mu}_{t_{m}}^{N})\Big]\\
&\hspace{0.4cm}+ h\,\mathbb{E}\Big[\big|\widetilde{X}_{t_{m}}^{n, N}-\bar{X}_{t_{m}}^{n, N}\big|^{2}+\big|b_{m}^{\text{Q},n}-b_{m}^{\text{P},n}\big|^{2}\Big]\\
&\leq \mathbb{E}\Big[\big|\widetilde{X}_{t_{m}}^{n, N}-\bar{X}_{t_{m}}^{n, N}\big|^{2}\Big]+ 2L^{2}( h^{2}+ h{q})\Big[\mathbb{E}\big|\widetilde{X}^{n, N}_{t_{m}}-\bar{X}_{t_{m}}^{n, N}\big|^{2}+\mathbb{E}\mathcal{W}_{2}^{2}(\widehat{\mu}_{t_{m}}^{K}, \bar{\mu}_{t_{m}}^{N})\Big]\\
&\hspace{0.4cm}+ h\,\mathbb{E}\Big[\big|\widetilde{X}_{t_{m}}^{n, N}-\bar{X}_{t_{m}}^{n, N}\big|^{2}\Big]+2L^{2} h\Big[\mathbb{E}\big|\widetilde{X}^{n, N}_{t_{m}}-\bar{X}_{t_{m}}^{n, N}\big|^{2}+\mathbb{E}\mathcal{W}_{2}^{2}(\widehat{\mu}_{t_{m}}^{K}, \bar{\mu}_{t_{m}}^{N})\Big],\\
&\leq \big({1+2L^{2}(h^{2}+hq)+h+2L^{2}h}\big)\mathbb{E}\Big[\big|\widetilde{X}_{t_{m}}^{n, N}-\bar{X}_{t_{m}}^{n, N}\big|^{2}\Big]+\big(2L^{2}(h^{2}+hq)+2L^{2}h\big)\mathbb{E}\mathcal{W}_{2}^{2}(\widehat{\mu}_{t_{m}}^{K}, \bar{\mu}_{t_{m}}^{N}).\nonumber
\end{align*}
Remark that for any $m\in\{1, ..., M\}$, the measure $\frac{1}{N}\sum_{n=1}^{N}\delta_{(\widetilde{X}_{t_{m}}^{n, N}, \bar{X}_{t_{m}}^{n, N})}$ is a random coupling of $\frac{1}{N}\sum_{n=1}^{N}\delta_{\widetilde{X}_{t_{m}}^{n, N}}$ and $\bar{\mu}_{t_{m}}^{N}=\frac{1}{N}\sum_{n=1}^{N}\delta_{\bar{X}_{t_{m}}^{n, N}}$. Thus, for any $m\in\{1, ..., M\},$
\begin{align}\label{coup}
\mathbb{E}\Big[\mathcal{W}_{2}^{2}\Big(\frac{1}{N}\sum_{n=1}^{N}\delta_{\widetilde{X}_{t_{m}}^{n, N}},  \bar{\mu}_{t_{m}}^{N}\Big)\Big]&\leq \mathbb{E}\Big[\int_{\mathbb{R}^{d}\times\mathbb{R}^{d}}\left|x-y\right|^{2}\frac{1}{N}\sum_{n=1}^{N}\delta_{(\widetilde{X}_{t_{m}}^{n, N}, \bar{X}_{t_{m}}^{n, N})}(dx, dy)\Big]\nonumber\\
&=\mathbb{E}\Big[\frac{1}{N}\sum_{n=1}^{N}\big|\widetilde{X}_{t_{m}}^{n, N}-\bar{X}_{t_{m}}^{n, N}\big|^{2}\Big].
\end{align}
Hence, by letting $C_{1}\coloneqq{1+2L^{2}(h^{2}+hq)+h+2L^{2}h}$, $C_{2}\coloneqq{2L^{2}(h^{2}+hq)+2L^{2}h}$ and by denoting the quantization error of the time $t_m$ by $\hat{\Xi}_{m}=\mathcal{W}_{2}\big(\widehat{\mu}_{t_{m}}^{K}, \frac{1}{N}\sum_{n=1}^{N}\delta_{\widetilde{X}_{t_{m}}^{n, N}}\big)$, we get
\begin{align}\label{inqua}
&\sqrt{\frac{1}{N}\sum_{n=1}^{N}\mathbb{E}\Big[\big|\widetilde{X}_{t_{m+1}}^{n, N}-\bar{X}_{t_{m+1}}^{\,n, N}\big|^{2}\Big]\,}\nonumber\\
&=\sqrt{\frac{C_{1}}{N}\sum_{n=1}^{N}\mathbb{E}\Big[\big|\widetilde{X}_{t_{m}}^{n, N}-\bar{X}_{t_{m}}^{n, N}\big|^{2}\Big]+C_{2}\mathbb{E}\big[\mathcal{W}_{2}^{2}(\widehat{\mu}_{t_{m}}^{K}, \bar{\mu}_{t_{m}}^{N})\big]}\nonumber\\
&\leq \sqrt{\frac{C_{1}}{N}\sum_{n=1}^{N}\mathbb{E}\Big[\big|\widetilde{X}_{t_{m}}^{n, N}-\bar{X}_{t_{m}}^{n, N}\big|^{2}\Big]+C_{2}\Big\{2\mathbb{E}\Big[\mathcal{W}_{2}^{2}(\widehat{\mu}_{t_{m}}^{K}, \frac{1}{N}\sum_{n=1}^{N}\delta_{\widetilde{X}_{t_{m}}^{n, N}})\Big]+2\mathbb{E}\Big[\mathcal{W}_{2}^{2}(\frac{1}{N}\sum_{n=1}^{N}\delta_{\widetilde{X}_{t_{m}}^{n, N}}, \bar{\mu}_{t_{m}}^{N})\Big]\Big\}}\nonumber\\
&\leq\sqrt{ (C_{1}+2C_{2})\cdot\frac{1}{N}\sum_{n=1}^{N}\mathbb{E}\Big[\big|\widetilde{X}_{t_{m}}^{n, N}-\bar{X}_{t_{m}}^{n, N}\big|^{2}\Big] + 2C_{2}\mathbb{E}\,\hat{\Xi}_{m}^{2}}\hspace{0.4cm}(\text{by (\ref{coup})})\nonumber\\
&\leq\sqrt{C_{1}+2C_{2}}\sqrt{\frac{1}{N}\sum_{n=1}^{N}\mathbb{E}\Big[\big|\widetilde{X}_{t_{m}}^{n, N}-\bar{X}_{t_{m}}^{n, N}\big|^{2}\Big]\,}+\sqrt{2C_{2}}\sqrt{\mathbb{E}\,\hat{\Xi}_{m}^{2}}.
\end{align}
Let $\bar{C}_{1}\coloneqq \sqrt{C_{1}+2C_{2}\,}$ and $\bar{C}_{2}=\sqrt{2C_{2}\,}$. The inequality (\ref{inqua}) implies 
\[\sqrt{\frac{1}{N}\sum_{n=1}^{N}\mathbb{E}\Big[\big|\widetilde{X}_{t_{m}}^{n, N}-\bar{X}_{t_{m}}^{n, N}\big|^{2}\Big]\,}\leq \bar{C}_{2}\sum_{j=0}^{m-1}\bar{C}_{1}^{j}\sqrt{\mathbb{E}\,\hat{\Xi}^{2}_{m-1-j}}.\]
Hence, it follows from (\ref{coup}) that 
\begin{align}
\mathbb{E}\,\left[\mathcal{W}_{2}\big(\frac{1}{N}\sum_{n=1}^{N}\delta_{\widetilde{X}_{t_{m}}^{n, N}},\,\bar{\mu}_{t_{m}}^{N}\big)\right]&\leq\sqrt{\mathbb{E}\,\mathcal{W}_{2}^{2}\Big(\frac{1}{N}\sum_{n=1}^{N}\delta_{\widetilde{X}_{t_{m}}^{n, N}},\,\bar{\mu}_{t_{m}}^{N}\Big)}\leq \sqrt{\frac{1}{N}\sum_{n=1}^{N}\mathbb{E}\Big[\big|\widetilde{X}_{t_{m}}^{n, N}-\bar{X}_{t_{m}}^{n, N}\big|^{2}\Big]\,}\nonumber\\
&\leq \bar{C}_{2}\sum_{j=0}^{m-1}\bar{C}_{1}^{j}\sqrt{\mathbb{E}\,\hat{\Xi}^{2}_{m-1-j}}.\nonumber
\end{align}
Consequently, 
\begin{align}
\mathbb{E}\,\Big[\mathcal{W}_{2}&\big(\widehat{\mu}_{t_{m}}^{K}, \bar{\mu}_{t_{m}}^{N}\big)\Big]\leq\mathbb{E}\,\left[\mathcal{W}_{2}\Big(\frac{1}{N}\sum_{n=1}^{N}\delta_{\widetilde{X}_{t_{m}}^{n, N}},\,\bar{\mu}_{t_{m}}^{N}\Big)\right]+\mathbb{E}\,\left[\mathcal{W}_{2}\Big(\frac{1}{N}\sum_{n=1}^{N}\delta_{\widetilde{X}_{t_{m}}^{n, N}},\,\widehat{\mu}_{t_{m}}^{K}\Big)\right]\nonumber\\
&\leq\bar{C}_{2}\sum_{j=0}^{m-1}\bar{C}_{1}^{j}\sqrt{\mathbb{E}\,\hat{\Xi}^{2}_{m-1-j}}+\mathbb{E}\,\hat{\Xi}_{m}.\nonumber\qedhere
\end{align}
\end{proof}

\begin{rem}
For every $m=1, ..., M$, it follows from (\ref{eq:diff-repre}) that
 \[\mathbb{E}\,\hat{\Xi}_{m}=\mathbb{E}\,\left[\mathcal{W}_{2}\Big(\widehat{\mu}_{t_{m}}^{K}, \frac{1}{N}\sum_{n=1}^{N}\delta_{\widetilde{X}_{t_{m}}^{n, N}}\Big)\right]=\mathbb{E}\,\left[e_{K, \frac{1}{N}\sum_{n=1}^{N}\delta_{\widetilde{X}_{t_{m}}^{n, N}}}(x^{(m)})\right].\] Thus one can implement Lloyd's algorithm at each time step in order to reduce the error bound on the right-hand side of (\ref{quapar}). Moreover, for a fixed $\omega\in\Omega$, finding optimal quantizer for the empirical measure $\frac{1}{N}\sum_{n=1}^{N}\delta_{\widetilde{X}_{t_{m}}^{n, N}(\omega)} $ is equivalent to compute the $K$-means cluster centers of the sampling $\big(\widetilde{X}_{t_{m}}^{1, N}(\omega), ..., \widetilde{X}_{t_{m}}^{N, N}(\omega)\big)$ for which we can use e.g. \texttt{sklearn.cluster.KMeans} package in \texttt{Python}. 
\end{rem}

\section{Simulation examples}\label{example}

In this section, we illustrate two simulation examples. The first one is the Burgers equation introduced and already considered for numerical tests in \cite{bossy1997stochastic} and \cite{gobet2005discretization}. This is a one-dimensional example with an explicit solution so we use this example to compare the accuracy and computational time of the different numerical methods under consideration. The second example, the network of FitzHugh-Nagumo neurons, already numerically investigated in \cite{MR2974499}, \cite{bossy2015clarification} (see also  \cite{reis2018simulation}), is a 3-dimensional example. All examples are implemented in \texttt{Python 3.7}.

\subsection{Simulation of the Burgers equation, comparison of three algorithms}

In \cite{bossy1997stochastic}, the authors analyse the solution and study the particle method of the Burgers equation
\begin{equation}\label{burger}
\begin{cases}
dX_{t}=\int_{\mathbb{R}}H({X_{t}}-y)\mu_{t}(dy)dt+\sigma dB_{t},\\
\forall \,t\in[0, T], \;\mu_{t}=P_{X_{t}} \text{ and }X_{0}: (\Omega, \mathcal{F},\mathbb{P})\rightarrow (\mathbb{R}, \mathcal{B}(\mathbb{R})),
\end{cases}
\end{equation}
where $H$ is the Heaviside function $H(z)=\mathbbm{1}_{\{z\geq 0\}}$ and $\sigma$ is a real constant. 
 The cumulative distribution function  $V(t, x)$ of $\mu_{t}$ satisfies 
 \vspace{-0.2cm}
\[\frac{\partial V}{\partial t}=\frac{1}{2}\sigma^{2}\frac{\partial^{2}V}{\partial x^{2}}-V\frac{\partial V}{\partial x}, \qquad
V(0, x)=V_{0}(x).\]
Moreover, since the initial cumulative distribution function $V_{0}$ satisfies $\int_{0}^{x}V_{0}(y)dy=\mathcal{O}(x)$, the function $V$ has a closed form given by (see \cite{MR0047234})
\vspace{-0.2cm}
\[V(t,x)=\frac{\int_{\mathbb{R}} V_{0}(y)\mathrm{exp}\Big(-\frac{1}{\sigma^{2}}\big[\frac{(x-y)^{2}}{2t}+\int_{0}^{y}V_{0}(z)dz\big]\Big)dy}{\int_{\mathbb{R}}\mathrm{exp}\Big(-\frac{1}{\sigma^{2}}\big[\frac{(x-y)^{2}}{2t}+\int_{0}^{y}V_{0}(z)dz\big]\Big)dy},\qquad(t,x)\in[0, T]\times \mathbb{R}.\]
Hence, if we consider $X_{0}=0$, then the cumulative distribution function at time $T=1$ reads 
\vspace{-0.1cm}
\begin{equation}\label{truecdf}
F_{T=1}(x)=\frac{\int_{\mathbb{R}_{+}}\mathrm{exp}\Big(-\frac{1}{\sigma^{2}}
\big[\frac{(x-y)^{2}}{2}+y\big]\Big)dy}{\int_{\mathbb{R}}\mathrm{exp}\Big(-\frac{1}{\sigma^{2}}
\big[\frac{(x-y)^{2}}{2}+y\mathbbm{1}_{y\geq0}\big]\Big)dy}.
\end{equation}
So we can  compare the accuracy of the different numerical methods proposed in the former sections by considering 
\begin{equation}\label{errorFsimuFtrue}
\vertii{F_{\mathrm{simu}}-F_{\mathrm{true}} }_{\sup},
\end{equation}
where $F_{\mathrm{simu}}$ represents the simulated cumulative distribution function by different numerical methods and $F_{\mathrm{true}}$ is the true cumulative distribution function (\ref{truecdf}).  This term \eqref{errorFsimuFtrue} is approximated by 
\begin{equation}\label{errorFsimuFtrue2}
\vertii{F_{\mathrm{simu}}(x)-F_{\mathrm{true}}(x)}_{\sup}\simeq \sup_{x\in \texttt{Unifset}}\left|F_{\mathrm{simu}}(x)-F_{\mathrm{true}}(x)\right|,
\end{equation}
where $\texttt{Unifset}$ is a uniformly spaced point set in $[-2.5, 3.5]$.

In the following simulation, we choose $\sigma^{2}=0.2$ and $M=50$ so that we have  the same time step $ h= \frac{T}{M}=0.02$ for each method.  We give a detailed comparison of the different methods in Table \ref{table}. Remark that the particle method \eqref{Deq} and the hybrid particle-quantization scheme \eqref{Feq} are random algorithms. Hence, their accuracy are computed by taking an average error computed over 50 independent simulations.

In a second phase, we show the converging rate of the simulation error (\ref{errorFsimuFtrue}) of the particle method \eqref{Deq} and the recursive quantization scheme without Lloyd quantizer optimization \eqref{Eeq}-\eqref{Geq}-\eqref{Geq2} respectively according to $N$ and $K$.  
 Table \ref{errortableN} shows  the simulation error \eqref{errorFsimuFtrue} of the particle method  \eqref{Deq} with respect to 
the numbers of particle $N=2^{8}, ..., 2^{13}$ for a fixed $M=50$. As the particle method  \eqref{Deq} is a random algorithm, the error shows in Table \ref{errortableN} is computed by rerunning independently 500 times for each value of $N$. 
\noindent  Figure \ref{logerrorpar} and \ref{errorstddevpar} show the log error  of the particle method with respect to $\log_2(N)$ and the standard deviation of the error. 

Table \ref{errortableK} shows the convergence rate of the simulation error (\ref{errorFsimuFtrue})  of the recursive quantization scheme  \eqref{Eeq}-\eqref{Geq} with respect to the quantizer size $K=2^{5}, 2^{6}, 2^{7}, 2^{8}, 2^{9}, 2^{10}$. Here we use a fixed quantizer sequence which is a uniformly spaced point set in [-2.5, 3.5]  without Lloyd's algorithm for each time step $t_m$.  Figure \ref{55b5} show the log error of the  recursive quantization scheme with respect to $\log_2(K)$.

The slope in Figure \ref{logerrorpar} is approximately -0.285, and in Figure \ref{55b5}, it's approximately -0.436. These slopes deviate results in Theorem \ref{thm1} and  \ref{thm:quadbasedscheme} since we evaluate the simulation error using $\vertii{F_{\mathrm{simu}}(x)-F_{\mathrm{true}}(x)}_{\sup}$ instead of the Wasserstein distance or the $L^2$-norm. We opt for this error measure due to the computational cost of the Wasserstein distance: for two probability distributions $\mu, \nu\in\mathcal{P}_{p}(\RD)$ with respective cumulative distribution functions $F$ and $G$, the Wasserstein distance $\mathcal{W}_{p}(\mu, \nu)$ can be computed by 
\vspace{-0.1cm}
\begin{equation}\label{wasscumu}
\mathcal{W}_{p}^{p}(\mu, \nu)=\int_{0}^{1}\left|F^{-1}(u)-G^{-1}(u)\right|^{p}du, \;p\geq1.\vspace{0.1cm}
\end{equation}
The computation cost arises from the inverse functions of the cumulative distribution functions. Moreover, if we compute (\ref{wasscumu}) by using Monte-Carlo simulation, it will induce its own statistical error which may perturb our comparisons.  
Additionally,  for a fair method comparison between the particle method \eqref{Deq} and the recursive quantization scheme \eqref{Eeq}-\eqref{Geq}-\eqref{Geq2}, consistency in error measurement is necessary.

\small

\begin{spacing}{0.9}

\begin{table}[H]
\centering
\begin{tabular}{|l|l|c|c|}
\cline{1-4}
  &\tabincell{l}{\\Particle size and \\ quantizer size\\\;}&\tabincell{l}{\\Computing  \\ time for each \\Euler step\\\;} & 
 \tabincell{l}{\\ Error \\$\vertii{F_{\mathrm{simu}}-F_{\mathrm{true}} }_{\sup}$\\\;}  \\
 \cline{1-4}
\tabincell{l}{\\Particle method \eqref{Deq}\\ \;}& Particle size $N=10000$
 & 0.00320s & 0.01021\\
 
\cline{1-4}
\tabincell{l}{\\Recursive quantization scheme \\without Lloyd iterations \\\eqref{Eeq}-\eqref{Geq2}\\\;} & Quantizer size $K=500$ &  0.00205s & 0.01054\\
\cline{1-4}
\tabincell{l}{\\Recursive quantization scheme \\with 5 Lloyd iterations \\ at each Euler step\\ \eqref{Eeq}-\eqref{Geq2} and \eqref{pm1given}-\eqref{whyslow}\\\;} &Quantizer size $K=500$& 8.21598s  & 0.01029 \\
\cline{1-4}
\tabincell{l}{\\Hybrid particle-quantization \\ scheme \eqref{Feq} without \\Lloyd iterations \\\;} &\tabincell{l}{\\Particle size $N=10000$\\  Quantizer size $K=500$\\\;}& 6.09480s & 0.01626\\
\cline{1-4}
\tabincell{l}{\\Hybrid particle-quantization \\ scheme \eqref{Feq} with 5 Lloyd  \\ iterations at each Euler step \\\;} &\tabincell{l}{\\Particle size $N=10000$ \\  Quantizer size $K=500$\\\;}& 9.37229s & 0.01013\\
\cline{1-4}
\end{tabular}
\vspace{0.1cm}
\caption{Comparison of three algorithms}\label{table}
\end{table}
\end{spacing}

\vspace{-0.3cm}

\normalsize
\vspace{-0.2cm}

\begin{table}[H]
\begin{spacing}{1.35}
\begin{center}
\begin{tabular}{|c|c|c|c|c|c|c|}
\hline  
N& $2^{8}$ & $2^{9}$ & $2^{10}$ & $2^{11}$ & $ 2^{12}$ & $2^{13}$\\
\hline  
Error $\vertii{F_{\mathrm{simu}}-F_{\mathrm{true}} }_{\sup}$& 0.04691 & 0.03409 & 0.02438 & 0.01785& 0.01407& 0.01131 \\
\hline
Standard deviation &  0.01207 & 0.00939 & 0.00687 & 0.00469 & 0.00408 & 0.00294 \\
\hline
\end{tabular}
\end{center}
\end{spacing}
\caption{Error of the particle method  \eqref{Deq} with respect to the particle size $N$\label{errortableN}}
\end{table}

\vspace{-0.5cm}

\begin{table}[H]
\begin{spacing}{1.35}
\begin{center}
\begin{tabular}{|c|c|c|c|c|c|c|}
\hline  
K& $2^{5}$ & $2^{6}$ & $2^{7}$ & $2^{8}$ & $ 2^{9}$ & $2^{10}$\\
\hline  
Error $\vertii{F_{\mathrm{simu}}-F_{\mathrm{true}} }_{\sup}$& 0.07347 & 0.04176& 0.02360& 0.01471& 0.01043& 0.00829\\
\hline
\end{tabular}
\end{center}
\end{spacing}
\caption{Error of the recursive quantization scheme  \eqref{Eeq}-\eqref{Geq} with respect to the quantizer size $K$\label{errortableK}}
\end{table}

Now we give some short comments on the numerical performance of different methods mainly with regards to two aspects: the accuracy and the computing time. 
\begin{itemize}
\item Comparison of the computing time.

The application of  Lloyd's iteration in the recursive quantization scheme is slightly faster than in the hybrid particle-quantization scheme since we used the closed formulas  (\ref{alterformula}). We can also remark that, when the quantization level $K$ is large,  Lloyd's iteration is numerically costly even in dimension 1 and the application of  Lloyd's algorithm only slightly reduces the quantization error. 
\vspace{0.2cm}
\item Comparison of the accuracy computed by $\vertii{F_{\mathrm{simu}}(x)-F_{\mathrm{true}}(x)}_{\sup}$.

\begin{itemize}

\item Comparing with the particle size $N$ in Table \ref{errortableN} and the quantizer size $K$ in Table \ref{errortableK}, one can remark that to achieve the same accuracy,  we need fewer points in the quantizer than in the particle. 

\item The error of the recursive quantization scheme \eqref{Eeq}-\eqref{Geq}-\eqref{Geq2}, especially when we implement without  Lloyd's quantizer optimization, strongly depends on the choice of the quantizer. Generally, a practical way to choose the initial quantizer of a probability distribution $\mu$ is to use the self-quantization technique for which we refer to \cite{delattre2006quantization}, \cite[Section 7.1 and Section 14]{graf2000foundations}, \cite{pages2003optimal} and \cite{pages2004optimal}. 
Another efficient trick  is to rely on a so-called ``splitting method'', that is, using  
the trained quantizer of the Euler step $t_m$ as an initial quantizer of the Euler step $t_{m+1}$.
\end{itemize}
In this one-dimensional example, we did not remark any obvious advantage of the hybrid particle quantization scheme \eqref{Feq} compared with other methods. 
\end{itemize}


\begin{figure}[H]
\centering
\begin{minipage}{0.4\textwidth}
\includegraphics[height=6cm, width=6cm]{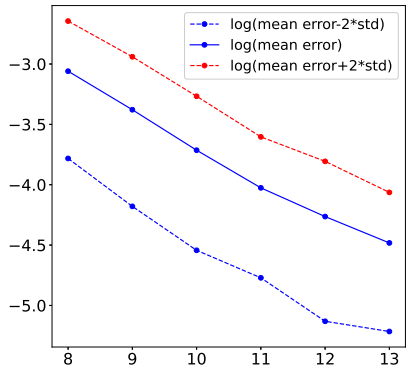}
\caption{Log-error of the particle method \eqref{Deq} with respect to $\log_{2}(N)$. \\\:\\\: }\label{logerrorpar}
\end{minipage}

\end{figure}

\begin{figure}[H]\centering
\begin{minipage}{0.34\textwidth}
\includegraphics[height=6cm, width=6cm]{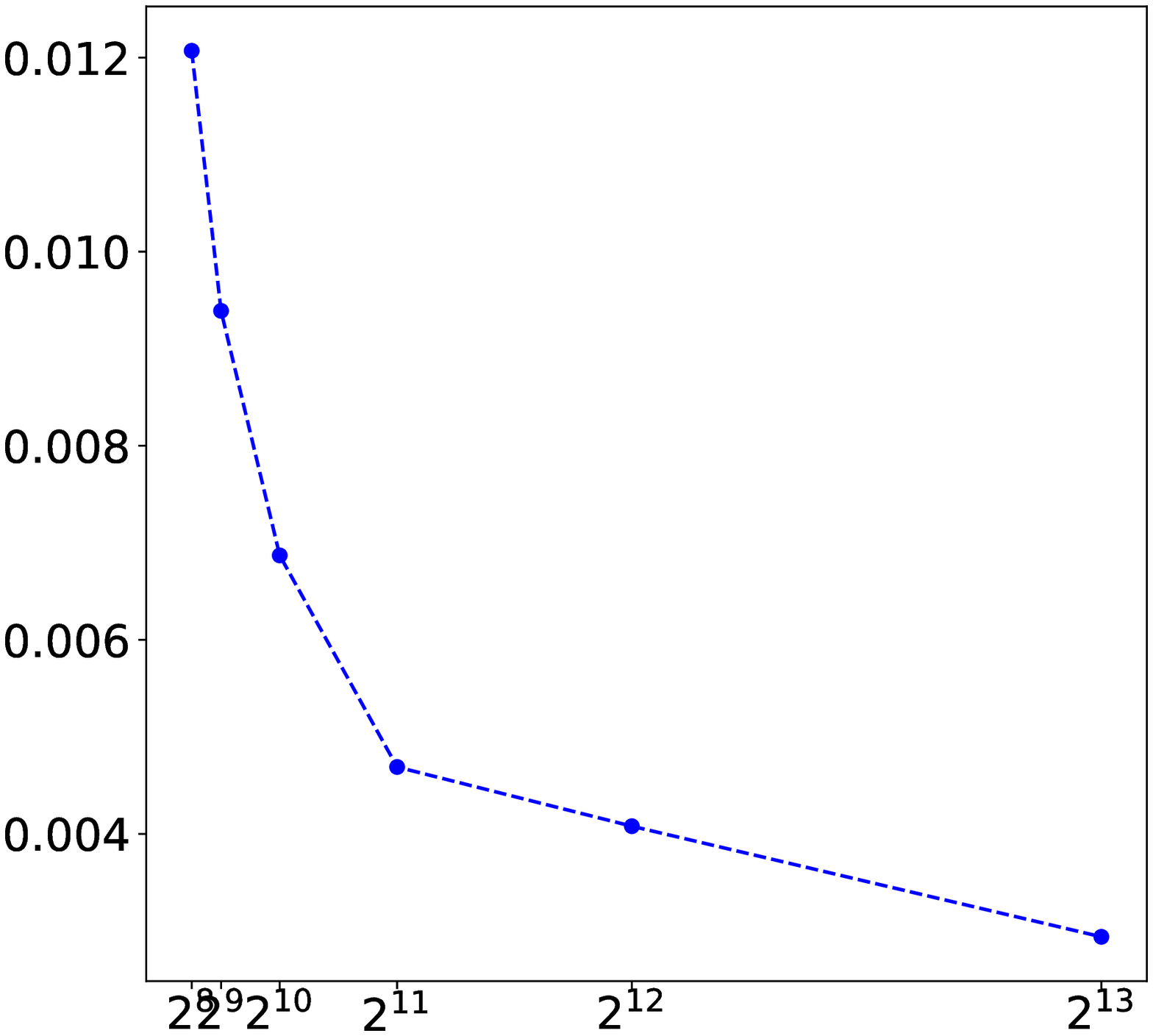}
\caption{Standard deviation of the error of the particle method.  The horizontal axis is $\log_{2}(N)$. }\label{errorstddevpar}
\end{minipage}\hspace{1cm}
\begin{minipage}{0.34\textwidth}
\includegraphics[height=6cm, width=6cm]{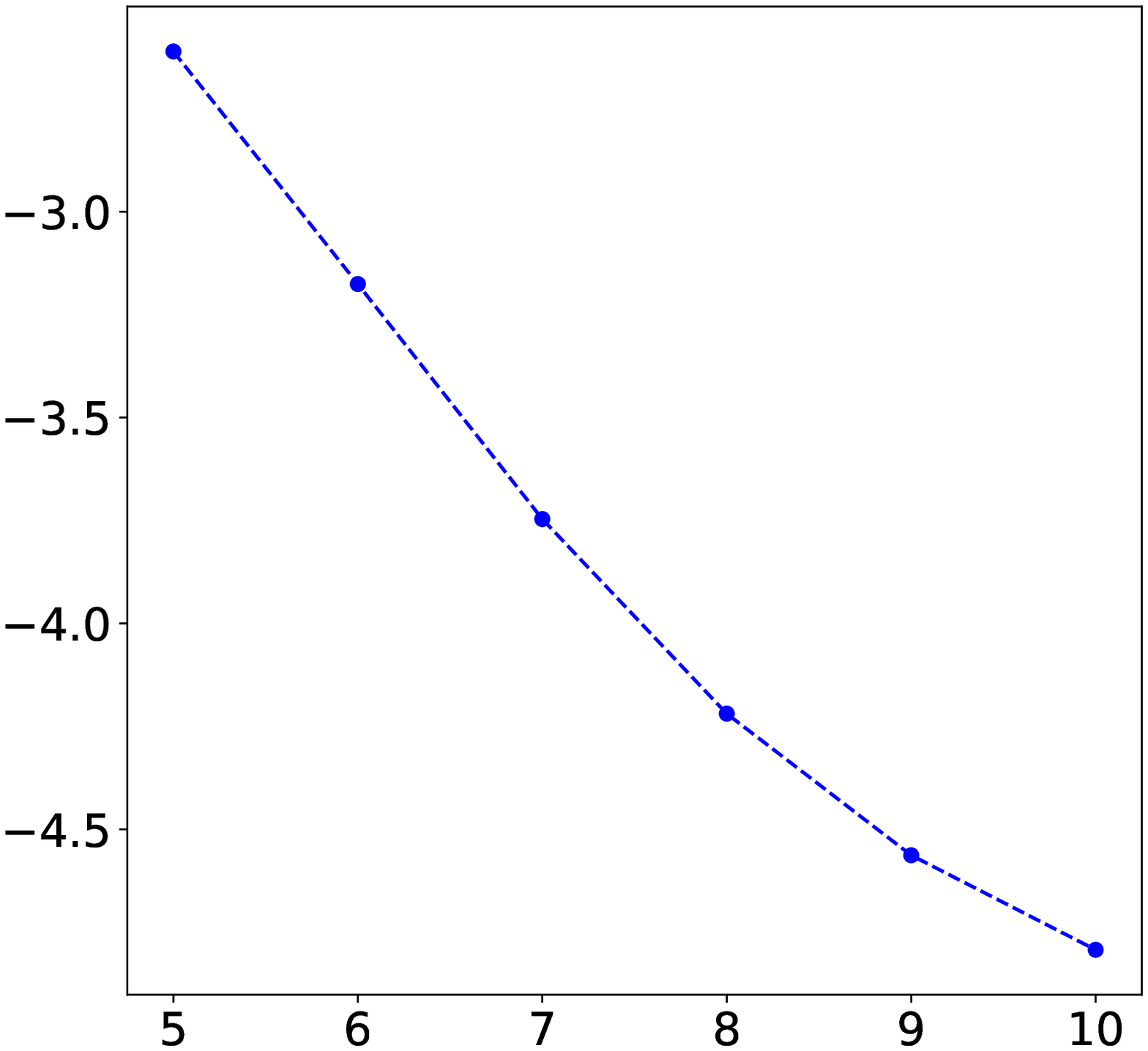}
\caption{Log-error of the recursive quantization scheme \eqref{Eeq}-\eqref{Geq}-\eqref{Geq2} with respect to $\log_{2}(K)$.}\label{55b5}
\end{minipage}

\end{figure}

\smallskip

\subsection{Simulation of the network of FitzHugh-Nagumo neurons in dimension 3}\label{simu2section}

In this section, we consider the network of FitzHugh-Nagumo neurons\footnote{ In this paper, we exclusively employ this model as a simulation illustration. For a comprehensive understanding of the network of FitzHugh-Nagumo neurons, we recommend referring to \cite{MR2974499} and \cite{bossy2015clarification} for more details. } introduced in \cite{MR2974499} and \cite{bossy2015clarification}:\vspace{-0.1cm}
\begin{equation}\label{FitzHughNagumo}
dX_{t}=b(X_{t}, \mu_{t})dt+\sigma(X_{t}, \mu_{t})dB_{t}\vspace{-0.1cm}
\end{equation}
where $b:\RR^{3}\times\mathcal{P}(\RR^{3})\rightarrow\RR^{3}$ and $\sigma:\RR^{3}\times\mathcal{P}(\RR^{3})\rightarrow \mathbb{M}_{3\times3}$ are defined by \vspace{-0.1cm}
\begin{equation}\label{drift12}
b(x, \mu)\coloneqq\left(\begin{array}{c}x_{1}-(x_{1})^{3}/3-x_{2}+I-\int_{\RR^{3}}J(x_{1}-V_{rev})z_{3}\,\mu(dz)\\
c(x_{1}+a-bx_{2})\\
a_{r}\frac{T_{\max}(1-x_{3})}{1+\exp\big(-\lambda(x_{1}-V_{T})\big)}-a_{d}x_{3}\end{array}\right),
\end{equation}
\[\sigma(x, \mu)\coloneqq\left(\begin{array}{ccc}
\sigma_{ext}&0&-\int_{\RR^{3}}\sigma_{J}(x_{1}-V_{rev})z_{3}\,\mu(dz)\\
0&0&0\\
0&\sigma_{32}(x)&0
\end{array}\right),\]
with
\[\sigma_{32}\coloneqq \mathbbm{1}_{x_{3}\in(0, 1)}\sqrt{a_{r}\frac{T_{\max}(1-x_{3})}{1+\exp\big(-\lambda(x_{1}-V_{T})\big)}+a_{d}\,x_{3}}\,\Gamma\exp\big(-\frac{\Lambda}{1-(2x_{3}-1)^{2}}\big),\] 
and    where the initial random variable $X_0$ follows a gaussian distribution \vspace{-0.1cm}
\[X_{0}\sim \mathcal{N}\left(\left(\begin{array}{c} V_{0}\\ \omega_{0}\\ y_{0}\end{array}\right),\left( \begin{array}{ccc}\sigma_{V_{0}}&0&0\\0&\sigma_{\omega_{0}}&0\\ 0&0&\sigma_{y_{0}}\end{array}\right)\right)\vspace{-0.2cm}\]
with the following parameter values \vspace{-0.2cm}
\[\begin{array}{lllllll}
V_{0}=0& \sigma_{V_{0}}=0.4& a=0.7&b=0.8&c=0.08&I=0.5&\sigma_{ext}=0.5\\
\omega_{0}=0.5& \sigma_{\omega_{0}}=0.4 &V_{rev}=1&a_{r}=1&a_{d}=1&T_{\max}=1&\lambda=0.2\\
y_{0}=0.3& \sigma_{y_{0}}=0.05&J=1&\sigma_{J}=0.2&V_{T}=2&\Gamma=0.1&\Lambda=0.5.
\end{array} \vspace{-0.1cm}\]

This model was first studied in~\cite{MR2974499} and then rigorously investigated and analyzed in~\cite{bossy2015clarification}. We are aware that the equation \eqref{FitzHughNagumo} does not fulfil Assumption \ref{AssumptionI} since the drift $b$ defined by~\eqref{drift12} is only locally Lipschitz in $(x, \mu)$. However, the drift  $b$ satisfies \vspace{-0.2cm}
\[
\big\langle\, x\,,\, b(x, \mu)\,\big\rangle \le c_1 \big(|x|^2 +1\big)\vspace{-0.2cm}
\]
for some $c_1>0$ and  the coefficient $\sigma$ is bounded. This  ensures the existence of a strong solution living on the whole $\mathbb{R}_{+}$. In fact, the presence of the term $-\frac{(x_1)^3}{3}$  also induces a mean-reversion property that makes it possible to control the long-run behaviour of the equation. On the other hand, such drift with non-linear growth (in norm) usually induces  instability of the  Euler scheme as emphasized e.g. in~\cite{Lemaire2007}, at least when using evenly spaced time steps. We nevertheless chose this model for a 3D numerical illustration due to its challenging feature,  but with a refined time step ($M=5000$, see below) to ensure its stability. Other numerical schemes, typically semi-implicit could be more stable but are out of the scope of this paper.

In this section, we compare the performance of the particle method \eqref{Deq} and the hybrid particle-quantization scheme \eqref{Feq}  in two aspects. First, we intuitively compare these two methods by simulating the density function of $(x_{1}, x_{2})$ for $T=1.5$, as in the original paper \cite{MR2974499}[Page 31, Figure 4, the third one in the right]. In this step, we choose  $M=5000$ for the temporal discretization to reduce (as much as possible) the error induced by the Euler scheme.  The images of the density function simulated  by the particle method are displayed in Figure \ref{coorparticle3d}, \ref{density3dparticle} and the simulation by the hybrid method is displayed in Figure \ref{quantizer01}, \ref{quantizer02}, \ref{densityquan2}. 

Next, as the particle method and the hybrid method are both random methods, we take 
\[\varphi(\mu_{T}^{\mathrm{simu}})\coloneqq \int_{\mathbb{R}^{3}}\left|\xi\right|^{2}\mu_{T}^{\mathrm{simu}}(d\xi)=\mathbb{E}\Big[\left|X_{T}^{\mathrm{simu}}\right|^{2}\Big]\]
as a test function for the simulated distribution $\mu_{T}^{\mathrm{simu}}$ at time $T$, rerun 200 times for each method and compare the mean and the standard deviation of $\varphi(\mu_{T}^{\mathrm{simu}})$. 

The obtained density functions have a similar form by these two methods, but the data volume obtained by the particle method is 
\[\text{5000 (the number of particle)$\;\times$ 3 (dimension)},\] 
while the data volume obtained by the hybrid method is 
\[\text{300 (the quantizer size)$\;\times$ 4 (dimension $+$ weight for each quantizer).}\]
For a more precise comparison, we now set  the time discretization number $M=150$ and we consider the following test function for the simulated distribution $\mu_{T}^{\mathrm{simu}}$ at $T=1.5$
\[\varphi (\mu_{T}^{\mathrm{simu}})\coloneqq \int_{\mathbb{R}^{3}}\left|\xi\right|^{2}\mu_{T}^{\mathrm{simu}}(d\xi)=\mathbb{E}\Big[\,\big|X_{T}^{\mathrm{simu}}\big|^{2}\,\Big]\]
and rerun 200 times for each method. Table \ref{comp2} shows the simulation results.

Intuitively, the hybrid method can be considered as adding a ``feature extraction'' step on the particle method. Comparing the third and fourth columns of  Table \ref{comp2}, one can notice that this added step needs more computing time but  reduces the size of the output data  for the further computation of  the test function $\varphi (\mu_{T}^{\mathrm{simu}})$ without increasing the standard deviation. However, the second column of Table \ref{comp2} shows that if we implement the particle method with a similar data size than the quantization level, the test value of $\varphi (\mu_{T}^{\mathrm{simu}})$ provides a much larger standard deviation.

\begin{figure}[H]
\centering
\hspace{-1.2cm}\begin{minipage}[t]{0.5\linewidth}
\centering
\includegraphics[height=6.5cm,width=7cm]{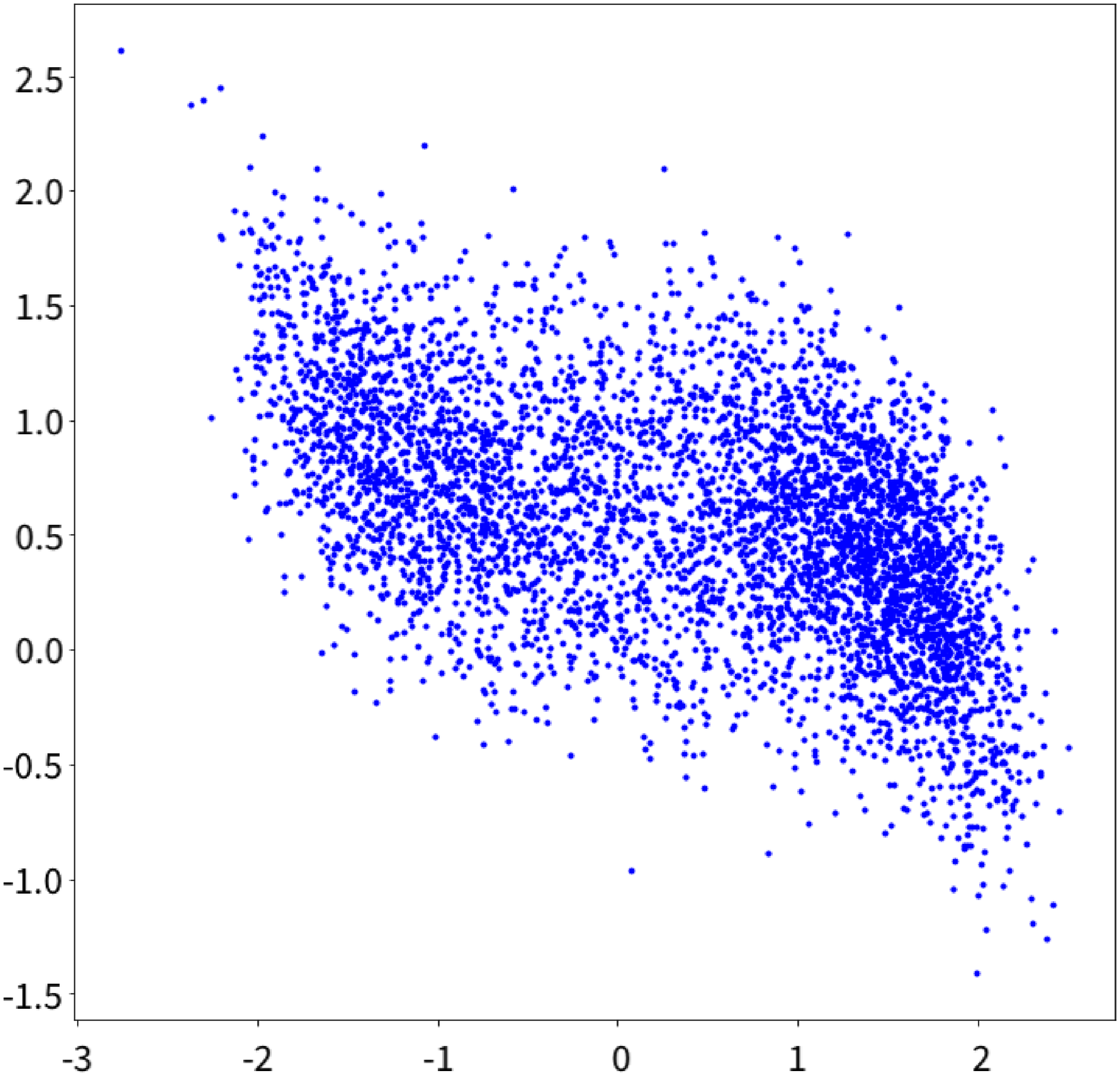}
\caption{The first and second coordinates of 5000 particles at time $T=1.5$.}\label{coorparticle3d}
\end{minipage}
\hspace{0.2cm}
\begin{minipage}[t]{0.5\linewidth}
\centering
\includegraphics[height=5.6cm ,width=9.5cm]{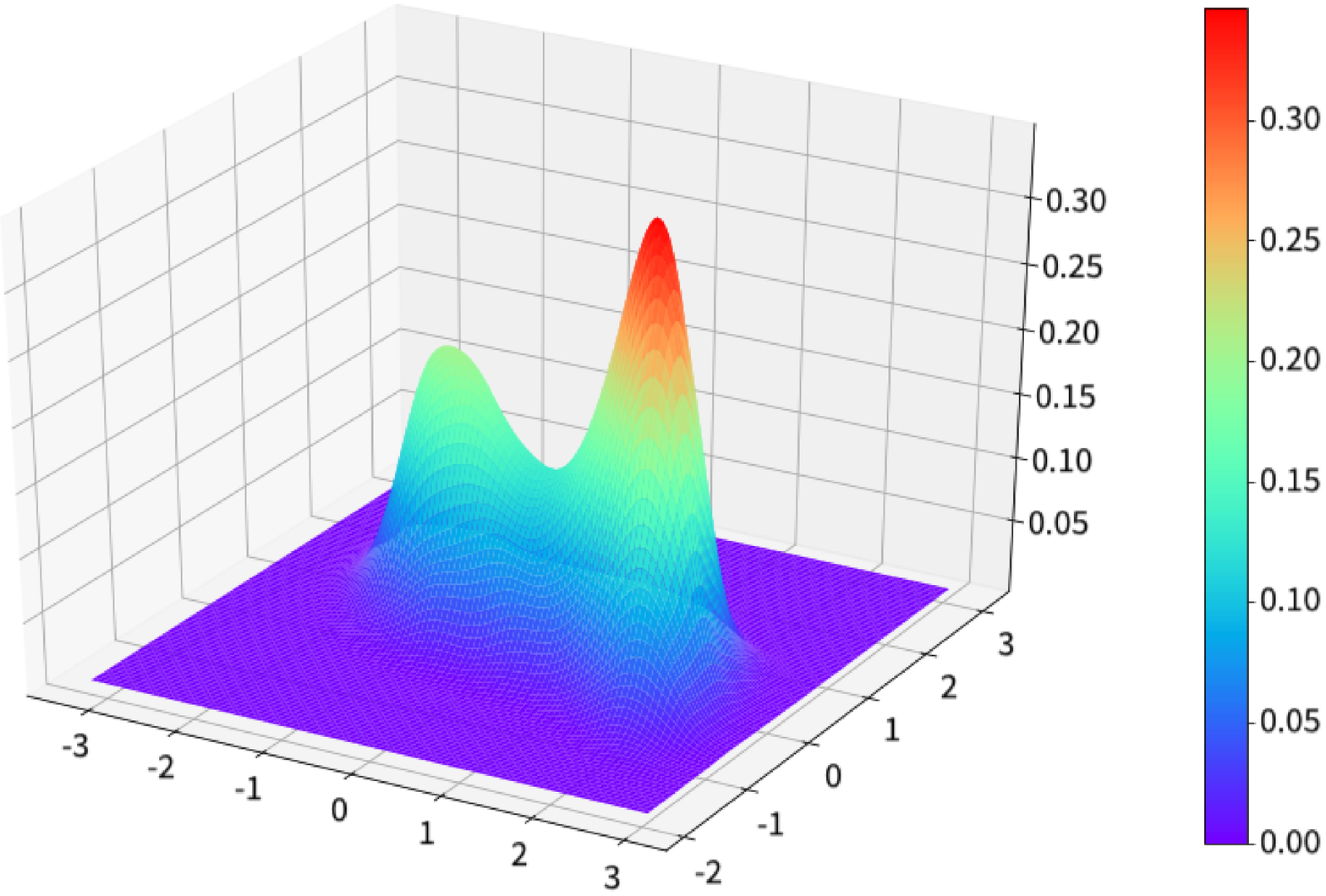}
\caption{The simulated density function smoothened  by the Gaussian kernel method (bandwith = 0.241).}\label{density3dparticle}
\end{minipage}
\end{figure}

\vspace{-0.5cm}
\begin{figure}[H]
\centering
\hspace{-1cm}\begin{minipage}[t]{0.5\textwidth}
\centering
\includegraphics[height=6.4cm ,width=7cm]{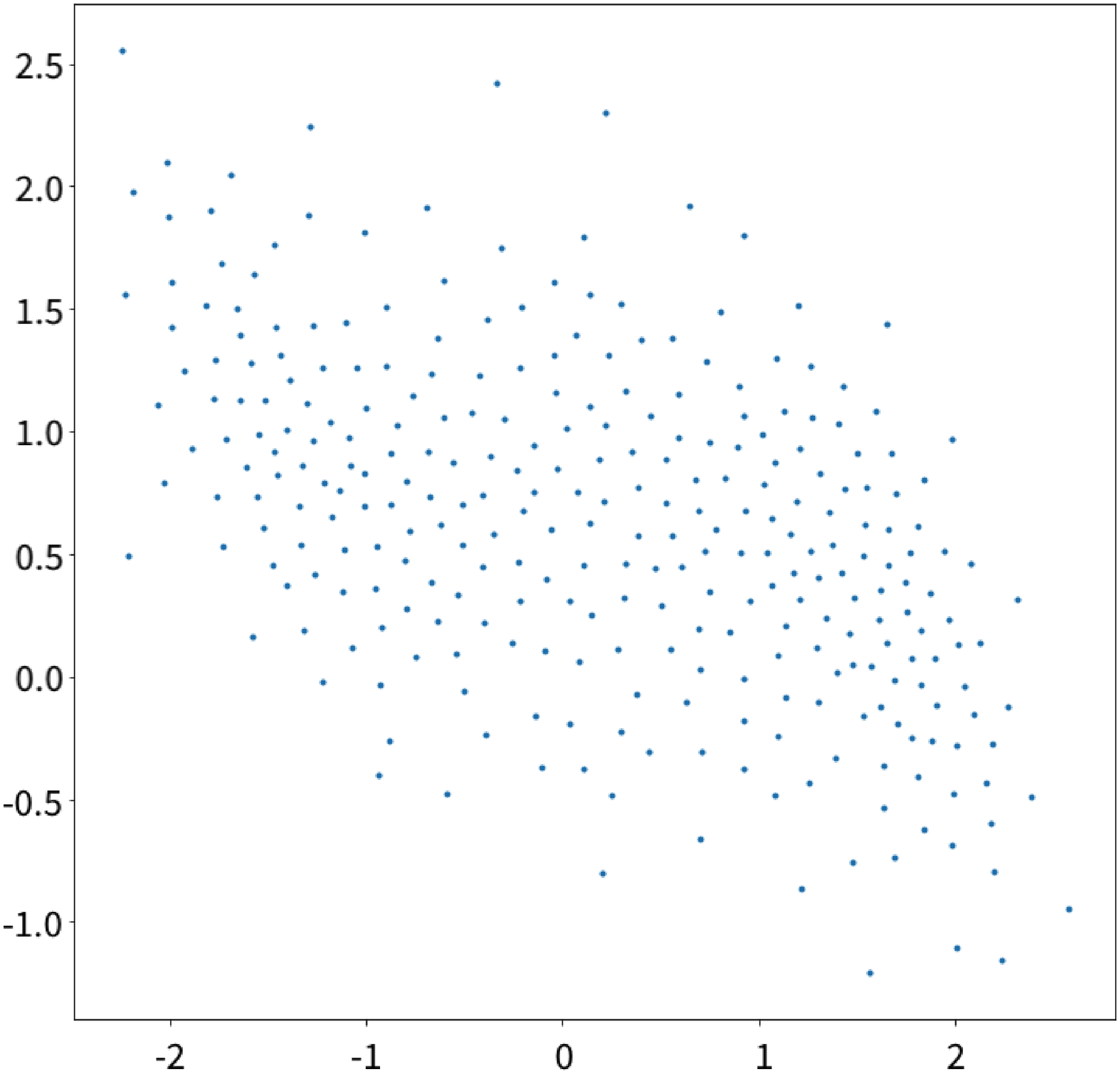}
\caption{The quantizer of $(x_{0}, x_{1})$, simulated with particle number $N=5000$, quantizer size $K=300$ and 10 Lloyd iterations at each Euler step.}\label{quantizer01}
\end{minipage}\hspace{0.2cm}
\begin{minipage}[t]{0.5\textwidth}
\centering
\includegraphics[height=6.4cm ,width=8cm]{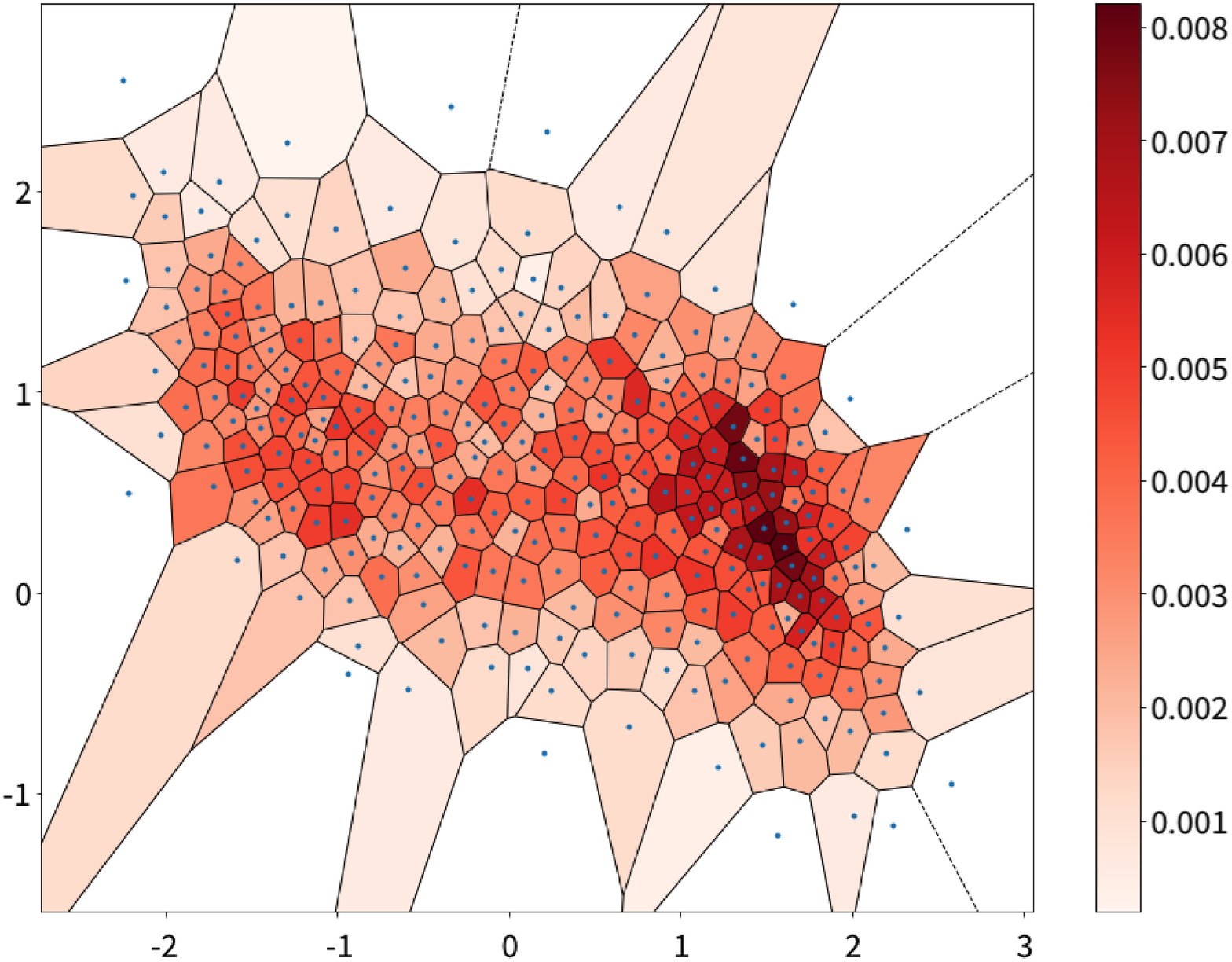}
\caption{The Vorono\"i cells of the above quantizer. The colour of each Vorono\"i cell represents the weight of this cell (the darker the heavier).}\label{quantizer02}
\end{minipage}
\end{figure}

\vspace{-0.5cm}

\begin{figure}[H]
\centering
\begin{minipage}[t]{0.65\textwidth}
\includegraphics[height=7cm ,width=10cm]{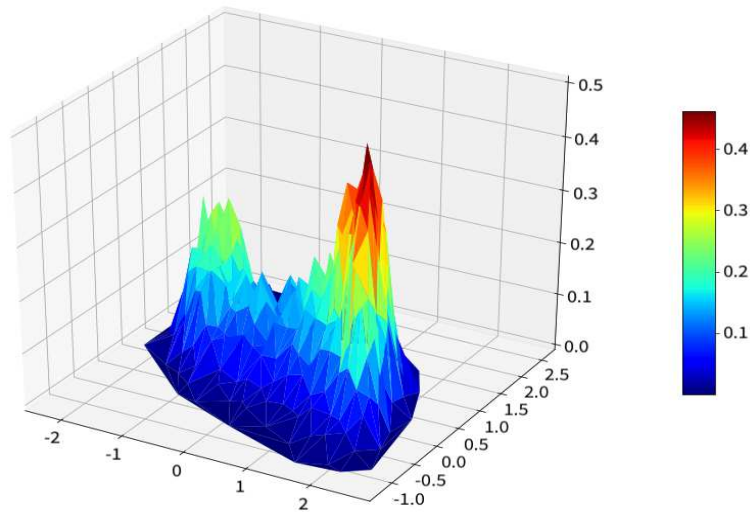}
\vspace{-0.2cm}
\caption{The density function simulated by the hybrid particle-quantization method \eqref{Feq}. The vertical axis is the weight divided by the area of the corresponding Vorono\"i cell.}\label{densityquan2}
\end{minipage}
\end{figure}

\begin{table}[H]
{\begin{spacing}{1.7}
\begin{center}
\begin{tabular}{ |c|c|c|c|c| } 
\hline
\multicolumn{2}{|c|}{} & \multicolumn{2}{c|}{Particle method } & Hybrid method  \\ 
 \cline{3-4}
\multicolumn{2}{|c|}{}  & $\quad$N=300$\quad$ & $\quad$N=5000$\quad$ & $N=5000$ and $K=300$  \\ 
\hline
\multicolumn{2}{|c|}{Data volume for $\mu_{T}^{\mathrm{simu}}$} & 300$\times$3   & 5000$\times$3 & 300$\times$4 \\ 
\hline
\multicolumn{2}{|c|}{\tabincell{l}{\vspace{-0.2cm}Average computing time\\  for each Euler step}} & 0.034s   & 0.608s & 2.670s \\ 
 \hline
\multirow{2}{*}{\tabincell{l}{\vspace{-0.5cm}\\\vspace{-0.2cm} Test function\\  $\varphi (\mu_{T}^{\mathrm{simu}})$}}  & Mean &1.194  &1.205 & 1.192\\ 
 \cline{2-5}
 &\tabincell{l}{\vspace{-0.2cm}Standard\\deviation }   & 0.058  &0.015 & 0.015\\
 \hline
\end{tabular}
\smallskip
\smallskip
\caption{$\quad$Comparison of the simulation result $\varphi (\mu_{T}^{\mathrm{simu}}).$\label{comp2}}
\end{center}
\end{spacing}}
\end{table}

\vspace{-1cm}

\appendix

\section{Proofs of Lemma \ref{relationd}, Lemma \ref{finitmomentXN} and Lemma \ref{lemthm1}}\label{appb}
\begin{proof}[Proof of Lemma \ref{relationd}]
We consider the canonical space $\Omega=\mathcal{C}([0,T], \mathbb{R}^{d})\times\mathcal{C}([0,T], \mathbb{R}^{d})$ equipped with the $\sigma$-algebra $\mathcal{F}$ generated by the distance $d\big((\omega^{1}, \omega^{2}), (\alpha^{1}, \alpha^{2})\big)\coloneqq\left\Vert\omega^{1}-\alpha^{1}\right\Vert_{\sup}\vee\left\Vert\omega^{2}-\alpha^{2}\right\Vert_{\sup}$ and $\mathbb{P}\in\Pi(\mu, \nu)$ where $\Pi(\mu, \nu)$ is the set of probability measures with marginals $\mu$ and $\nu$. For any $\omega=(\omega^{1}, \omega^{2})\in\Omega$, we define the canonical projections $X: \Omega\rightarrow\mathcal{C}([0,T], \mathbb{R}^{d})$ and $Y: \Omega\rightarrow\mathcal{C}([0,T], \mathbb{R}^{d})$ by
\[\forall \,\omega=(\omega^{1}, \omega^{2}),\quad\forall\, t\in[0, T], \hspace{0.5cm} X_{t}(\omega)=\omega_{t}^{1}\;\;\text{and}\;\; Y_{t}(\omega)=\omega_{t}^{2}.\]
The couple $(X, Y)$ makes up the canonical process on $\Omega$. Since $\mathbb{P}\in\Pi(\mu, \nu)$, then 
 $X$ has probability distribution $\mu$ and $Y$ has probability distribution $\nu$. Moreover, we have
\[\sup_{s\in[0,t]}\mathcal{W}_{p}^{p}(\mu_{s}, \nu_{s})\leq \sup_{s\in[0,t]}\mathbb{E}\left|X_{s}-Y_{s}\right|^{p}\leq\mathbb{E}\sup_{s\in[0,t]}\left|X_{s}-Y_{s}\right|^{p}.\]
Then we can choose by the usual arguments $\mathbb{P}\in\Pi(\mu, \nu)$ such that $\mathbb{E}\sup_{s\in[0,t]}\left|X_{s}-Y_{s}\right|^{p}=\big( \mathbb{W}_{p,t}(\mu_{s}, \nu_{s})\big)^{p}$ to conclude the proof. 
\end{proof}

\vspace{-0.3cm}
\begin{proof}[Proof of Lemma \ref{finitmomentXN}]
$(a)$ For any $x\in\mathbb{R}^{d}$ and for any $\mu\in\mathcal{P}_{p}(\mathbb{R}^{d})$, Assumption \ref{AssumptionI} implies that
\[\forall t\,\in[0, T],\;\;\left|b(t, x, \mu)\right|-\left|b(t, 0, \delta_{0})\right|\leq\left|b(t, x, \mu)-b(t, 0, \delta_{0})\right|\leq L\big(\left|x\right|+\mathcal{W}_{p}(\mu, \delta_{0})\big).\]
Hence, 
$\left|b(t, x, \mu)\right|\leq\left|b(t, 0, \delta_{0})\right| + L\big(\left|x\right|+\mathcal{W}_{p}(\mu, \delta_{0})\big)\leq (\left|b(t, 0, \delta_{0})\right| \vee L)(1+\left|x\right|+\mathcal{W}_{p}(\mu, \delta_{0}))$. 
Similarly, we have $\vertiii{\sigma(t, x, \mu)}\leq (\vertiii{\sigma(t, 0,  \delta_{0})} \vee L)(1+\left|x\right|+\mathcal{W}_{p}(\mu,  \delta_{0}))$, so we can take \[C_{b,\sigma, L, T}=\sup_{t\in[0, T]}\left|b(t, 0, \delta_{0})\right|\vee\sup_{t\in[0, T]}\vertiii{\sigma(t, 0, \delta_{0})}\vee L\] to complete the proof. 

\smallskip
\noindent $(b)$ {\sc Step 1.} We prove $\displaystyle\sup_{0\leq m\leq M} \sup_{1\leq n\leq N}\big\Vert \bar{X}_{t_{m}}^{n,N}\big\Vert_p <+\infty$ by induction in this step. First, we know that for every $n=1, ..., N$, $\big\Vert \bar{X}_{0}^{n,N}\big\Vert_p=\big\Vert X_0\big\Vert_p<+\infty$ by the definition of $\bar{X}_{0}^{n,N}$ so that $\sup_{1\leq n \leq N}\big\Vert \bar{X}_{0}^{n,N}\big\Vert_p<+\infty$. Now assume that $\sup_{1\leq n \leq N}\big\Vert \bar{X}_{t_m}^{n,N}\big\Vert_p<+\infty$. For every $n=1,..., N$, we have
\begin{align}
\Big\Vert \bar{X}_{t_{m+1}}^{n,N} \Big\Vert_p 
&\leq  \Big\Vert \bar{X}_{t_{m}}^{n,N} \Big\Vert_p+h\cdot \Big\Vert b\big(t_{m}, \bar{X}^{n, N}_{t_{m}},\bar{\mu}^{N}_{t_{m}}\big) \Big\Vert_p+\sqrt{h}\cdot \Big\Vert \;\vertiii{\sigma(t_{m}, \bar{X}^{n, N}_{t_{m}}, \bar{\mu}^{N}_{t_{m}})}\;\Big\Vert_p \Big\Vert Z_{m+1}^{n}\Big\Vert_p\nonumber
\end{align}
where the inequality above follows from the Minkowski inequality and  the fact that $Z_{m+1}^{n}$ is independent of $\mathcal{F}_{t_{m}}$ and $\sigma(t_{m}, \bar{X}^{n, N}_{t_{m}}, \bar{\mu}^{N}_{t_{m}})$ is $\mathcal{F}_{t_{m}}$ measurable. 
By applying the inequality \eqref{lgrowth} in the first part, we obtain 
\begin{align}
\Big\Vert b(t_{m}, \bar{X}^{n, N}_{t_{m}}, \bar{\mu}^{N}_{t_{m}}) \Big\Vert_p \vee  \Big\Vert \;\vertiii{\sigma(t_{m}, \bar{X}^{n, N}_{t_{m}}, \bar{\mu}^{N}_{t_{m}})}\;\Big\Vert_p \leq C_{b,\sigma, L, T}\Big(1+\Big\Vert  \bar{X}^{n, N}_{t_{m}} \Big\Vert _p+ \Big\Vert \mathcal{W}_{p}\big(\bar{\mu}^{N}_{t_{m}},\delta_{0}\big) \Big\Vert_p\Big)\nonumber
\end{align}
and by the definition of $\bar{\mu}^{N}_{t_{m}}$, we have 
\vspace{-0.2cm}
\begin{align}\label{wpsupinq}
\Big\Vert \;\mathcal{W}_{p}^{\,p}\big(\bar{\mu}^{N}_{t_{m}},\delta_{0}\big) \;\Big\Vert_p^p=\EE\Big[\mathcal{W}_{p}^{\,p}\big( \frac{1}{N}\sum_{n=1}^{N}\delta_{\bar{X}^{n, N}_{t_{m}}},\delta_{0}\big) \Big]\leq \frac{1}{N}\sum_{n=1}^{N}\EE\Big[ \big |\bar{X}^{n, N}_{t_{m}} \big|^{p}\Big]<+\infty.
\end{align}
Hence, for every $n=1, ..., N$, $\Big\Vert \bar{X}_{t_{m+1}}^{n,N} \Big\Vert_p<+\infty$. Then we can conclude $\displaystyle\sup_{0\leq m\leq M} \sup_{1\leq n\leq N}\big\Vert \bar{X}_{t_{m}}^{n,N}\big\Vert_p<+\infty$ by induction.

\smallskip
\noindent {\sc Step 2.} We prove \eqref{finitenorm} in this step. The  first step implies that 
\begin{equation}\label{supinq1}
\Big\Vert \sup_{0\leq m\leq M} \sup_{1\leq n\leq N} \big|\bar{X}_{t_{m}}^{n,N}\big|\:\Big\Vert_p^{p}\leq \sum_{m=0}^{M}\sum_{n=1}^{N}\EE\Big[ \big| \bar{X}_{t_{m}}^{n,N} \big|^{p}\Big]\leq (M+1)N \cdot \!\! \sup_{0\leq m\leq M} \sup_{1\leq n\leq N}\big\Vert \bar{X}_{t_{m}}^{n,N}\big\Vert_p^{p}<+\infty.
\end{equation}
Then,
\vspace{-0.2cm}
\begin{align}
&\left\Vert \;\sup_{t\in[0,T]} \Big| \bar{X}_{t}^{n,N}\Big| \;\right\Vert_p =\left\Vert \;\sup_{0\leq m\leq M} \sup_{t\in[t_m,t_{m+1}]} \Big| \bar{X}_{t}^{n,N}\Big| \;\right\Vert_p\nonumber\\
& \leq \left\Vert \;\sup_{0\leq m\leq M} \sup_{t\in[t_m,t_{m+1}]} \Big| \bar{X}_{t_m}^{n,N} + b\big(t_{m}, \bar{X}_{t_{m}}^{\,n, N}, \bar{\mu}_{t_{m}}^{N}\big)(t-t_{m}) +\sigma\big(t_{m}, \bar{X}_{t_{m}}^{\,n, N}, \bar{\mu}_{t_{m}}^{N}\big)(B_{t}^{n}-B_{t_{m}}^{n})\Big| \;\right\Vert_p \nonumber\\
& \leq \left\Vert  \;\sup_{0\leq m\leq M}  \Big| \bar{X}_{t_m}^{n,N} \Big| + h\Big|b\big(t_{m}, \bar{X}_{t_{m}}^{\,n, N}, \bar{\mu}_{t_{m}}^{N}\big)\Big|\:\:\right\Vert_p+ \left\Vert  \;\sup_{0\leq m\leq M} \vertiii{\sigma\big(t_{m}, \bar{X}_{t_{m}}^{\,n, N}, \bar{\mu}_{t_{m}}^{N}\big)} \sup_{t\in[t_m,t_{m+1}]} \Big|B_{t}^{n}-B_{t_{m}}^{n}\Big|\Big|  \right\Vert_p\nonumber\\
& \leq \left\Vert \sup_{0\leq m\leq M}  \Big| \bar{X}_{t_m}^{n,N} \Big| + h\Big|b\big(t_{m}, \bar{X}_{t_{m}}^{\,n, N}, \bar{\mu}_{t_{m}}^{N}\big)\Big|\right\Vert_p \!\!+\sum_{m=0}^{M} \left\Vert  \vertiii{\sigma\big(t_{m}, \bar{X}_{t_{m}}^{\,n, N}, \bar{\mu}_{t_{m}}^{N}\big)}  \right\Vert_p \!\left\Vert\sup_{t\in[t_m,t_{m+1}]} \Big|B_{t}^{n}-B_{t_{m}}^{n}\Big|  \right\Vert_p \!\!<+\infty \nonumber
\end{align}
by applying \eqref{lgrowth}, \eqref{wpsupinq} and \eqref{supinq1}.
\end{proof}

For the proof of Lemma \ref{lemthm1}, we need the following two inequalities and we refer to \cite{pages2018numerical}[Section 7.8] among other references for more details. 
\begin{lem}[The generalized Minkowski inequality]\label{gemin}
For any (bi-measurable) process $X=(X_{t})_{t\geq0}$, for every $p\in[1,\infty)$ and for every $ T\in[0, +\infty],$ 
$\big\Vert \int_{0}^{T}X_{t}dt\big\Vert_{p}\leq\int_{0}^{T}\left\Vert X_{t}\right\Vert_{p}dt.$
\end{lem}
\begin{lem}[Burk\"older-Davis-Gundy inequality (continuous time)]\label{BDGin}
For every $p\in(0, +\infty)$, there exists two real constants $c_{p}^{BDG}>0$ and $C_{p}^{BDG}>0$ such that, for every continuous local martingale $(X_{t})_{t\in[0, T]}$ null at 0, we have
$c_{p}^{BDG}\big\Vert \sqrt{\langle X \rangle_{T}}\big\Vert_{p}\leq\big\Vert \sup_{t\in[0,T]}\left|X_{t}\right|\big\Vert_{p}\leq C_{p}^{BDG}\big\Vert \sqrt{\langle X\rangle_{T}}\big\Vert_{p}.$
\end{lem}
\smallskip
In particular, if $(B_{t})$ is an $(\mathcal{F}_{t})$-standard Brownian motion and $(H_{t})_{t\geq0}$ is an $(\mathcal{F}_{t})$-progressively measurable process having values in $\mathbb{M}_{d, q}(\mathbb{R})$ such that $\int_{0}^{T}\left\Vert H_{t}\right\Vert^{2}dt<+\infty$ $\;\mathbb{P}-a.s.$, then the $d$-dimensional local martingale $\int_{0}^{\cdot}H_{s}dB_{s}$ satisfies
\vspace{-0.2cm}
\begin{equation}\label{BDGinequality}
\left\Vert\sup_{t\in[0, T]}\left|\int_{0}^{t}H_{s}dB_{s}\right|\right\Vert_{p}\leq C_{d,p}^{BDG}\left\Vert\sqrt{\int_{0}^{T}\left\Vert H_{t}\right\Vert^{2}dt}\right\Vert_{p}.
\end{equation}

\begin{proof}[Proof of Lemma \ref{lemthm1}]
For every \[\big(X, (\mu_{t})_{t\in[0, T]}\big), \big(Y, (\nu_{t})_{t\in[0, T]}\big)\in\LPC\times \mathcal{C}\big([0, T], \mathcal{P}_{p}(\mathbb{R}^{d})\big),\] for every $t\in[0, T]$, we have 
\begin{flalign}
&\left\Vert\sup_{s\in[0,t]}\left|\int_{0}^{t}\big[b(u, X_{u}, \mu_{u})-b(u, Y_{u}, \nu_{u})\big]du\right|\right\Vert_{p}\leq\int_{0}^{t}\vertii{b(u, X_{u}, \mu_{u})-b(u, Y_{u}, \nu_{u})}_{p}du\;\;\text{(by Lemma \ref{gemin})}&\nonumber\\
&\hspace{0.5cm}\leq\int_{0}^{t}\vertii{L\big[\left|X_{u}-Y_{u}\right|+\mathcal{W}_{p}(\mu_{u}, \nu_{u})\big]}_{p}du\leq L\int_{0}^{t}\big[\vertii{X_{u}-Y_{u}}_{p}+\vertii{\mathcal{W}_{p}(\mu_{u}, \nu_{u})}_{p}\big]du,&\nonumber
\end{flalign}
and
\vspace{-0.2cm}
\begin{align}
&\left\Vert\sup_{s\in[0, t]}\left|\int_{0}^{s}\big[\sigma(u, X_{u}, \mu_{u})-\sigma(u, Y_{u}, \nu_{u})\big]dB_{u}\right|\right\Vert_{p}&\nonumber\\
&\hspace{0.5cm}\leq C_{d, p}^{BDG}\vertii{\sqrt{\int_{0}^{t}\vertiii{\sigma(u, X_{u}, \mu_{u}) - \sigma(u, Y_{u}, \nu_{u})}^{2}du}}_{p}\quad\text{\small(by Lemma \ref{BDGin})}&\nonumber\\
&\hspace{0.5cm}\leq C_{d, p}^{BDG}\vertii{\int_{0}^{t}\vertiii{\sigma(u, X_{u}, \mu_{u}) - \sigma(u, Y_{u}, \nu_{u})}^{2}du}_{\frac{p}{2}}^{\frac{1}{2}}\quad\text{\small(since $\left\Vert\sqrt{U}\right\Vert_{p}=\big[\mathbb{E}U^{\frac{p}{2}}\big]^{\frac{2}{p}\times\frac{1}{2}}=\left\Vert U\right\Vert_{\frac{p}{2}}^{\frac{1}{2}}$, when $U\geq 0$)}&\nonumber\\
&\hspace{0.5cm}\leq C_{d, p}^{BDG}\Big[\int_{0}^{t}\vertii{\vertiii{\sigma(u, X_{u}, \mu_{u}) - \sigma(u, Y_{u}, \nu_{u})}^{2}}_{\frac{p}{2}}du\Big]^{\frac{1}{2}}&\nonumber\\
&\hspace{0.5cm}\leq C_{d, p}^{BDG}\Big[\int_{0}^{t}\big\Vert\vertiii{\sigma(u, X_{u}, \mu_{u}) - \sigma(u, Y_{u}, \nu_{u})}\big\Vert^{2}_{p}du\Big]^{\frac{1}{2}}\quad\text{\big(since $\left\Vert\left| U\right|^{2}\right\Vert_{\frac{p}{2}}=\Big[\big(\mathbb{E}\left| U\right|^{p}\big)^{\frac{1}{p}}\Big]^{2}=\left\Vert U\right\Vert_{p}^{2} $\big)}&\nonumber\\
&\hspace{0.5cm}\leq \sqrt{2}\,C_{d, p}^{BDG}L\Big[\int_{0}^{t}\big[\vertii{X_{u}-Y_{u}}_{p}^{2}+\vertii{\mathcal{W}_{p}(\mu_{u}, \nu_{u})}_{p}^{2}\big]du\Big]^{\frac{1}{2}}.&\nonumber
\end{align}
Then we can conclude the proof by setting $C_{d, p, L}=\sqrt{2}\,C_{d, p}^{BDG}L.$\end{proof}

\newpage
\section{Organisation of different numerical schemes and pseudo-code}\label{appa}

\vspace{0.1cm}

\tikzstyle{mybox} = [draw=black, fill=white, thick,
    rectangle, rounded corners,  inner sep=10pt, inner ysep=15pt]
\tikzstyle{mybox2} = [draw=black, fill=white, dashed, thick,
    rectangle, rounded corners,  inner sep=10pt, inner ysep=15pt]
\tikzstyle{fancytitle}=[draw=black, fill=white, rectangle, rounded corners, thick, text=black]
\tikzstyle{arrow} = [->,>=stealth]

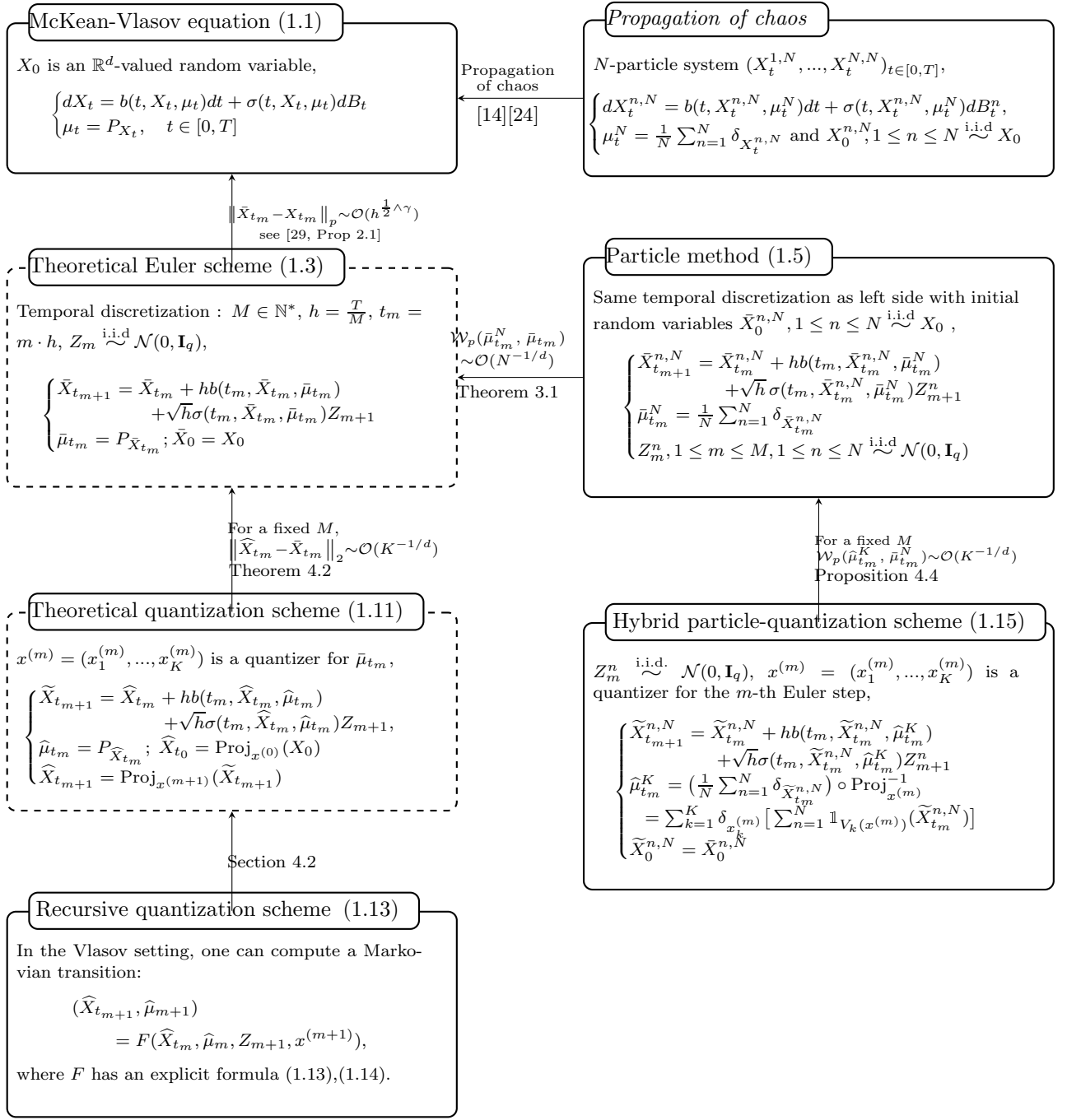
\begin{figure}[hbt!]
\centering

\begin{tikzpicture}

\hspace{-0.2cm}\node (MVequation) [mybox]
{\begin{minipage}{0.4\textwidth}
\footnotesize
    $X_{0}$ is an $\mathbb{R}^{d}$-valued random variable, 
         \[\begin{cases}
     dX_{t}=b(t, X_{t}, \mu_{t})dt +\sigma(t, X_{t}, \mu_{t})dB_{t}\\
        \mu_{t}=P_{X_{t}}, \quad t\in[0, T]
        \end{cases}\]
        \vspace{-0.2cm}
        \end{minipage}
};
\node[fancytitle, right=10pt] at (MVequation.north west) {\:McKean-Vlasov equation \eqref{Aeq} };

\node (Nparticle) [mybox, right = 2cm of MVequation]
{
    \begin{minipage}{0.42\textwidth}
    \footnotesize
$N$-particle system $(X_{t}^{1, N}, ..., X_{t}^{N, N})_{t\in[0,T]}$, 
\[\hspace{-0.1cm}\begin{cases}
dX^{n, N}_{t}=b(t, X^{n, N}_{t}, \mu^{N}_{t})dt+\sigma(t, X^{n, N}_{t}, \mu^{N}_{t})dB_{t}^{n}, \\
\mu^{N}_{t}=\frac{1}{N}\sum_{n=1}^{N}\delta_{X_{t}^{n, N}}\text{ and }  X_{0}^{n, N}\!\!\!, 1\leq n\leq N\widesim{\text{i.i.d}}X_{0}
\end{cases}\]
        
        \vspace{-0.4cm}
        \end{minipage}};
\node[fancytitle, right=10pt] at (Nparticle.north west) {\:\textit{Propagation of chaos}};

\node (theoricalscheme) [mybox2, below =1.5cm of MVequation]
{\begin{minipage}{0.4\textwidth}
\footnotesize
Temporal discretization : 
 $M\in\mathbb{N}^{*}$, $h=\frac{T}{M}$, $t_{m}=m\cdot h$, $Z_{m}\widesim{\text{i.i.d}}\mathcal{N}(0, \mathbf{I}_{q})$,
 
\vspace{-0.1cm}
\[\begin{cases}
\bar{X}_{t_{m+1}}=\bar{X}_{t_{m}}+hb(t_m, \bar{X}_{t_{m}},\bar{\mu}_{t_{m}}) \\
\hspace{1.5cm}+\sqrt{ h}\sigma(t_m, \bar{X}_{t_{m}}, \bar{\mu}_{t_{m}}) Z_{m+1}\\
\bar{\mu}_{t_{m}}=P_{\bar{X}_{t_{m}}} ; \bar{X}_{0}=X_{0}
\end{cases}\]
        
        \vspace{-0.2cm}
        \end{minipage}};
\node (theoricalschemetitle) [fancytitle, right=10pt] at (theoricalscheme.north west) {\:Theoretical Euler scheme \eqref{Ceq} };

\node (schemeNparticle) [mybox, right = 2cm of theoricalscheme]
{    \begin{minipage}{0.42\textwidth}
\footnotesize 

Same temporal discretization as left side with initial random variables
$\bar{X}^{n, N}_{0}, 1\leq n\leq N\widesim{\text{i.i.d}}X_{0}$ , 
\[\begin{cases}
\bar{X}^{n, N}_{t_{m+1}}=\bar{X}^{n, N}_{t_{m}}+hb(t_m, \bar{X}^{n, N}_{t_{m}},\bar{\mu}^{N}_{t_{m}}) \\
\hspace{1.4cm}+\sqrt{ h}\,\sigma(t_m, \bar{X}^{n, N}_{t_{m}}, \bar{\mu}^{N}_{t_{m}}) Z_{m+1}^{n}\\
\bar{\mu}^{N}_{t_{m}}=\frac{1}{N}\sum_{n=1}^{N}\delta_{\bar{X}^{n, N}_{t_{m}}}\\ 
Z_{m}^{n}, 1\leq m\leq M, 1\leq n\leq N\widesim{\text{i.i.d}}\mathcal{N}(0, \mathbf{I}_{q})
\end{cases}\]
        
        \vspace{-0.2cm}
        \end{minipage}};
\node[fancytitle, right=10pt] at (schemeNparticle.north west) {\:Particle method \eqref{Deq} };

\node (quantization) [mybox2, below  = 2cm of theoricalscheme]
{ \begin{minipage}{0.4\textwidth}
\footnotesize
$x^{(m)}=(x_{1}^{(m)}, ..., x_{K}^{(m)})$ is a quantizer for $\bar{\mu}_{t_{m}}$, 
\[\begin{cases}
\widetilde{X}_{t_{m+1}}=\widehat{X}_{t_{m}}+h b(t_m, \widehat{X}_{t_{m}},\widehat{\mu}_{t_{m}}) \\
\hspace{2cm}+\sqrt{ h}\sigma(t_m, \widehat{X}_{t_{m}}, \widehat{\mu}_{t_{m}}) Z_{m+1},\\
\widehat{\mu}_{t_{m}}=P_{\widehat{X}_{t_{m}}}; \;\widehat{X}_{t_{0}}=\Proj_{x^{(0)}}(X_0)\\
\widehat{X}_{t_{m+1}}=\mathrm{Proj}_{x^{(m+1)}}(\widetilde{X}_{t_{m+1}})
\end{cases}\]

        \vspace{-0.4cm}
\end{minipage}
    };
\node[fancytitle, right=10pt] at (quantization.north west) {\:Theoretical quantization scheme \eqref{Eeq} };

\node (Nparticlequantization) [mybox, below = 2cm of schemeNparticle]
{\begin{minipage}{0.42\textwidth}
 \footnotesize
$Z_{m}^{n}\widesim{\text{i.i.d.}}\mathcal{N}(0, \mathbf{I}_{q})$, $x^{(m)}=(x^{(m)}_{1}, ..., x^{(m)}_{K})$ is a quantizer for the $m$-th Euler step,\\
\[\begin{cases}
\widetilde{X}^{n, N}_{t_{m+1}}=\widetilde{X}^{n, N}_{t_{m}}+hb(t_m, \widetilde{X}^{n, N}_{t_{m}},\widehat{\mu}^{K}_{t_{m}}) \\
\hspace{1.4cm}+\sqrt{ h}\sigma(t_m, \widetilde{X}^{n, N}_{t_{m}}, \widehat{\mu}^{K}_{t_{m}}) Z_{m+1}^{n}\\
\widehat{\mu}^{K}_{t_{m}}=\big(\frac{1}{N}\sum_{n=1}^{N}\delta_{\widetilde{X}^{n, N}_{t_{m}}}\big)\circ \Proj_{x^{(m)}}^{-1}\\
\hspace{0.25cm}=\sum_{k=1}^{K} \delta_{x_{k}^{(m)}}\big[\sum_{n=1}^{N}\mathbbm{1}_{V_{k}(x^{(m)})}(\widetilde{X}_{t_{m}}^{n, N})\big] \\ 
\widetilde{X}_{0}^{n, N}=\bar{X}_{0}^{n, N}
\end{cases}\]
        
        \vspace{-0.2cm}
        \end{minipage}};
\node[fancytitle, right=10pt] at (Nparticlequantization.north west) {\: Hybrid particle-quantization scheme \eqref{Feq}};

\node (recursive) [mybox, below =1.6cm  of quantization]
{\begin{minipage}{0.4\textwidth}
\footnotesize
In the Vlasov setting, 
one can compute a Markovian transition: 

\vspace{-0.3cm}
\begin{align}
(&\widehat{X}_{t_{m+1}}, \widehat{\mu}_{m+1})\nonumber\\
&\hspace{0.5cm}=F(\widehat{X}_{t_{m}}, \widehat{\mu}_{m}, Z_{m+1}, x^{(m+1)}),\nonumber
\end{align}

where $F$ has an explicit formula \eqref{Geq},\eqref{Geq2}. 
\end{minipage}
};
\node[fancytitle, right=10pt] at (recursive.north west) {\: Recursive quantization scheme\: \eqref{Geq}};

\draw [arrow] (Nparticle) --node [above] {$\substack{\text{Propagation}\\ \text{of chaos}}$} node [below] {\small\color{black}\cite{gartner1988onthe}\cite{lacker2018mean}}(MVequation);
\draw [arrow] (theoricalscheme) --node [right] {\footnotesize $\substack{\vertii{\bar{X}_{t_{m}}-X_{t_{m}}}_{p}\sim\mathcal{O}(h^{\frac{1}{2}\wedge \gamma})\\\footnotesize\text{see \cite[Prop 2.1]{liu2020functional} }}$}  (MVequation);
\draw [arrow] (schemeNparticle) --node [above] {$\substack{\mathcal{W}_{p}(\bar{\mu}_{t_{m}}^{N},\;\bar{\mu}_{t_{m}})\\\sim\mathcal{O}(N^{-1/d})}$} node [below] {\footnotesize Theorem \ref{thm1}} (theoricalscheme);
\draw [arrow] (quantization) -- node [right] {$\substack{\text{For a fixed $M$, \hspace{1.5cm}}\\\vertii{\widehat{X}_{t_{m}}-\bar{X}_{t_{m}}}_{2}\sim\mathcal{O}(K^{-1/d})\\\text{\footnotesize Theorem \ref{thm:quadbasedscheme}} \hspace{1.7cm}}$}(theoricalscheme);
\draw [arrow] (Nparticlequantization) -- node [right] {\footnotesize$\substack{\text{For a fixed $M$\hspace{1.7cm}}\\\mathcal{W}_{p}(\widehat{\mu}_{t_{m}}^{K},\;\bar{\mu}_{t_{m}}^{N})\sim\mathcal{O}(K^{-1/d})\\\text{\footnotesize  Proposition \ref{quanNparti}} \hspace{1.3cm}}$} (schemeNparticle);
\draw [arrow] (recursive) --node [right]{\footnotesize Section \ref{recurq}} (quantization);
\end{tikzpicture}
\caption{The organisation of different numerical schemes in this paper}
\label{stru}

\vspace{0.1cm}
\end{figure}

\newpage

\begin{algorithm}[H]
\caption{Recursive quantization scheme (without Lloyd's algorithm integrated)}\label{recurquanti1}

\smallskip

\textbf{Input:} The quantization level $K\in\mathbb{N}^{*}$. 

\textbf{Input:} The quantizer sequence $x^{(m)}= (x_{1}^{(m)}, ..., x_{K}^{(m)}),\;0\leq m\leq M$.

\BlankLine
\tcc{Compute the weight vector $p^{(0)}=(p^{(0)}_1, ..., p^{(0)}_K)$ for time $t_0=0$.}

\For{$k$ in $\{1, ..., K\}$}{
\[\hspace{-2cm} p^{(0)}_k=\mu_0\big(V_k(x^{(0)})\big).\]
}

\BlankLine

\tcc{Compute the weight vector $p^{(m+1)}=(p^{(m+1)}_1, ..., p^{(m+1)}_K)$ for time $t_{m+1}$ from $x^{(m)}$ and $p^{(m)}$.}

\BlankLine

\For{$m$ in $\{0, ..., M-1\}$}{

\BlankLine

\For{$j$ in $\{1, ..., K\}$}{
\BlankLine

\For{$i$ in $\{1, ..., K\}$}{
\BlankLine 

Set \vspace{-0.8cm}
\begin{align}
\hspace{-2cm}\widetilde{\mu}_{t_{m+1}}^{\,i}=\mathcal{N}\Big(&x_i^{(m)}\!\! +\! h\sum_{k=1}^{K}p_{k}^{(m)}\beta(t_m, x_i^{(m)}, x_{k}^{(m)}), \nonumber\\
&\qquad h\Big[\sum_{k=1}^{K}p_{k}^{(m)}a(t_m, x_i^{(m)}, x_{k}^{(m)})\Big]^{\top}\Big[\sum_{k=1}^{K}p_{k}^{(m)}a(t_m, x_i^{(m)}, x_{k}^{(m)})\Big]\Big).\nonumber
\end{align}

Compute the transition probability $\pi_{ij}$ in  \eqref{Geq} by  
$\pi_{ij} = \widetilde{\mu}_{t_{m+1}}^{\,i}\left( V_j(x^{(m+1)}) \right)$.}

\BlankLine
Compute the weight $p_{j}^{(m+1)}$ in  \eqref{Geq2} by $p_{j}^{(m+1)}=\sum_{i=1}^{K}\pi_{ij}\cdot p_i^{(m)}$.
}}

\textbf{Output:} The discrete probability distributions $\widehat{\mu}_{t_m}=\sum_{k=1}^{K}p_{k}^{(m)}\delta_{x_{k}^{(m)}},\,0\leq m\leq M.$

\end{algorithm}

\pagebreak

\begin{algorithm}[H]
\caption{Recursive quantization scheme (with Lloyd's algorithm integrated)}\label{recurquanti2}

\BlankLine

\textbf{Input:} The quantization level $K\in\mathbb{N}^{*}$. 

\textbf{Input:} The Lloyd's iteration number $L\in\mathbb{N}^{*}$. 

\textbf{Input:} The initial quantizer sequence $x^{(m,[0])}= (x_{1}^{(m,[0])}, ..., x_{K}^{(m,[0])}),\;0\leq m\leq M$. 

\BlankLine
\BlankLine

\tcc{Compute an optimal quantizer $x^{(0)}\!=\!(x^{(0)}_1, \!..., x^{(0)}_K)$ for $\!\mu_0$ \!by using Lloyd's algorithm. }

\BlankLine

\For{$l$ in $\{0, ..., L-1\}$}{
\BlankLine

\For{$k$ in $\{1, ..., K\}$}{
\vspace{-0.2cm}
\[
x_{k}^{(0,[l+1])}=\begin{cases}
x_{k}^{(0,[l])},&\text{ if } \mu_0 \big(V_{k}(x^{(0,[l])})\big)=0,\\ \\\displaystyle
\frac{\int_{V_{k}(x^{(0,[l])})}\xi\mu_0(d\xi)}{\mu_0 \big(V_{k}(x^{(0,[l])})\big)}, &\text{otherwise.}\end{cases} 
\]
}
}

\BlankLine
\tcc{Compute the weight vector $p^{(0)}=(p^{(0)}_1, ..., p^{(0)}_K)$ for time $t_0=0$.}

\BlankLine
Set  $(x^{(0)}_1, ..., x^{(0)}_K) = (x_{1}^{(0,[L])}, ..., x_{K}^{(0,[L])})$.

\BlankLine
\For{$k$ in $\{1, ..., K\}$}{
\[\hspace{-2cm}p^{(0)}_k=\mu_0\big(V_k(x^{(0)})\big).\]
}



\BlankLine
\tcc{Compute the weight vector $p^{(m+1)}=(p^{(m+1)}_1, ..., p^{(m+1)}_K)$ for time $t_{m+1}$.}
\BlankLine

\For{$m$ in $\{0, ..., M-1\}$}{

\BlankLine

\tcc{Compute an optimal quantizer $x^{(m+1)}\!=\!(x^{(m+1)}_1, \!..., x^{(m+1)}_K)$ \!by using Lloyd's algorithm. }

\BlankLine

\For{$i$ in $\{1, ..., K\}$}{
\BlankLine 

Set \vspace{-0.8cm}
\begin{align}
\hspace{-2cm}\widetilde{\mu}_{t_{m+1}}^{\,i}=\mathcal{N}\Big(&x_i^{(m)}\!\! +\! h\sum_{k=1}^{K}p_{k}^{(m)}\beta(t_m, x_i^{(m)}, x_{k}^{(m)}), \nonumber\\
&\qquad h\Big[\sum_{k=1}^{K}p_{k}^{(m)}a(t_m, x_i^{(m)}, x_{k}^{(m)})\Big]^{\top}\Big[\sum_{k=1}^{K}p_{k}^{(m)}a(t_m, x_i^{(m)}, x_{k}^{(m)})\Big]\Big).\nonumber
\end{align}\vspace{-0.2cm}
}

\For{$l$ in $\{0, ..., L-1\}$}{
\BlankLine
\For{$j$ in $\{1, ..., K\}$}{
\BlankLine

\[
\hspace{-3cm}x_{j}^{(m+1,[l+1])}=\begin{cases}
x_{j}^{(m+1,[l])},\qquad\text{ if } \sum_{i=1}^{K} \widetilde{\mu}_{t_{m+1}}^{\,i}\big( V_{j}(x^{(m+1,[l])})\big)\cdot p_{i}^{(m)}=0,\\ \\
\displaystyle
\frac{\sum_{i=1}^{K} \big(\int_{V_{j}(x^{(m+1,[l])})} \xi \widetilde{\mu}_{t_{m+1}}^{\,i}(d\xi)\big)\cdot p_{i}^{(m)}}{\sum_{i=1}^{K} \widetilde{\mu}_{t_{m+1}}^{\,i}\big( V_{j}(x^{(m+1,[l])})\big)\cdot p_{i}^{(m)}}, \qquad\text{otherwise.}\end{cases} 
\]
}
}
\BlankLine
Set  $(x^{(m+1)}_1, ..., x^{(m)}_K) = (x_{1}^{(m+1,[L])}, ..., x_{K}^{(m+1,[L])})$.

\BlankLine
\tcc{Compute the weight vector $p^{(m+1)}=(p^{(m+1)}_1, ..., p^{(m+1)}_K)$ from $x^{(m)}$, $x^{(m+1)}$ and $p^{(m)}$.}

\For{$j$ in $\{1, ..., K\}$}{
\BlankLine
\For{$i$ in $\{1, ..., K\}$}{
\BlankLine 
Compute the transition probability $\pi_{ij}$ in  \eqref{Geq} by
$\pi_{ij} = \widetilde{\mu}_{t_{m+1}}^{\,i}\left( V_j(x^{(m+1)}) \right)$.

}

 \BlankLine   
Compute the weight $p_{j}^{(m+1)}$ in  \eqref{Geq2} by $p_{j}^{(m+1)}=\sum_{i=1}^{K}\pi_{ij}\cdot p_i^{(m)}$.
}
}

\BlankLine
\textbf{Output:} The discrete probability distributions $\widehat{\mu}_{t_m}=\sum_{k=1}^{K}p_{k}^{(m)}\delta_{x_{k}^{(m)}},\,0\leq m\leq M.$
\end{algorithm}

\pagebreak

\begin{algorithm}[H]
\caption{Hybrid particle-quantization method (without Lloyd's algorithm integrated)}\label{hybridcode1}
\BlankLine
\textbf{Input:} The quantization level $K\in\mathbb{N}^{*}$. 

\textbf{Input:} The quantizer sequence $x^{(m)}= (x_{1}^{(m)}, ..., x_{K}^{(m)}),\;0\leq m\leq M$.

\BlankLine
Generate $\widetilde{X}_{t_0}^{n, N}\widesim{\text{i.i.d}} X_0$ and $Z_m^n \widesim{\text{i.i.d}} \mathcal{N}(0,\mathbf{I}_q)$, $1\leq n \leq N, \;1\leq m \leq M$. 

\BlankLine
\For{$m$ in $\{0, ..., M-1\}$}{
\BlankLine

Compute $\widehat{\mu}^{K}_{t_{m}}=\sum_{k=1}^{K} \big[\delta_{x_{k}^{(m)}}\cdot\sum_{n=1}^{N}\mathbbm{1}_{V_{k}(x^{(m)})}(\widetilde{X}_{t_{m}}^{n, N})\big]$.

\For{$n$ in $\{1, ..., N\}$}{
\BlankLine

Compute $\widetilde{X}^{n, N}_{t_{m+1}}=\widetilde{X}^{n, N}_{t_{m}}+ h \cdot b(t_m, \widetilde{X}^{n, N}_{t_{m}},\widehat{\mu}^{K}_{t_{m}})+\sqrt{h\,} \sigma(t_m, \widetilde{X}^{n, N}_{t_{m}}, \widehat{\mu}^{K}_{t_{m}})Z_{m+1}^{n}$. 
}
}
\BlankLine
\textbf{Output:} The discrete probability distributions $\,\widehat{\mu}^{K}_{t_{m}}\!\!=\!\sum_{k=1}^{K} \!\big[\delta_{x_{k}^{(m)}}\!\cdot\!\sum_{n=1}^{N}\!\mathbbm{1}_{V_{k}(x^{(m)})}(\widetilde{X}_{t_{m}}^{n, N})\big],\,0\!\leq \!m\!\leq \!M.$
\BlankLine
\end{algorithm}

\newpage

\begin{algorithm}[H]
\caption{Hybrid particle-quantization method (with Lloyd's algorithm integrated)}\label{hybridcode2}

\BlankLine
\textbf{Input:} The quantization level $K\in\mathbb{N}^{*}$. 

\textbf{Input:} The Lloyd's iteration number $L\in\mathbb{N}^{*}$. 

\textbf{Input:} The initial quantizer sequence $x^{(m,[0])}= (x_{1}^{(m,[0])}, ..., x_{K}^{(m,[0])}),\;0\leq m\leq M$.

\BlankLine
Generate $\widetilde{X}_{t_0}^{n, N}\widesim{\text{i.i.d}} X_0$ and $Z_m^n \widesim{\text{i.i.d}} \mathcal{N}(0,\mathbf{I}_q)$, $1\leq n \leq N, \;1\leq m \leq M$. 

\BlankLine

\tcc{Compute an optimal quantizer $x^{(0)}\!=\!(x^{(0)}_1, \!..., x^{(0)}_K)$ for $\!\mu_0$ \!by using Lloyd's algorithm. }

\BlankLine

\For{$l$ in $\{0, ..., L-1\}$}{
\BlankLine

\For{$k$ in $\{1, ..., K\}$}{
\vspace{-0.2cm}
\[\hspace{-2cm}
x_{k}^{(0,[l+1])}=\begin{cases}
x_{k}^{(0,[l])},&\text{ if } \mu_0 \big(V_{k}(x^{(0,[l])})\big)=0,\\ \\\displaystyle
\frac{\int_{V_{k}(x^{(0,[l])})}\xi\mu_0(d\xi)}{\mu_0 \big(V_{k}(x^{(0,[l])})\big)}, &\text{otherwise.}\end{cases} 
\]
}
}

\BlankLine

Set  $(x^{(0)}_1, ..., x^{(0)}_K) = (x_{1}^{(0,[L])}, ..., x_{K}^{(0,[L])})$.

\BlankLine

\BlankLine

\For{$m$ in $\{0, ..., M-1\}$}{
\BlankLine

Compute $\widehat{\mu}^{K}_{t_{m}}=\sum_{k=1}^{K} \big[\delta_{x_{k}^{(m)}}\cdot\sum_{n=1}^{N}\mathbbm{1}_{V_{k}(x^{(m)})}(\widetilde{X}_{t_{m}}^{n, N})\big]$.

\For{$n$ in $\{1, ..., N\}$}{
\BlankLine

Compute $\widetilde{X}^{n, N}_{t_{m+1}}=\widetilde{X}^{n, N}_{t_{m}}+ h \cdot b(t_m, \widetilde{X}^{n, N}_{t_{m}},\widehat{\mu}^{K}_{t_{m}})+\sqrt{h\,} \sigma(t_m, \widetilde{X}^{n, N}_{t_{m}}, \widehat{\mu}^{K}_{t_{m}})Z_{m+1}^{n}$. 
}

\BlankLine

\tcc{Compute an optimal quantizer $x^{(m+1)}\!=\!(x^{(m+1)}_1, \!..., x^{(m+1)}_K)$ by using Lloyd's algorithm. }

\For{$l$ in $\{1, ..., L\}$}{
\BlankLine

\For{$k$ in $\{1, ..., K\}$}{
\BlankLine
\[
\hspace{-2cm}x_{k}^{(m+1,[l+1])}=\begin{cases}
x_{k}^{(m+1,[l])},&\text{ if } \sum_{n=1}^{N}\mathbbm{1}_{V_{k}(x^{(m+1,[l])})}(\widetilde{X}^{n, N}_{t_{m+1}})=0,\\ \\\displaystyle
\frac{\sum_{n=1}^{N}\widetilde{X}^{n, N}_{t_{m+1}}\mathbbm{1}_{V_{k}(x^{(m+1,[l])})}(\widetilde{X}^{n, N}_{t_{m+1}})}{\sum_{n=1}^{N}\mathbbm{1}_{V_{k}(x^{(m+1,[l])})}(\widetilde{X}^{n, N}_{t_{m+1}})}, &\text{otherwise.}\end{cases} 
\]
}
}
\BlankLine
Set  $(x^{(m+1)}_1, ..., x^{(m)}_K) = (x_{1}^{(m+1,[L])}, ..., x_{K}^{(m+1,[L])})$.

\BlankLine

}
\BlankLine

\textbf{Output:} The discrete probability distributions $\,\widehat{\mu}^{K}_{t_{m}}\!\!=\!\sum_{k=1}^{K} \!\big[\delta_{x_{k}^{(m)}}\!\cdot\!\sum_{n=1}^{N}\!\mathbbm{1}_{V_{k}(x^{(m)})}(\widetilde{X}_{t_{m}}^{n, N})\big],\,0\!\leq \!m\!\leq \!M.$
\BlankLine
\end{algorithm}

\bibliographystyle{abbrv}
\bibliography{Numerical_McKean_version_revision}

\end{document}